\documentclass[11pt,oneside]{amsart}
\usepackage{amsfonts, amsmath, amssymb, amsthm, amscd, enumerate, graphicx}

\makeatletter
\def\section{\@startsection{section}{1}%
\z@{.7\linespacing\@plus\linespacing}{.5\linespacing}
{\normalfont\scshape\LARGE}}
\makeatother

\newcommand{\fa}{\mathfrak a}

\newcommand{\fg}{\mathfrak g}
\newcommand{\fh}{\mathfrak h}
\newcommand{\fk}{\mathfrak k}

\newcommand{\fn}{\mathfrak n}

\newcommand{\ft}{\mathfrak t}

\newcommand{\fC}{\mathfrak C}

\newcommand{\fF}{\mathfrak F}
\newcommand{\fH}{\mathfrak H}
\newcommand{\fJ}{\mathfrak J}
\newcommand{\fL}{\mathfrak L}

\newcommand{\fU}{\mathfrak U}

\newcommand{\bbA}{\mathbb A}
\newcommand{\bbC}{\mathbb C}
\newcommand{\bbN}{\mathbb N}
\newcommand{\bbR}{\mathbb R}
\newcommand{\bbZ}{\mathbb Z}

\newcommand{\rA}{\mathrm A}
\newcommand{\rB}{\mathrm B}
\newcommand{\rE}{\mathrm E}
\newcommand{\rF}{\mathrm F}
\newcommand{\rG}{\mathrm G}
\newcommand{\rH}{\mathrm H}
\newcommand{\rK}{\mathrm K}
\newcommand{\rM}{\mathrm M}
\newcommand{\rN}{\mathrm N}
\newcommand{\rP}{\mathrm P}
\newcommand{\rS}{\mathrm S}
\newcommand{\rT}{\mathrm T}
\newcommand{\rU}{\mathrm U}
\newcommand{\rV}{\mathrm V}
\newcommand{\rZ}{\mathrm Z}
\newcommand{\rSL}{\mathrm S \mathrm L}
\newcommand{\rSU}{\mathrm S \mathrm U}

\newcommand{\bm}{\mathbf m}
\newcommand{\bs}{\mathbf s}
\newcommand{\bt}{\mathbf t}
\newcommand{\bx}{\mathbf x}
\newcommand{\bpi}{\boldsymbol \pi }
\newcommand{\bsigma}{\boldsymbol \sigma }

\newcommand{\bH}{\mathbf H}
\newcommand{\bL}{\mathbf L}
\newcommand{\bW}{\mathbf W}

\newcommand{\cA}{\mathcal A}
\newcommand{\cF}{\mathcal F}
\newcommand{\cH}{\mathcal H}
\newcommand{\cG}{\mathcal G}
\newcommand{\cL}{\mathcal L}
\newcommand{\cR}{\mathcal R}
\newcommand{\cV}{\mathcal V}

\newcommand{\Cl}{\mathrm {Cl} }
\newcommand{\Rint}{\sideset{_{\mathrm R}}{}\int}

\theoremstyle{plain}
\newtheorem{thm}{Theorem}[section]
\newtheorem{lem}[thm]{Lemma}
\newtheorem{prop}[thm]{Proposition}
\newtheorem{cor}[thm]{Corollary}

\theoremstyle{remark}
\newtheorem{rem}[thm]{Remark}

\theoremstyle{definition}
\newtheorem{exam}[thm]{Example}
\newtheorem{nota}[thm]{Notation}
\newtheorem{dfn}[thm]{Definition}
\newtheorem{conj}[thm]{Conjecture}

\DeclareMathOperator{\ad}{a d}
\DeclareMathOperator{\im}{i m}
\DeclareMathOperator{\id}{i d}
\DeclareMathOperator{\re}{r e}
\DeclareMathOperator{\ind}{i n d}
\DeclareMathOperator{\Span}{Span}
\DeclareMathOperator{\Supp}{Supp}
\DeclareMathOperator{\End}{End}
\DeclareMathOperator{\Hom}{Hom}

\DeclareMathOperator{\Ker}{Ker}

\DeclareMathOperator{\str}{s t r}
\DeclareMathOperator{\Str}{S t r} 

\begin{document}
\thispagestyle{empty}
\pagenumbering{Alph}

\title{Pseudo-differential operators, 
heat calculus and index theory of groupoids satisfying the Lauter-Nistor condition}
\author{Bing Kwan SO}

\begin{abstract}
In this thesis, 
we study singular pseudo-differential operators defined by groupoids satisfying the Lauter-Nistor condition, 
by a method parallel to that of manifolds with boundary and edge differential operators.
The example of the Bruhat sphere is studied in detail.
In particular,
we construct an extension to the calculus of uniformly supported pseudo-differential operators
that is analogous to the calculus with bounds defined on manifolds with boundary.
We derive a Fredholmness criterion for operators on the Bruhat sphere,
and prove that their parametrices up to compact operators lie inside the extended calculus;
we construct the heat kernel of perturbed Laplacian operators;
and prove an Atiyah-Singer type renormalized index formula for perturbed Dirac operators on the Bruhat sphere
using the heat kernel method.
\end{abstract}

\markboth{}{}

\pagebreak
\thispagestyle{empty}

\begin{figure}
\begin{center}
\includegraphics[width=45mm]{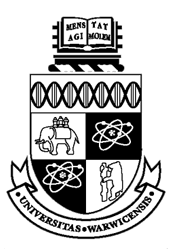} 
\end{center}
\end{figure}

\begin{center}
\begin{huge}
{\sc Pseudo-differential operators, heat calculus and index theory of groupoids satisfying the Lauter-Nistor condition }
\end{huge}
\end{center}

\vspace{1.5cm}

\begin{center}
\begin{Large}
{\sc Bing Kwan SO }
\end{Large}
\end{center}
\vspace{3.0cm}

\begin{center}
A thesis submitted in partial fulfillment of the requirements for \\
the degree of Doctor of Philosophy \\
at The University of Warwick. \\
January 2010.
\end{center}
\vspace{0.5cm}

\pagebreak
\pagenumbering{roman}
\tableofcontents

\pagebreak
\begin{LARGE}
{\centering {\sc Acknowledgment} \\}
\end{LARGE}
$ \; \\ $
I wish to express my sincere gratitude to my supervisor Prof. J. Rawnsley 
for his continual guidance throughout the period of postgraduate studies.
I wish to thank The Croucher Foundation, 
whose scholarship makes this project possible.
Last but not the least, my parents for their continual support through the telephone from the other side of the globe,
during the last three very difficult years of my lifetime.

\pagebreak
\begin{LARGE}
{\centering {\sc Declaration} \\}
\end{LARGE}
$ \; \\ $
Unless otherwise specified, all material in this thesis is the result of 
my own, independent, and original work to the best of my knowledge.
The research was done under the supervision of 
Prof. J. Rawnsley during the period 2006-2010 for the degree of Doctor of Philosophy at The University of Warwick. 
The work submitted has not been previously included in a thesis, 
dissertation or report submitted to any institution for a degree,
diploma or other qualification.

\vspace{7.0cm}

\begin{flushright}
==================\\
SO, Bing Kwan
\end{flushright}
\vspace{2.0cm}

$ \; $

\pagebreak
\begin{LARGE}
{\centering {\sc Abstract} \\}
\end{LARGE}
$ \; \\ $
In this thesis, 
we study singular pseudo-differential operators defined by groupoids satisfying the Lauter-Nistor condition, 
by a method parallel to that of manifolds with boundary and edge differential operators.
The example of the Bruhat sphere is studied in detail.
In particular,
we construct an extension to the calculus of uniformly supported pseudo-differential operators
that is analogous to the calculus with bounds defined on manifolds with boundary.
We derive a Fredholmness criterion for operators on the Bruhat sphere,
and prove that their parametrices up to compact operators lie inside the extended calculus;
we construct the heat kernel of perturbed Laplacian operators;
and prove an Atiyah-Singer type renormalized index formula for perturbed Dirac operators on the Bruhat sphere
using the heat kernel method.

\pagebreak
\pagenumbering{arabic}
\thispagestyle{firstpage}
\section{Introduction: From singular to groupoid pseudo-differential calculus} 
Traditionally, 
the way to study singular pseudo-differential operators involves studying underlying manifolds 
with embedded boundaries or corners.
These boundaries are always defined by the zero set of some functions 
(known as the boundary defining functions $\rho $),
with non-vanishing differentials near the boundary.
As a consequence, 
a neighborhood of the boundary $\partial \mathrm M$ is of the form 
$[0 , 1) \times \partial \mathrm M$ (with the closed interval $[0, 1) $ parameterized by $\rho $).
Then, one would typically consider the calculus of pseudo-differential operators whose kernels have 
poly-homogeneous expansions in $\rho $ near the boundary (see \cite{Mazzeo;EdgeRev} and the reference there).

The use of groupoids for studying the geometry of manifolds with boundaries (or corners)
was a much later development.
Early use of groupoids in pseudo-differential analysis include the convolution algebra 
defined on the holonomy groupoid of a regular foliation by Connes et. al. 
(see \cite{Connes;Book} for a review).
The notion of pseudo-differential operators on a groupoid was developed by Nistor, Weinstein and Xu 
\cite{NWX;GroupoidPdO}.
Subsequently, Monthubert \cite{M'bert;CornerGroupoids} showed that the $b$-calculus is, indeed,
the vector representation of pseudo-differential operators on some groupoids.
The theory is further formalized by Ammann et. al. into so called Lie manifolds,
or manifolds with Lie structure at infinity
\cite{Nistor;Polyhedral,Nistor;LieMfld,Nistor;Polyhedral2}.

Despite the development of the groupoid theory, 
most, if not all analysis was done on examples very similar to the manifold with boundary case.

In this thesis, 
we study the analysis of pseudo-differential operators 
in a systematic way parallel to that of singular pseudo-differential operators on manifolds with boundary
(i.e. \cite{Melrose;Book} etc.).
Our work is motivated by the study of the Poisson (co)-differential operator and its homology.
These invariants are difficult to compute. 
The only attempt to develop any form of machinery seems to be \cite{So;MPhil},
and the Laplacian defined there is not elliptic in the usual sense.
Also it is clear that the singularities are not explicitly defined by any boundary defining function.
Moreover, even if the homology is computed directly,
the result is often infinite dimensional,
and therefore not very meaningful.
For this reason,
we consider renormalized index theory,
which gives finite results.

The approach in this thesis is based on the principle that all singular pseudo-differential operators are defined by
vector representations of operators on the \\ groupoids.
Therefore instead of studying the calculus of singular pseudo-differential operators, 
one only needs to study non-singular pseudo-differential operators on the groupoid. 
The main part of this thesis, Sections 2-5, 
is an account of the technical details on how this principle is implemented,
particularly to the example of the Bruhat Poisson structure on the sphere $\bbC \rP (1)$.

Here, we shall give an overview of our approach.
In Section 2,
We collect together background material from several standard sources, which is needed for the thesis.
We begin with reviewing the well known formalism of uniformly supported
groupoid pseudo-differential operators by Nistor et. al. \cite{NWX;GroupoidPdO}.
The uniformly supported calculus is comparable to the small calculus in manifolds with boundary.
We shall also define the notion of a Dirac operator on a groupoid.
Then we shall introduce some examples, 
most notably the symplectic groupoids of the Bruhat Poisson structure on flag manifolds,
where the Bruhat sphere is the simplest case.

Section 3 focuses on two questions, which are exact counterparts of \cite[Chapter 5]{Melrose;Book}:
\begin{enumerate}
\item
What (elliptic) pseudo-differential operator on a groupoid has Fredholm vector representation? 
\item
What does the parametrix of a Fredholm operator on a the groupoid defining the Bruhat sphere look like? 
\end{enumerate}
Lauter and Nistor's \cite{Nistor;GeomOp} theory gives a quick answer for (1), 
namely, if an operator is invertible on all the singular leaves,
then its vector representation is Fredholm.
In the simple case of the Bruhat sphere,
question (1) therefore immediately reduces to checking the invertibility of the operator over the only singular leaf.
Due to some invariance properties, 
the natural way to proceed is by Fourier-Laplace transform.
We remark that our approach is again parallel to the indicial family formalism for manifolds with boundary
(recall that given a partial differential operator $\varPsi$,
the indicial family is defined to be the family of differential operators 
$ (e ^{-i \xi \rho } \varPsi e^ {i \xi \rho } ) |_{\partial \rM}, \quad \xi \in \bbC $,
see \cite{Loya;Pert,Melrose;Book} for detailed definitions).
Indeed, it can be said that Fourier-Laplace transform {\it is the right} definition for indicial family.
We then turn to describe the parametrix of Fredholm operators on the Bruhat sphere,
using the fact that the inverse of a properly supported, invariant pseudo-differential operator 
is an invariant pseudo-differential operator with exponential decaying kernel.
We then generalize the notion to groupoids with sub-exponential growth, 
and prove that the resulting calculus has a composition rule similar to that of calculus with bounds.

In Section 4,
we turn to the heat calculus of Laplacian operators.
The treatment here is very different from that of \cite{Albin;EdgeInd,Melrose;Book}, 
and much simpler. 
That is not surprising because the source fibers are just non-singular manifolds with bounded geometry,
and the heat kernel is essentially the leaf-wise heat kernel.
Therefore the classical construction suffices.
Perhaps the only unexpected result is the proof of transverse smoothness,
which requires additional growth conditions on the differential of the product map.
We shall leave the technical details to Section 4.2.

Given a perturbed Dirac operator that is Fredholm (one satisfying the conditions in Section 3),
it is natural to seek an Atiyah-Singer type formula for its Fredholm index.
That is the theme of Section 5. 
Again the technique we use is parallel to that of \cite{Albin;EdgeInd,Loya;Pert,Melrose;Book},
and is fairly standard.
We use the stereographic coordinates on the open leaf of the Bruhat sphere to define a renormalized trace.
The we derive the local index formula.
We do have to fall back to the machinery of \cite{Albin;EdgeInd} to describe the long time behavior of the heat kernel.
However, it can be said that the `cheating' occurs already when we use the stereographic coordinates,
which effectively serves as a boundary defining function.
Nevertheless, our result is stronger than that of \cite{Albin;EdgeInd} 
in the sense that we are able to derive an explicit trace defect formula.

\pagebreak \thispagestyle{firstpage}
\section{Lie groupoids and pseudo-differential operators}

\subsection{ The differential geometry of Lie groupoids}
We begin our technical discussion with the basic definition of a Lie groupoid.
Our definition follows the convention of \cite{Mackenzie;Book2}, 
but with the source and target maps denoted by $\bs$ and $\bt$ instead of $\alpha $ and $\beta $.
\begin{dfn}
A {\it Lie groupoid} $\cG \rightrightarrows \rM $ consists of:
\begin{enumerate}
\item
Manifolds $\cG$ and $\rM$;
\item
A unit inclusion ${\mathbf u} : \rM \rightarrow \cG$;
\item
Submersions $\bs , \bt : \cG \rightarrow \rM$, called the source and target map respectively,
satisfying
$$ \bs \circ {\mathbf u} = \id _\rM = \bt \circ {\mathbf u}; $$
\item
A multiplication map 
$\mathbf m : \{ (a, b) \in \cG \times \cG : \bs (a) = \bt (b) \} \rightarrow \cG,
(a , b) \mapsto a b$
that is associative and satisfies
$$ \bs (a b) = \bs (b) , \quad \bt (a b) = \bt (a), 
\quad a ( \mathbf u \circ \bs (a)) = a = ( \mathbf u  \circ \bt (a)) a ; $$
\item
An inverse diffeomorphism $\mathbf i : \cG \rightarrow \cG, a \mapsto a^{-1}$,
such that $\bs (a^{-1}) = \bt (a), \\ \bt (a^{-1}) = \bs (a)$ and 
$$ a a^{-1} = \mathbf u (\bt (a)), a^{-1} a = \mathbf u (\bs (a)).$$
\end{enumerate} 
\end{dfn}

\begin{rem}
In this thesis,
we assume that the groupoid $\cG$ is Hausdorff.
This extra assumption is clearly satisfied in all of the examples we shall shortly see. 
Note that many important groupoids, like holonomy groupoids of foliations, are not Hausdorff.
\end{rem}

\begin{nota}
For simplicity we shall denote a Lie groupoid $\cG \rightrightarrows \rM$ by $\cG$ and call it a groupoid;
Also, with an abuse in notation we consider $\rM$ as a subset of $\cG$ via the unit inclusion
$\mathbf u$.
For each $x \in \rM $, we write 
$$ \cG _x := \bs ^{-1} (x) .$$ 
\end{nota}

\begin{dfn}
We say that a groupoid $\cG$ is $\bs${\it -connected} if $\cG _x $ is connected for all $x \in \rM$.
\end{dfn}

\begin{dfn}
Let $\cG$ be a Lie groupoid and $a \in \cG$.
The {\it right translation} is the diffeomorphism:
$$ R_a : \bs ^{-1} (a) \rightarrow \bt ^{-1} (a), b \mapsto b a, b \in \cG .$$
\end{dfn}

\begin{dfn}
A {\it right-invariant function} on $\cG$ is a smooth function $f$ such that
$$ f (b a) = f (b), \quad \forall a \in \cG, b \in \bs^{-1} (a) ;$$
A {\it right-invariant vector field} on $\cG$ is a vector field
$X$ such that $d \bs X = 0 $ (i.e., $X$ is a vector field along the $\bs$-fibers) and 
$$ d R_a (X (b)) = X (b a), \quad \forall a \in \cG, b \in \bs^{-1} (a) .$$
\end{dfn}
 
From the definition, 
one immediately observes that any right invariant function $f \in C^\infty (\cG)$ can be written in the form
\begin{equation}
f = \bt ^{-1} \tilde f, \; \text {where} \; \tilde f := \mathbf u ^* f \in C^\infty (\rM).
\end{equation}

\subsubsection{\bf Lie algebroids and singular foliations}

\begin{dfn}
A {\it Lie algebroid} $\cA$ is a vector bundle over $\rM$, together with 
a Lie algebra structure $[ \cdot , \cdot ]$ on the space of smooth sections $\Gamma ^\infty (\cA)$,
and a bundle map $\nu : \cA \rightarrow T \rM$ satisfying 
$$ \nu ([X, Y]) = [\nu (X) , \nu (Y)],
\; \text {and} \; [X , f Y] = f [X, Y] + (\fL_{\nu (X)} f) Y, $$
for any $ X, Y \in \Gamma ^\infty (\cA), f \in C^\infty (\rM).$
\end{dfn}

\begin{exam}
Let $(\rM , \varPi)$ be a Poisson manifold \cite{Vas;Book}.
Denote the contraction with the Poisson bi-vector field $\varPi$ by $\tilde \varPi : T^* \rM \to T \rM $.
Define the bracket 
$$ [ \omega _ 1 , \omega _2 ] 
:= d (\omega _1 \wedge \omega _2 (\varPi))
+ \iota _{\tilde \varPi (\omega _1)} d \omega _2 - \iota _{\tilde \varPi (\omega _2)} d \omega _1, $$
for any 1-forms $\omega _1 , \omega _2$.
It is easy to check that $T ^* \rM $ is a Lie algebroid using $\tilde \varPi$ as the anchor map.
\end{exam} 

In many ways the Lie algebroid plays the role of tangent bundle in our study.
For example we have:
\begin{dfn}
\cite{Fern'd;HoloAndChar}
Let $\rE $ be a vector bundle over $\rM $.
An $\cA$-connection on $\rE$ is a differential operator 
$\nabla ^\rE : \Gamma ^\infty (\rE ) \to \Gamma ^\infty (\cA ' \otimes \rE) $
satisfying the relations
\begin{align*}
\nabla ^\rE _{f X} u =& \: f \nabla ^\rE _X u \\
\nabla ^\rE _X (f u) =& \: f \nabla ^\rE _X u + \fL _{\nu (X)} u,
\end{align*}
for any $X \in \Gamma ^\infty (\cA) , f \in C^\infty (\rM ), u \in \Gamma ^\infty (\rE)$.
\end{dfn}

\begin{exam}
As in the case of Riemannian manifolds, given a metric $g _\cA$, 
i.e., a positive definite symmetric bi-linear form on $\cA$,
one can define the {\it Levi-Civita $\cA$-connection} $\nabla ^{g _\cA} $ on $\cA$ by the formula
\begin{align*}
2 g_\cA (\nabla ^{g _ \cA} _X Y, Z) 
:=& g_\cA ([ X, Y ], Z) - g_\cA ([Y, Z ], X) + g_\cA ([Z, X ], Y) \\
&+ \fL_{\nu (X)} g_\cA (Y, Z) + \fL_{\nu (Y)} g_\cA (Z, X) - \fL_{\nu (Z)} g_\cA (X, Y),
\end{align*}
for any $X, Y, Z \in \Gamma ^\infty (\cA)$.
\end{exam}

It is well known that every Lie groupoid $\cG$ determines a Lie algebroid:
Define the vector bundle 
$$\cA := \{ X \in T_x \cG : x \in \rM \subset \cG, d \bs ( X ) = 0 \}.$$
It is clear that restriction gives a 1-1 correspondence between 
$\Gamma ^\infty (\cA)$ and the space of right invariant vector fields on $\cG$.
Define $[ \cdot , \cdot ]$ to be the Lie bracket between invariant vector fields, and define
\begin{equation}
\nu := d \bt |_{\cA} : \cA \rightarrow T \rM.
\end{equation}
It is straightforward to check that $\cA$ is a Lie algebroid over $\rM$.

\begin{dfn}
A Lie algebroid defined by some Lie groupoid as above is said to be {\it integrable}.
\end{dfn}

Note that not all Lie algebroids are integrable.
See \cite{Fern'd;IntAlgebroid} for details. 

For any Lie algebroid $\cA \rightarrow \rM$,
the family of vector fields 
$$ \cF := \{ \nu (X) : X \in \Gamma ^\infty (\cA ) \}.$$
defines a (singular) integrable foliation on $\rM$ in the sense of Sussmann \cite{Sussmann;SingFol}.
We denote the leaf space of $\cF$ by $\rM / \cF$. 
For each $x \in \rM$, we denote the leaf of $\cF $ through $x$ by $\cF_x$.
Note that the leaves may be non-embedded sub-manifolds of $\rM$.
Given a singular foliation $\cF $ defined by an integrable Lie algebroid $\cA$, 
the following propositions, 
both are direct consequences of the results in \cite{Fern'd;HoloAndChar} (in particular Theorem 1.1),
describe the leaves of $\cF$.
\begin{prop}
\label{SubmersionProp}
Let $\cG \rightrightarrows \rM$ be a Lie groupoid.
For each $x \in \rM$, 
the map $\bt |_{\cG _x } : \cG _x \rightarrow \rM$ is a submersion onto its image.
\end{prop}


\begin{prop}
\label{btImage}
Let $\cG$ be an $\bs$-connected Lie groupoid.
Then for each $x \in \rM$, one has
$$\bt ( \cG _x ) = \cF_x; \quad \text{and } \quad \cF_x \cong \cG _x / \cG^x_x,$$ 
where $\cG ^x_x$ is the Lie group 
$\cG ^x_x := \{ a \in \cG : \bs (a) = \bt (a) = x \}$,
known as the isotropy subgroup. 
\end{prop}



\subsubsection{\bf Riemannian geometry of the $\bs$-fibers}
\label{RiemVert}
Let $\cG \rightrightarrows \rM$ be a Lie groupoid over a compact manifold $\rM$.
Let $\cA \rightarrow \rM$ be its Lie algebroid. 
Fix a metric $g _\cA$ (i.e. a symmetric, positive definite bi-linear form) on $\cA $. 
For each $x \in \rM$, $g _\cA$ defines a Riemannian metric on the $\bs$-fiber $\bs ^{-1} (x)$ by
$$ g_\bs (X , Y ) := g _\cA (\bt (a) ) (d R_a (X) , d R_a (Y) ). $$
Observe that $g_\bs $ is right invariant in the sense that the right translation
$$ R_a : \cG _{\bt (a)} \rightarrow \cG _{\bs (a)}, \quad \forall a \in \cG, X, Y \in T_a \cG _x $$
is an isometry for any $a \in \cG$.
As a direct consequence of the assumptions, one has
\begin{lem}
For each $x \in \rM$, 
the Riemannian manifold $(\cG _x , g_\bs )$ is a manifold with bounded geometry (see Appendix \ref{BdGeomNonSense}). 
\end{lem}
\begin{proof}
Consider the $\cA$-Levi-Civita connection:
\begin{align*}
2 g_\cA (\nabla ^\cA _X Y, Z) 
:=& g_\cA ([ X, Y ], Z) - g_\cA ([Y, Z ], X) + g_\cA ([Z, X ], Y) \\
&+ \fL_{\nu (X)} g_\cA (Y, Z) + \fL_{\nu (Y)} g_\cA (Z, X) - \fL_{\nu (Z)} g_\cA (X, Y),
\end{align*}
where $ X, Y, Z \in \Gamma ^\infty (\cA).$
Let $R^\cA$ be the curvature of $\nabla ^\cA$.

Consider $\nabla ^{\cG _x} _ {\tilde X} \tilde Y$,
where $\tilde X, \tilde Y$ are right invariant vector fields,
and $\nabla ^{\cG _x}$ is the Levi-Civita connection of $(\cG _x, g _\bs ) $ for each $x \in \rM$.
Write $X:= \tilde X|_\rM , Y:= \tilde Y|_\rM $, then $X, Y \in \Gamma ^\infty (\cA)$. 
Then for any right invariant vector field $\tilde Z, a \in \cG$,
one has
$$
2 g_\bs (a) (\nabla ^{\cG _{\bs (a)}} _{\tilde X} \tilde Y , \tilde Z)
= 2 g_\bs (a) ((d R _a) ((\nabla ^\cA _X , Y) (\bt (a))) , (d R _a) (Z (\bt (a)))).
$$
It follows that for any $\tilde X, \tilde Y$ right invariant, 
the vector field $ a \mapsto \nabla ^{\cG _{\bs (a)}} _{\tilde X} \tilde Y (a) $ is also right invariant.
Furthermore, $\nabla ^{\cG _ {\bs (a)}} _{\tilde X} \tilde Y (x) = \nabla ^\cA _X Y (x)$ for any $x \in \rM$.

By similar arguments, for any $\tilde X, \tilde Y, \tilde Z$ right invariant, 
$R (\tilde X, \tilde Y) \tilde Z $ is right invariant and one has
$$ R (\tilde X (a) , \tilde Y (a) ) \tilde Z (a)) = R ^\cA (X (\bt (a), Y (\bt (a))) Z (\bt (a))$$ 
for any $a \in \cG$.
Clearly, the right hand side $R ^\cA (X (\bt (a), Y (\bt (a))) Z (\bt (a)) $ is bounded since $\rM$ is compact.
Formulas for higher covariant derivatives also follow from these arguments.

Finally, to prove that the $\bs$-fibers have positive injectivity radius, observe that $\rM$ is compact.
It follows that there exists $r_0 > 0$ such that $\exp ^{\nabla ^\cA} $ is a diffeomorphism form the set 
$$ \cA_{r_0} := \{ X \in \cA :g _\cA (X, X) < r_0 ^2 \} $$
onto its image.
In proof of boundedness of curvature above, 
we saw that the Levi-Civita connection is obtained by right translating $\nabla ^ \cA$.
Therefore, for any $X \in T _a \cG _{\bs (a)}, a \in \cG$,
$$ \exp ^{\nabla ^{\cG _{\bs (a)}}} X = (d R _a ) \circ \exp ^{\nabla ^\cA } \circ (d R ^{-1} _a ) X .$$
It follows that the injectivity radius of $\cG _{\bs (a)} \geq r_0 $.
\end{proof}

The bounded geometry of the $\bs$-fibers means that the notion from manifolds of bounded geometry applies. 
In particular, we say that
\begin{dfn}
A function $ u \in C ^\infty (\cG )$ is said to have {\it bounded (fiberwise) derivatives} 
if for any $x \in \rM$, $u |_{\cG _x} $ has uniformly bounded covariant derivatives.
\end{dfn}

\subsubsection{\bf Examples of Lie groupoids}
We give some examples of Lie groupoids relevant to Poisson geometry.
\begin{exam}
Let $\rM$ be a manifold. 
The pair groupoid over $\rM$ is the manifold $\cG := \rM \times \rM$ together with the operations:
\begin{align*}
\text {source and target maps: } &  \bs (x, y) = y, \bt (x, y) = x, \quad \forall (x, y) \in \rM \times \rM , \\
\text {multiplication: } &  \mathbf m ((x, y), (y, z)) = (x, z), \quad \forall (x, y) , (y, z) \in \rM \times \rM , \\
\text {inverse: } & \mathbf i (x, y) = (y, x), \quad \forall (x, y) \in \rM \times \rM , \\
\text {unit: } & \mathbf u (x) = (x, x), \quad \forall x \in \rM .
\end{align*}   
The anchor map is the identity on $T \rM$. 
If, in addition, $\omega $ is a symplectic 2-form on $\rM$,
then $(\rM \times \rM , \wp_1 ^* \omega - \wp_2 ^* \omega )$ is the symplectic groupoid of $(\rM , \omega )$,
where $\wp_1 , \wp_2: \rM \times \rM \rightarrow \rM $ is the projection to the first and second factor respectively. 
\end{exam}

\begin{exam}
\label{Iwasawa example}
(See Lu and Weinstein \cite{LuWein;DressingTran})
Let $\fg$ be a complex semi-simple Lie algebra, 
let $\fk$ be a compact real form of $\fg$.
Let $\theta $ be the Cartan involution on $\fg$ fixing $\fk$. 
Let $\fa$ be a maximal Abelian subalgebra of $i \fk$. 
Then $\fh = \fa + i \fa$ is a Cartan subalgebra of $\fg$.
Let $\fg = \fh \oplus \sum_{\alpha \in \Delta} \fg_\alpha$ be the root space decomposition. 
Choose a set of positive roots $\Delta^+$ and let $\fn = \sum_{\alpha \in \Delta^+} \fg_\alpha$.
Then $\fg = \fk \oplus \fa \oplus \fn$ is an Iwasawa decomposition of $\fg$ (see \cite[Chapter IV.4]{Knapp;Book1}).

Let $ \langle \cdot , \cdot \rangle $ be the imaginary part of the Killing form.
Then $(\fg, \fk, \fa + \fn, \langle \cdot , \cdot \rangle)$ is a Manin triple (see \cite[Chapter 10]{Vas;Book}).
Its corresponding Poisson Lie group structure can be written as
$$ \varPi_\rK (g) := \frac{1}{2} \sum_{\alpha \in \Delta ^+} 
(d R_g) (X_\alpha \wedge Y_\alpha ) - (d L_g) (X_\alpha \wedge Y_\alpha ) , \quad g \in \rK , $$
where
$$ X_\alpha := E_\alpha + \theta E_\alpha, 
\; {\mbox {\rm and}} \; Y_\alpha := i E_\alpha - i \theta (E_\alpha ) \in \fk, \alpha \in \Delta ^+ ,$$
and $L_g , R_g$ denotes the left and right translation by $g$ respectively.

We turn to construction of the symplectic groupoid.
From the construction of Iwasawa decomposition of Lie algebra above, 
one gets the Iwasawa decomposition of Lie group:
$$ \rG = \rK \rA \rN.$$
Take $\cG := \rG $ as a manifold. Define:
\begin{align*}
\text {source and target maps: } &  \bs (g) := k , \bt (g) := k', 
\text {where } g = a n k = k' a' n' \\ 
& \text {is the (unique) Iwasawa decomposition;} \\  
\text {multiplication: } &  \mathbf m (g_1 , g_2) := g_1 (\bs (g_1))^{-1} g_2 ; \\
\text {inverse: } & \mathbf i (g) :=  k (n ')^{-1} (a ')^{-1} = n ^{-1} a ^{-1} k' ; \\
\text {unit: } & \mathbf u (k) := k \in \rG \supset \rK .
\end{align*}    
\end{exam} 

\begin{exam}
\label{BruhatExam}
\cite{Lu;PoissonCohNotes}
Let $\rG = \rK \rA \rN $ be the Iwasawa decomposition as above.
Let $\rT \subset \rK$ be the maximal torus with $\ft = i \fa$.
Then the Poisson bi-vector field $\varPi _\rK$ on $\rK$ is $\rT$-invariant.
Hence one has a well defined Poisson manifold
$$ (\rT \backslash \rK , d \wp_ \rT ( \varPi _\rK )),$$
where $\wp_ \rT : \rK \rightarrow \rT \backslash \rK $ is the natural projection onto coset space.
This Poisson structure is known as the Bruhat Poisson structure.

Define the left action of $\rT$ on $\rK \times \rN$ by
$$ g \cdot (k , n) := (g k, g n g ^{-1}), \quad \forall (k, n) \in \rK \times \rN, g \in \rT .$$
It is easy to see that the projection onto 
$$ \rT \backslash (\rK \times \rN) $$
is a submersion. 
Define the groupoid operations on $\cG := \rT \backslash (\rK \times \rN) \rightrightarrows \rT \backslash \rK $:
\begin{align*}
\text {source and target maps: } & \bs ( {}_\rT (k, n) ) = {}_ \rT k , \bt ( {}_\rT (k, n) ) := {}_\rT k' , \\
& \text {where } n k = k' a' n' \text { is the (unique) Iwasawa decomposition;} \\  
\text {multiplication: } &  \mathbf m ({}_\rT (k _1 , n _1) , {}_ \rT (k _2 , n _2) ) := {}_\rT (k _2 , n _1 n _2) , \\
& \text {provided one has Iwasawa decomposition } n _2 k _2 = k _1 a' n' ; \\
\text {inverse: } & \mathbf i ( {}_\rT (k, n)) := {}_\rT (k', n^{-1}) , \\
& \text {where } n k = k' a' n' \text { is the (unique) Iwasawa decomposition;} \\  
\text {unit: } & \mathbf u ({}_ \rT k ) := {}_ \rT (k , e) , e \in \rN .
\end{align*} 
\end{exam}

\subsection{ Uniformly supported pseudo-differential calculus on a Lie groupoid}
In this section, we review the standard theory of pseudo-differential calculus developed by
Nistor, Weinstein and Xu \cite{NWX;GroupoidPdO}. 
We refer to Appendix \ref{PDONonSense} for notations on pseudo-differential operators (on ordinary manifolds).

\begin{dfn}
A pseudo-differential operator $\varPsi $ on a groupoid $\cG$ of order $\leq m$ 
is a smooth family of pseudo-differential operators $\{ \varPsi _x \}_{x \in \rM}$,
where $\varPsi _x \in \Psi ^m ( \cG _{x} )$,
and satisfies the right invariance property
$$ \varPsi _{\bs (a)} (R_a^* f) = R_g^* \varPsi _{\bt (a)} (f), 
\quad \forall a \in \cG, f \in C^\infty_c (\cG _{\bs (a)}).$$
If, in addition, all $\varPsi _x $ are classical of order $m$, then we say that $\varPsi $ is classical of order $m$.
\end{dfn}

\begin{dfn}
For a pseudo-differential operator $\varPsi = \{ \varPsi _x \}$ on $\cG$.
The support of $\varPsi $ is defined to be 
$$ \Supp (\varPsi) = \overline {\bigcup_{x \in \rM} \Supp (\varPsi _x)}.$$
The operator $\varPsi $ is called properly supported if the set 
$$ (\rK \times \cG) \bigcap \Supp (\varPsi) $$
is compact for every compact subset $\rK \subseteq \cG$;
The operator $\varPsi $ is called uniformly supported if the set
$$ \{ a b^{-1} : (a, b) \in \Supp (\varPsi) \} $$
is a compact subset of $\cG$.
\end{dfn}

We denote the space of uniformly supported pseudo-differential operators 
(resp. classical pseudo-differential operators) on $\cG$,
of order $\leq m$, by $\Psi _\mu ^m (\cG)$ (resp. $\Psi _\mu ^{[m]} (\cG)$).

The way to define the total symbol for 
$\varPsi \in \Psi ^\infty (\cG)$ is similar to that of an ordinary pseudo-differential operator.
Fix an $\cA$-connection $\nabla$ 
(say, $\nabla _X Y := \nabla ^\cA _{\nu (X)} Y$ for some usual connection $\nabla ^\cA$). 
Then there is a neighborhood of the zero section $\Omega \subset \cA$
such that the exponential map 
$\exp _\nabla : \Omega \rightarrow \cG$ is a diffeomorphism onto its image.
Fix a smooth function $\chi (g)$ supported on the image of $\exp _\nabla$ 
and equal to 1 on a smaller neighborhood of $\rM$.
Define $\Theta (g, h) := \chi (g) \exp ^{-1}_\nabla (g)$.

\begin{dfn}
\label{GpoidTotalDfn}
\cite[Equation (16)]{NWX;GroupoidPdO}
Given $\varPsi \in \Psi ^\infty (\cG)$. Define $\sigma \in C^\infty (\cA^*)$ by
$$ \sigma (\zeta ) := \varPsi _x (e^{i \langle \zeta , \Theta (\cdot ) \rangle} \chi ( \cdot))(x),
\quad \forall x \in \rM \subset \cG, \zeta \in \cA^*_x.$$
The function $\sigma $ is called the total symbol of $\varPsi $ with respect to $(\nabla, \chi )$.

As in the case of manifolds, 
if there exist homogeneous symbols $\sigma _m , \sigma _{m-1} , \cdots$,
of orders $m, m-1, \cdots$ respectively, such that 
$$ \sigma - \sum_{l=0}^{N - 1} \sigma _{m - l} \in S^{m - N} (\rM )$$
for $N = 1, 2, \cdots$,
then we say that $\varPsi $ is a classical pseudo-differential operator on $\cG$. 
In this case, we define the principal symbol of $\varPsi $ as
$$ \sigma _{\mathrm {top}} (\varPsi ) := \sigma _{m}.$$
\end{dfn}
As in the case of manifolds,
we denote the space of uniformly supported classical pseudo-differential operator of order $m$
by $\Psi ^{[m]} _\mu (\cG)$.

\begin{dfn}
A classical pseudo-differential operator $\varPsi \in \Psi ^{[m]} _\mu (\cG)$ is said to be elliptic if 
$$ \sigma _{\mathrm {top}} (\varPsi ) (X) \neq 0 $$
for any $X \neq 0 \in \cA^*$.
\end{dfn}

A pseudo-differential operator $\varPsi \in \Psi ^\infty (\cG)$ acts on $C^\infty (\cG)$ by 
$$ \varPsi (u) (a):= \varPsi _{\bs (a)} (u |_{\bs^{-1} (\bs (a)})).$$
It is easy to see that the composition
$ \varPhi \circ \varPsi $ is well defined as long as either $\varPhi $ or $ \varPsi $ is uniformly supported.
Furthermore, the composition respects the grading:
\begin{lem}
Let $\varPsi \in \Psi ^{[m]} (\cG), \varPhi \in \Psi ^{[m']} (\cG)$ 
be such that either $\varPsi $ or $\varPhi $ is properly supported. 
Then $\varPhi \circ \varPsi \in \Psi ^{[m+m']} (\cG)$.
\end{lem}

\subsubsection{\bf Example: Dirac operators on a groupoid}
In this section, we briefly describe the Dirac type operators on a groupoid $\cG$ \cite[Section 6]{Nistor;GeomOp}.

We begin with recalling the notion of Clifford algebra, following \cite[Chapter 3]{BGV;Book}.
Let $\rV$ be a finite dimensional vector space over $\bbR$ or $\bbC$.
Let $B ( \cdot , \cdot ) $ be a symmetric bi-linear form on $\rV$.
Then the {\it Clifford algebra} of $(\rV, B )$,
denoted by $\Cl ( \rV , B ) $, is the algebra generated by $\rV$ with the relation
$$ v w + w v = - 2 B (v , w).$$
The algebra $\Cl (\rV , B)$ is $\bbZ _2$-graded by
$$ \Cl (\rV , B) = 
\mathrm {span} \{ 1 , v_{i_1} \cdots v _{i_{2 j }} : j = 1, 2 \cdots \} \oplus
\mathrm {span} \{ v_{i_1} \cdots v _{i_{2 j + 1}} : j = 0, 1, 2 \cdots \}, $$
where $\{ v _i \}$ is any basis of $\rV$.

A {\it Clifford module} of $\Cl (\rV)$ is a $\bbZ _2$-graded vector space $\rE = \rE ^+ \oplus \rE ^- $
such that the Clifford action $\gamma : \Cl (\rV) \to \End (\rE)$ satisfies
\begin{align*}
\gamma (\Cl ^+ (\rV )) \rE ^\pm & \subseteq \rE ^ \pm \\
\gamma (\Cl ^- (\rV )) \rE ^\pm & \subseteq \rE ^ \mp.
\end{align*}

\begin{exam}
\label{FormModule}
Let $B$ be an inner product on $\rV$.
Then $\wedge ^\bullet \rV = (\bigoplus _{i=0} \wedge ^{2 i} \rV) \oplus (\bigoplus _{i=0} \wedge ^{2 i + 1} \rV) $ 
is a natural $\Cl (\rV , B)$ module,
with action defined by:
$$ \gamma _\wedge (v) \omega := v \wedge \omega - \iota _{B (v , \cdot)} \omega , 
\quad \forall v \in \rV, \omega \in \wedge ^\bullet \rV,$$
where $\iota $ denotes the contraction. 
It is easy to verify that such an action of $\rV$ extends to $\Cl (\rV)$.
\end{exam} 
Example $\ref{FormModule}$ also provides a canonical bijective map between 
$\Cl (\rV)$ and $\wedge ^\bullet \rV $ as vector spaces, namely,
\begin{equation}
v \mapsto \gamma _\wedge (v) 1 , \quad v \in \Cl (\rV), 
\end{equation}
where $1 $ is the identity in the exterior algebra $\wedge ^\bullet \rV $. 
It is easy to see that the $\bbZ _2 $ splitting of 
$ \wedge ^\bullet \rV $ into even and odd orders gives a 
$\bbZ _2$ grading of the Clifford algebra $\Cl (\rV)$.

\begin{exam}
Let $\rV$ be an even dimensional vector space with inner product $B$.
Let $e _1 , e_2 , \cdots , e_{2n} $ be an orthonormal basis of $\rV$.
Define 
$$ \rP := \Span \{ e_{2 i - 1 } + i e _{2 i} : i = 1, \cdots , n \} \subset \rV \otimes \bbC .$$
Then $\rP \oplus \bar \rP = \rV \otimes \bbC $.
Define the action of $\rV \otimes \bbC $ on $\rS := \wedge ^\bullet \rP $ by
\begin{equation}
\gamma _\rS ( v ) \omega := 
\begin{cases}
\: v \wedge \omega , & \forall v \in \rP \\ 
\: \iota _{B (v , \cdot ) } \omega , & \forall v \in \bar \rP .
\end{cases}
\end{equation} 
The Clifford module $\rS$ is known as the {\it spin representation } of the Clifford algebra $\Cl (\rV)$. 
\end{exam}

Here, we list some basic facts about Clifford modules. See \cite[Chapter 3]{BGV;Book} for details.
\begin{lem}
\label{CliffLem}
Let $\rV$ be an even dimensional vector space over $\bbR$.
\begin{enumerate}
\item
The complexified Clifford algebra $\Cl (\rV) \otimes \bbC $ is isomorphic to the matrix algebra
$ \End (\rS)$, where $\rS$ is the spinor module;
\item
The spinor module $\rS$ is the only irreducible representation of $\Cl (\rV)$;
\item
For any Clifford module $\rE$,
$ \End (\rE ) \cong \Cl (\rV) \otimes \Hom _{\Cl (\rV)} (\Cl (\rV) , \rE),$
with isomorphism given by 
$ v \otimes T \mapsto T (v) .$
\end{enumerate}
\end{lem}

We turn to consider bundles of Clifford modules. 
Let $\cG \rightrightarrows \rM$ be a groupoid.
Let $\cA \rightarrow \rM$ be the Lie algebroid of $\cG$, equipped with a metric $g _\cA$.
Abusing notation we also use $g _\cA$ to denote the inner product on $\cA '$.
Then we define the {\it Clifford bundle}, 
to be the vector bundle 
$$\Cl (\cA ') := \bigcup _{x \in \rM} \Cl (\cA '_x , g _\cA (x)). $$
Note that $\Cl (\cA ')$ is also $\bbZ _2 $-graded and we write
$$ \Cl (\cA ') := \Cl (\cA ')^+ \oplus \Cl (\cA ')^- .$$
Analogous to the case of Clifford algebras, we define:
\begin{dfn}
A (bundle of) Clifford module is a $\bbZ _2$-graded Hermitian vector bundle 
$ \rE = \rE ^+ \oplus \rE ^- $ over $\rM$,
with an action map $\gamma \in \Gamma ^\infty (\cA \otimes \rE \otimes \rE ')$,
such that 
\begin{enumerate}
\item
For any $ \xi \in \cA' \subset \Cl (\cA ')$, 
$ \gamma (\xi ) : \rE \to \rE $ is skew-symmetric;
\item
Each $\rE _x , x \in \rM $ is a $\Cl (\cA '_x)$-module.
\end{enumerate}
\end{dfn}

A Hermitian $\cA$-connection $\nabla ^\rE$ is called {\it Clifford} if
for any $X \in \Gamma ^\infty (\cA), \xi \in \Gamma ^\infty (\cA '), u \in \Gamma ^\infty (\rE)$,
$$ \nabla ^\rE _X (\gamma (\xi ) u) = \gamma (\xi ) \nabla ^\rE _X u + \gamma (\nabla ^{g _\cA} _X \xi ) u, $$
where $\nabla ^{g _\cA} $ is the Levi-Civita connection.
It can be shown that Clifford $\cA$-connections always exist (see \cite[Section 6]{Nistor;GeomOp}).

Consider the pullback bundle $\bt ^{-1} \rE$. 
Any $\cA$-connection $\nabla ^\rE$ on $\rE$ uniquely determines a right-invariant family of connections,
still denoted by $\nabla ^{\rE}$ for simplicity,
on the $\bs$-fibers of $\cG$ by requiring that 
$$ \nabla ^{\rE} _{\tilde X} (\bt ^{-1} u) = \bt ^{-1} (\nabla ^\rE _X u), $$
for any right-invariant vector field $\tilde X$ with $\tilde X |_\rM = X$, 
and $u \in \Gamma ^\infty (\rE)$.
Furthermore, if $\rE $ is a $\Cl (\cA)$-module,
then $\bt ^{-1} \rE |_{\cG _x} $ is a $\Cl (T ^* \cG _x )$-module for each $x \in \rM$,
and $\nabla ^{\rE} $ is a Clifford connection in the usual sense.

The curvature of any even rank Clifford $\cA$-connection $\nabla ^\rE$ decomposes under the isomorphism 
$\End (\rE) \cong \Cl (\cA ') \otimes \End _{\End \Cl (\cA ')} (\rE) $ as
$$ \gamma (R) + F ^{\rE / \rS} , $$ 
where $R $ is the Riemannian curvature of $\cA$, considered as a section in 
$ \Gamma ^\infty (\wedge ^2 \rA ' \otimes \Cl (\rA) )$ with $\gamma $ acting on the $\Cl (\rA)$ factor, 
and $F ^{\rE / \rS} \in \Gamma ^\infty (\wedge ^2 \cA ' \otimes \End _{\Cl (\rA ') } (\rE))$
is known as the {\it twisting curvature}.   

\begin{dfn}
A (groupoid) {\it Dirac operator} is a differential operator from $\bt ^{-1} \rE $ to itself of the form
$$ \eth = (\bt ^{-1} \gamma ) \circ \nabla ^\rE ,$$
where $\nabla ^\rE $ is a right-invariant, Clifford connection on the $\bs$-fibers of $\cG$;
A {\it perturbed Dirac operator} is an operator of the form
$$ \eth + \varPsi \in \Psi ^1 _\mu (\cG , \rE ),$$
where $\eth $ is a Dirac operator, 
and $\varPsi $ is an odd degree operator in $ \Psi ^{- \infty } _\mu (\cG , \rE ) $
satisfying
$$ \varPsi (a ^{-1} ) = (\varPsi (a ) )^* , \quad \forall a \in \cG .$$ 
\end{dfn} 
It is easy to see that all Dirac operators are symmetric, hence essentially self adjoint.
From our definition, it is also clear that any perturbed Dirac operators are also essentially self-adjoint.


\subsubsection{\bf The reduced kernel and convolution product}
Let $\cG \rightrightarrows \rM$ be a groupoid with compact set of units $\rM$.
Recall that we fixed a fiberwise metric $g_\cA $ on the Lie algebroid $\cA$ and extended it to a Riemannian metric 
on each $\bs$-fiber by right translation.
Hence, one has a family of Riemannian volume densities $\mu _x$ on $\bs^{-1} (x) $.
We shall also regard $\mu \in \Gamma ^\infty (| \wedge ^{\mathrm {top}} \Ker (d \bs ) |)$.

\begin{dfn}
\label{ConvDfn}
For any pair of functions $f, g \in C^\infty (\cG)$,
such that $f (b) g (a b^{-1}) \in \bL^1 (\cG _{\bs (a)}, \mu_{ \bs (a) }), \quad \forall a \in \cG$,
the convolution product $f \circ g$ is defined to be
$$ f \circ g (a) := \int _{b \in \cG _{\bs (a)}} f (a b^{-1} ) g ( b ) \: \mu _{\bs (a)} (b). $$
\end{dfn}
In particular, the convolution product is well defined for any pair $f, g \in C^\infty _c (\cG) $,
and $f \circ g \in C^\infty _c (\cG)$. 
The resulting algebra $( C^\infty _c (\cG), \circ )$ is known as the {\it convolution algebra} of $\cG$.

The convolution product can also be defined for sections of vector bundles.
Let $\rE , \rF$ be vector bundles over $\rM$,
$f \in \Gamma ^\infty (\bt^{-1} \rE \otimes \bs ^{-1} \rE '), 
g \in \Gamma ^\infty (\bt ^{-1} \rE \otimes \bs ^{-1} \rF) $.
Since one has natural identifications 
$$ (\bt ^{-1} \rE \otimes \bs ^{-1} \rE ')_{a b^{-1}} \cong \bs ^{-1} \rE _{\bt (a)} \otimes \bt ^{-1} \rE ' _{\bt (b)},
\text { and } (\bt ^{-1} \rE \otimes \bs ^{-1} \rF) _b \cong \rE _{\bt (b)} \otimes \rF _{\bs (b)},$$
the point-wise multiplication 
$$ f (a b^{-1}) g (b) \in (\bt ^{-1} \rE \otimes \bs ^{-1} \rF )_a $$
is well defined for each $a \in \cG, b \in \cG _{\bs (a)}$,
using the pairing between $\rE ' _{\bt (b)} $ and $\rE _{\bt (b)}$.
Hence the convolution product can be defined as:
\begin{dfn}
\label{ConvDfn2}
For any $f \in \Gamma ^\infty (\bt ^{-1} \rE \otimes \bs ^{-1} \rE '), 
g \in \Gamma ^\infty (\bt ^{-1} \rE \otimes \bs ^{-1} \rF) $,
such that $f (b) g (a b^{-1}) $
is a $\bL^1 (\cG _{\bs (a)}, \mu_{ \bs (a) })$
section with values in $\rE _{\bt (a)} \otimes \rF _{\bs (a)} $ for all $a \in \cG$,
then the convolution product $f \circ g$ is defined to be
$$ f \circ g (a) := \int _{b \in \cG _{\bs (a)}} f (a b^{-1} ) g ( b ) \: \mu _{\bs (a)} (b)
\in \Gamma ^\infty (\bt ^{-1} \rE \otimes \bs ^{-1} \rF). $$
\end{dfn}

Alternatively, consider the set
$$ \tilde \cG := \{ (a , b) \in \cG \times \cG : \bs (a) = \bs (b) \}.$$
On $\tilde \cG$ one defines the natural maps
\begin{align*}
\tilde \bt : \tilde \cG \to \cG, & \quad \tilde \bt (a, b) := a \\
\tilde \bs : \tilde \cG \to \cG, & \quad \tilde \bs (a, b) := b \\
\bs ^{(2)} : \tilde \cG \to \rM, & \quad \bs ^{(2)} (a, b) := \bs (a) = \bs (b) \\
\bpi : \tilde \cG \to \cG, & \quad \bpi (a , b) = a b ^{-1} .
\end{align*}
Note that $\tilde \cG$ is just the fibered product groupoid of $\cG$, 
with source and target maps $\tilde \bs , \tilde \bt $.
Using the relations
$$ \bt \circ \widetilde \bm = \bt \circ \tilde \bt , \quad
\bs \circ \widetilde \bm = \bt \circ \tilde \bs , \quad
\text { and } \quad \bs \circ \tilde \bt = \bs ^{(2)} = \bs \circ \tilde \bs,$$
one naturally identifies the bundles (over $\tilde \cG$):
\begin{align*}
\widetilde \bm ^{-1} (\bt ^{-1} \rE \otimes \bs ^{-1} \rE ') 
\cong & \: \tilde \bt ^{-1} (\bt ^{-1} \rE ) \otimes \tilde \bs ^{-1} (\bt ^{-1} \rE ') \\
\tilde \bs ^{-1} \otimes (\bt ^{-1} \rE \otimes \bs ^{-1} \rF)
\cong & \: \tilde \bs ^{-1} (\bt ^{-1} \rE ) \otimes \tilde \bt ^{-1} (\bs ^{-1} \rF ).
\end{align*}
Hence, one can rewrite Definition \ref{ConvDfn2} using the language of Appendix \ref{DGNonsense} as 
\begin{equation}
\label{ConvDef3}
f \circ g (a)
= \int _{ (b' , b) \in \tilde \bt ^{-1} (a)}
\big( \widetilde \bm ^{-1} f (b' , b) \big)
\big( \tilde \bs ^{-1} g (b' , b) \big) \tilde \mu (b' , b),
\end{equation}
where $\tilde \mu \in \Gamma ^\infty ( | \wedge ^{\mathrm {top}} \ker (d \tilde \bs ) |)$ 
is defined by $\tilde \mu = \mu $ at $ \tilde \bs ^{-1} (b') \cong \bs ^{-1} (\bs (b))$,
regarded as a family of measures (densities) on the fibers.  

\begin{dfn}
\label{RedKer}
For any $\varPsi = \{ \varPsi _x \}_{x \in \rM} \in \Psi ^\infty (\cG)$.
The {\it reduced kernel} of $\varPsi $ is defined to be the distribution
$$ K_\varPsi (f) :=  \int _\rM \mathbf u^* (\varPsi (\mathbf i^* f)) (x) \: \mu_\rM (x), 
\quad f \in C^\infty _c (\cG),$$
where $\mathbf i$ and $\mathbf u$ denote respectively the inversion and unit inclusion.
\end{dfn}

Observe that, if $\varPsi \in \Psi ^{- \infty} (\cG)$, 
then $K_\varPsi \in C^\infty (\cG)$, i.e., there exists $\kappa \in C^\infty (\cG) $ such that
$$ K_\varPsi (f) = 
\int _{x \in \rM} \left( \int _{b \in \cG _x} \kappa (b) f (b^{-1}) \: \mathbf i ^* \mu _{\bs (b)} \right) \mu _\rM , 
\quad \forall f \in C^\infty _c (\cG), $$
and one can recover $\varPsi $ by the formula:
$$ \varPsi (f) (a) = \int _{\cG _{\bs (a)}} \kappa (a b ^{-1}) f (b) \: \mu_{\bs (a) } (b). $$

\begin{rem}
In \cite{NWX;GroupoidPdO}, the authors defined the reduced kernel canonically using 1-densities.
\end{rem}  
 
One particularly important property of the reduced kernel of a 
pseudo-differential operator on a groupoid is the following:
\begin{lem}
\cite[Corollary 1]{NWX;GroupoidPdO}
For any $\varPsi \in \Psi ^\infty (\cG)$, the reduced kernel is co-normal at $\rM$ and smooth elsewhere.
\end{lem}

\subsubsection{\bf Some representations of $\Psi ^\infty (\cG)$}
In this section, we recall some homomorphisms from $\Psi ^\infty _\mu (\cG)$ to other spaces of operators. 
The materials in this section can be found in \cite{Nistor;GeomOp}.
Let $\cG \rightrightarrows \rM$ be an $\bs$-connected Lie groupoid. 
Let $\cA$ be the Lie algebroid of $\cG$. 

\begin{dfn}
\label{OpNorm}
Given any $\varPsi \in \Psi _\mu ^{-n-1} (\cG)$,
define the 1-norm of $\varPsi $ by (see \cite[Equation (16)]{Nistor;GeomOp})
\begin{equation}
\| \varPsi \|_1 := \sup _{x \in \rM} 
\left\{ \int _{ \cG _{x} } | \kappa (a) | d \mu _x (a) , 
\int _{ \cG _{x} } | \kappa (a^{-1}) | d \mu _x (a) \right\},
\end{equation}
where $\kappa (a)$ is the reduced kernel of $\varPsi $.
Note that $\kappa $ is continuous because $\varPsi \in \Psi _\mu ^{-n-1} (\cG)$.

Next, we define the full norm of any $\varPsi \in \Psi _\mu ^0 (\cG)$ by
\begin{equation}
\| \varPsi \| := \sup _{\rho } \| \rho (\varPsi ) \| _{\bH},
\end{equation} 
where $\| \cdot \|_\bH$ is just the operator norm,
and the supremum ranges through all bounded representation $\rho $ of $\Psi ^0 _\mu (\cG)$ on $\bH$ satisfying
$$ \| \rho (\varPsi ) \| _\bH \leq \| \varPsi \|_1 , \quad \forall \varPsi \in \Psi ^0 _\mu (\cG). $$

We denote the closure of $\Psi ^0 _\mu (\cG)$ under $\| \cdot \|$ by 
$$\fU (\cG),$$
and the closure of $\Psi ^{- \infty } _\mu (\cG) $ under $\| \cdot \|$ by 
$$\fC^* (\cG) \subset \fU (\cG).$$
\end{dfn}

Another important homomorphism is the so called vector representation,
which defines the class of (leafwise)-differential operators on a manifold that we are interested in: 
\begin{dfn}
The {\it vector representation} is the homomorphism defined by
$\nu : \Psi _\nu ^\infty (\cG) \rightarrow \End (C^\infty (\rM))$,
$$ (\nu (\varPsi ) u)(x) := \varPsi _x (\bt ^{-1} u)(x). $$ 
\end{dfn}

\begin{rem}
Equivalently, one can define $(\nu (\varPsi ))u$ to be the (unique) function on $\rM$ satisfying
$ (\nu (\varPsi ))u \circ \bt = \varPsi (u \circ \bt) $
\end{rem}

\begin{rem}
Observe that if $X \in \Gamma ^\infty (\cA)$ is regarded as a differential operator on $\cG$,
then the vector representation of $X$ is just $\nu (X)$,
the image of $X$ under the anchor map (regarded as a differential operator on $\rM$),
so there is no confusion using the same notation for both.
\end{rem}

\vspace{15cm}
$ \; $

\pagebreak \thispagestyle{firstpage}
\section{Elliptic and Fredholm operators}
Using the same arguments as in the construction of parametrices of elliptic pseudo-differential operators on a manifold, 
one has:
\begin{lem}
Let $\varPsi \in \Psi ^{[m]} _\mu (\cG) $ be elliptic. 
Then there exists an operator $Q \in \Psi ^{[-m]} _\mu (\cG)$, 
known as the parametrix of $\varPsi $, such that
\begin{equation}
\label{CrudePara}
R _1 = \varPsi \circ Q - \id \text { and } R _2 = Q \circ \varPsi - \id 
\end{equation}
are elements in $ \Psi ^{- \infty } _\mu (\cG)$.
\end{lem}

If $\cG$ is the pair groupoid over a compact manifold,
then all elements in $\Psi ^{-\infty } (\cG)$ are compact. 
It follows from Equation (\ref{CrudePara}) that all elliptic operators are Fredholm.
Unfortunately, in general, elements in $\Psi ^{- \infty } _\mu (\cG)$ are not compact operators.
In the following section we review a Fredholmness criterion given by Lauter and Nistor \cite{Nistor;GeomOp}. 
Here, we first recall the notion of an invariant sub-manifold.
\begin{dfn}
Let $\cG \rightrightarrows \rM$ be a groupoid.
A proper sub-manifold $ \rZ \subset \rM $ is called an {\it invariant sub-manifold} if 
$\bs ^{-1} (\rZ ) = \bt ^{-1} (\rZ ) $.
For an invariant sub-manifold,
we denote $\cG _{\rZ} := \bs ^{-1} (\rZ)$.
It is clear that $\cG _\rZ $ is a groupoid over $\rZ $ by restricting the groupoid structure on $\cG$.
Also, for any $\varPsi = \{ \varPsi _x \} _{x \in \rM } \in \Psi ^{\infty } ( \cG )$, 
define the {\it restriction} of $ \varPsi $ to be the operator
$$ \varPsi |_ \rZ := \{ \varPsi _x \} _{x \in \rZ } \in \Psi ^{\infty } (\cG _\rZ ).$$
\end{dfn}
 
\subsection{ Lauter and Nistor's Fredholmness criterion }
\label{LauNis}
Let $\cG \rightrightarrows \rM $ be a groupoid with compact units $\rM$. 
Assume that the anchor map $\nu : \cA \to T \rM $ is an isomorphism when restricted to some open dense subset
$\rM_0 \subseteq \rM$. 
Then one can also define the metric 
$$ g _{\rM _0} (X, Y) := g_\cA (\nu^{-1} X , \nu ^{-1} Y) , \quad \forall X, Y \in T_x \rM_0 , x \in \rM_0. $$
By definition, it is clear that $\bt |_{ \cG _{x} } : \cG _x \rightarrow \rM_0$ is a local isometry.

Following \cite{Nistor;GeomOp}, we shall make the following assumptions:
\begin{dfn}
\label{BdGpoid}
An $\bs$-connected groupoid $\cG \rightrightarrows \rM$ is said to be a {\it Lauter-Nistor groupoid} if
\begin{enumerate}
\item
The unit set $\rM $ is compact;
\item
The anchor map $\nu : \cA \rightarrow T \rM$ is bijective over an open dense subset $\rM_0 \subseteq \rM$;
\item
The Riemannian manifold $(\rM_0 , g _{\rM _0}) $ has positive injectivity radius
and has finitely many connected components $\rM _0 = \coprod _\alpha \rM _\alpha $;
\item
As a groupoid,
$\cG_{\rM_0} \cong \coprod _\alpha \rM_\alpha \times \rM_\alpha $, the pair groupoid.
\end{enumerate}
\end{dfn}
Note that condition (2) implies the Lie algebroid is integrable, 
using the following result from Debord \cite{Debord;IntAlgebroid}.
\begin{thm}
Let $\cA$ be a Lie algebroid over $\rM$, with anchor map $\nu : \cA \rightarrow T \rM$.
Suppose that there exists an open dense subset $U \subset \rM$ such that $\nu $ is injective on $\cA |_U$.
Then $\cA$ is integrable.
\end{thm} 
Indeed, we shall mainly be studying examples where the groupoid is explicitly given.

The following lemma is useful for verifying assumption (4).
\begin{lem}
If all connected components of $\rM _0$ are simply connected, 
then $\cG_{\rM_0} \cong \coprod _\alpha \rM_\alpha \times \rM_\alpha $.
\end{lem}
\begin{proof}
Observe that, for each $x \in \rM _0$, $\cG _x $ is a covering of $\cF _x $, 
the connected component in $\rM _0 $ containing $x $.
If all connected components of $\rM _0$ are simply connected, 
then $\cG _x \cong \cF _x $ for all $x \in \rM _0$.
It follows from Proposition \ref{btImage} that the isotropy subgroups $\cG ^x _x $ are trivial for all $x \in \rM _0 $.
Hence the assertion.
\end{proof}

Since the Riemannian curvature is a local object, it follows that $( \rM _0 , g _{\rM _0} ) $ 
is a manifold with bounded geometry.
Also, it is easy to see that for any vector bundle $\rE \to \rM$, 
the restriction $\rE | _{\rM _0} \to \rM _0 $ is a vector bundle of bounded geometry.
Hence one can consider the Sobolev spaces $\bW ^l (\rM _0 , \rE )$ for any $l \in \bbR$.

Let $\varPsi = \{ \varPsi _x \} \in \Psi ^{[m]} _\mu (\cG , \rE)$.
For any $x \in \rM _0$,
assumption (4) enables one to identify 
$$ \cG _x \cong \rM _\alpha ,$$
where $M _\alpha $ is the connected component of $\rM _0 $ containing $x$.
Hence one identifies $\Gamma ^\infty (\rM _\alpha  , \rE ) \cong \Gamma ^\infty (\cG _x , \bs ^{-1} \rE ) $.
Under such identification,  
one has 
\begin{equation}
\label{VecSobo}
\nu (\varPsi) (f) |_{\rM _\alpha } = \varPsi _x (f |_{\rM _\alpha }) ,
\end{equation}
for any $f \in \Gamma ^\infty (\rM , \rE) $. 
Since $\varPsi _x$ is a pseudo-differential operator of order $\leq m$,
Equation (\ref{VecSobo}) enables one to extend the vector representation $\nu (\varPsi)$
to a bounded map
$$ \nu _l (\varPsi) : \bW ^l (\rM _0 , \rE ) \to \bW ^{l - m} (\rM _0 , \rE) ,$$
for any $ l \geq m$.
In particular, if $\varPsi \in \Psi ^{- \infty } _\mu (\cG , \rE)$,
then $\nu _0 (\varPsi)$ is just the smoothing map
\begin{equation}
\label{LNOpen}
\nu _0 (\varPsi) f (x) = \int _{y \in \rM _\alpha } \psi |_{\cG _{\rM _0}} (x , y) f (y ) \mu _{\rM _0} (y), 
\quad f \in \bL ^2 (\rM , \rE) ,
\end{equation}
where $\psi \in \Gamma ^\infty _c (\cG , \bt ^{-1} \rE \otimes \bs ^{-1} \rE ')$ is the reduced kernel of $\varPsi$,
and we have used the identification $\cG _{\rM _0} \cong \coprod _\alpha \rM _\alpha \times \rM _\alpha $.

Recall that we defined $\fU (\cG)$ and $\fC^* (\cG)$ to be the closure of 
$\Psi ^0 _\mu (\cG)$ and $\Psi ^{-\infty} _\mu (\cG)$ under the full norm $\| \cdot \|$ respectively.
We shall denote $\fJ := \fC^* (\cG_{\rM _0})$ 
(the closure of pseudo-differential operators of order $-\infty$ on the groupoid over $\rM_0$).
The importance of $\fJ$ lies in 
\begin{lem}
For any $\varPsi \in \fJ$, the vector representation $\nu (\varPsi )$ is a compact operator on $\bL^2 (\rM_0)$.
\end{lem}
\begin{proof}
If $\varPsi \in \Psi _\mu ^{- \infty } (\cG _{\rM_0})$, 
then Equation (\ref{LNOpen}) says that $\nu (\varPsi )$ is just a properly supported, smoothing operator on $\rM_0$,
which is well known to be compact. 
The assertion follows by taking limits.
\end{proof} 

One remarkable fact about these spaces is the following lemma:
\begin{lem}
\label{ExactC^*}
\cite[Lemma 2]{Nistor;Family}
One has short exact sequences
\begin{align*}
0 & \rightarrow \fC^* (\cG _{\rM _0}) = \fJ \rightarrow \fC^* (\cG) 
\rightarrow \fC^* (\cG _{\rM \setminus \rM_0}) \rightarrow 0 \\
0 & \rightarrow \fU (\cG _{\rM _0}) \rightarrow \fU (\cG) 
\rightarrow \fU (\cG _{\rM \setminus \rM_0}) \rightarrow 0.
\end{align*}
\end{lem}

Another useful fact about Lauter-Nistor groupoids is that their vector representation is faithful.
In other words:
\begin{lem}
{\em \cite{Nistor;Polyhedral3}}
The map $\nu : \Psi ^{\infty } _\mu (\cG , \rE ) \to \End (\Gamma ^\infty _c (\bt ^{-1} \rE ) )$ is injective.
\end{lem}
\begin{proof}
Let $\varPsi = \{ \varPsi _x \} _{x \in \rM } \in \Psi ^{\infty } _\mu (\cG , \rE ) $ be such that 
$\nu (\varPsi ) = 0 $.
First consider $\varPsi _x $ for arbitrary $x \in \rM _0 $.
For any $u \in \Gamma ^\infty _c (\bt ^{-1} \rE |_{\cG _x} )$, 
let $\tilde u \in \Gamma ^\infty _c (\rE )$ be the extension of $u$ by $0$.
Then 
$$ \varPsi _x (u ) = \nu (\varPsi ) (\tilde u) |_{\rM _0 } = 0. $$
Therefore $\varPsi _x = 0$ for any $x \in \rM _0 $.
Now consider $x \in \rM \setminus \rM _0 $.
For any $u \in \Gamma ^\infty _c (\bt ^{-1} \rE |_{\cG _x} ) $,
let $\hat u \in \Gamma ^\infty _c (\bt ^{-1} \rE )$ be any extension of $u$.
Then one has
$$ \varPsi (\hat u) = 0 $$
on $\cG _{\rM _0}$,
because $\varPsi _x = 0 $ for any $x \in \rM _0$.
Since $ \varPsi (\hat u) $ is continuous and $\cG _{\rM _0 }$ is dense in $\cG $, 
it follows that $\varPsi _x (\hat u) = 0$ everywhere, hence $\varPsi = 0$.
\end{proof}

In the following theorem, 
let $A $ be any fixed elliptic (pseudo)-differential operator of order $k > 0 $.
(One can take $A $ to be, say, a Laplacian operator in Definition \ref{LapDfn}). 
Then $(\id + A ^* A ) ^{- \frac{1}{2 k}} $ is well defined by functional calculus.
Moreover, by \cite[Theorem 4]{Nistor;GeomOp} and its corollaries,
$\nu \big( (\id + A ^* A ) ^{\frac{m}{2 k}} \big) : \bW ^m (\rM _0 , \rE ) \to \bW ^ 0 (\rM _0 , \rE ) $ 
is bounded for all $m$.
With these preliminaries, the main result of Lauter and Nistor can be stated as:
\begin{thm}
\label{Nis;Thm7}
\cite [Theorem 7]{Nistor;GeomOp}
For any $\varPsi \in \Psi ^0 (\cG)$, or $\varPsi \in \Psi ^{[m]} (\cG)$ elliptic self-adjoint,
the spectrum and essential spectrum of $\nu (\varPsi )$ satisfy
\begin{equation}
\label{nuSpec}
\bsigma (\nu (\varPsi )) \subseteq \bsigma _{\fU (\cG)} (\varPsi) 
\; \text { and } \; \bsigma ^e (\nu (\varPsi )) \subseteq \bsigma _{\fU / \fJ } (\varPsi ). 
\end{equation}
In particular, for any $\varPsi \in \Psi ^{[m]}_\mu (\cG)$ such that 
$ \varPsi (\id + A^* A)^{- \frac{m}{2 k}} $ is invertible in 
then $\fU (\cG) / \fJ $, $\nu (\varPsi ) $ extends to a Fredholm operator 
from $\bW^m (\rM _0 , \rE) $ to $ \bL^2 (\rM _0 , \rE)$;
if \\ $ \varPsi (\id + A^* A)^{- \frac{m}{2 k}} \in \fJ $, then 
$\nu (\varPsi ) : \bL^2 (\rM _0 , \rE) \to \bL^2 (\rM _0 , \rE)$ is compact.
\end{thm}

\begin{proof}
By definition, for each $\lambda \in \bbC \setminus \bsigma _{\fU (\cG) / \fJ}$, 
there exists $ Q \in \fU (\cG) $ such that
$$ (\varPsi - \lambda ) Q - \id _{\fU (\cG)} , Q (\varPsi - \lambda ) - \id _{\fU (\cG)} \in \fJ. $$
Since $\nu $ maps $\fJ$ to compact operators, it follows that $\nu (\varPsi )$ is Fredholm,
hence $\lambda \in \bbC \setminus \bsigma ^{e} (\nu (\varPsi ))$.
The second inclusion follows by contra-positivity.  
The first inclusion is similar (with $\fJ $ replaced by $\{ 0 \} $).

To prove that $\nu (\varPsi )$ is Fredholm (resp. compact) from the hypothesis,
observe that $\nu (\varPsi ) = \nu ( \varPsi (\id + A^* A)^{- \frac{m}{2 k}} ) \nu ((\id + A^* A)^{\frac{m}{2 k}}) $,
and use the well known fact that the composition between a Fredholm (resp. compact) operator 
and a bounded invertible operator is Fredholm (resp. compact). 
\end{proof}  

Using the injectivity of the vector representation,
and the fact that injective homomorphisms of $C ^*$-algebra preserve the spectrum 
\cite[p.12]{Averson;Book},
the inclusion in Equation (\ref{nuSpec}) can be sharpen to an equality.
In particular:
\begin{thm}
\label{Nis;Thm8}
\cite [Theorem 8]{Nistor;GeomOp}
Suppose the groupoid $\cG$ is Hausdorff. 
Then, for any $\varPsi \in \Psi ^0 (\cG)$, or $\varPsi \in \Psi ^{[m]} (\cG)$ elliptic self-adjoint,
the spectrum and essential spectrum of $\nu (\varPsi )$ satisfy
\begin{equation}
\bsigma (\nu (\varPsi )) = \bsigma _{\fU (\cG)} (\varPsi) 
\; \text { and } \; \bsigma ^e (\nu (\varPsi )) = \bsigma _{\fU / \fJ } (\varPsi ). 
\end{equation}
\end{thm}

Suppose that $\rM \backslash \rM_0$ is a disjoint union of closed immersed invariant sub-manifolds
$$ \rM \backslash \rM_0 = \bigcup _{j=1}^k \rZ _k . $$  
Then the hypothesis of Theorem \ref{Nis;Thm7} can be made more explicit by 
\begin{thm}
\label{Nis;Thm10}
\cite[Theorem 10]{Nistor;GeomOp}
For any $\varPsi \in \fU (\cG)$,
the spectrum $\varPsi + \fJ$ in $\fU (\cG) / \fJ$ can be written as a union
$$\bsigma _{\fU (\cG) / \fJ} (\varPsi + \fJ ) = 
\bsigma _{S (\cA^*)} (\sigma _{\mathrm {top}} (\varPsi )) 
\bigcup \bigcup _{j=1}^k \bsigma _{\fU (\cG _{\rZ_j})} (\varPsi |_{\rZ _j}), $$
where $\sigma _{\mathrm {top}} (\varPsi )$ is the principal symbol of $\varPsi $.
\end{thm}
\begin{proof}
It suffices to prove that the homomorphism 
$$ \varPsi + \fJ 
\mapsto \sigma _{\mathrm {top}} (\varPsi ) \oplus \varPsi \rvert_{\rZ_1 } \oplus 
\cdots \oplus \varPsi \rvert _{\rZ_j} $$
is injective. That is true because
$\sigma _{\mathrm {top}} (\varPsi ) = 0 $ implies $\varPsi \in C^* (\cG)$,
and the first exact sequence of Lemma \ref{ExactC^*} implies $\varPsi \in \fJ$.
\end{proof}

Combining Theorem \ref{Nis;Thm7} and Theorem \ref{Nis;Thm10}, we get
\begin{cor}
\label{NisLem}
\cite[Theorem 10]{Nistor;GeomOp}
Given an elliptic operator $\varPsi \in \Psi ^{[m]} _\mu (\cG) , m \geq 0$. 
Suppose for all invariant sub-manifolds $\rZ _j$,
there exist $\varPhi _j \in \Psi ^{-m} (\cG _{\rZ _j} , \rE |_{\rZ _j} ) \bigcap \fU (\cG _{\rZ _j}) $
such that 
$$ (\varPsi |_{\rZ_j} ) \varPhi _j = \varPhi _j ( \varPsi |_{\rZ _j } ) = \id ,$$
then $\nu (\varPsi )$ is Fredholm.
\end{cor} 


\subsection{ Application: Fredholm operators on the Bruhat sphere}
In this section, we study the Bruhat sphere $\bbC \rP (1)$ in greater detail.

\subsubsection{\bf The Bruhat sphere and its symplectic groupoid}
The Bruhat Poisson structure is obtained by taking $\rG = \rSL (2 , \bbC) , \rK = \rSU (2) $,
and $\rA \rN = $ set of upper diagonal matrices in Example \ref{BruhatExam}.
It is well known that the Bruhat sphere has two $\cA$-leaves:
${}_\rT e$ and its complement. 
As we have seen in Example \ref{BruhatExam}, the symplectic groupoid over the Bruhat sphere is
$ \rT \backslash (\rSU (2) \times \rN ) . $
Here, we describe the groupoid structure in greater detail.
\begin{nota}
\label{BruhatPair}
Let $\alpha , \beta , w \in \bbC , |\alpha |^2 + |\beta |^2 = 1 $.
Then we write
$$ [\alpha , \beta ] _\rT ^ { w } :=
\left(
\left( 
\begin{smallmatrix}
\alpha & \beta \\
- \bar \beta & \bar \alpha 
\end{smallmatrix}
\right), 
\left( 
\begin{smallmatrix}
1 & w \\
0 & 1 
\end{smallmatrix}
\right)
\right) \rT \in \cG = \rT \backslash (\rSU (2) \times \rN ) .
$$
Also, recall that one can define stereographic coordinates 
\begin{align*}
z =& \: x + \imath y \mapsto [ z , 1 ] \in \bbC \rP (1) - [1 , 0] , \quad x , y \in \bbR \\
\dot z =& \: \dot x + \imath \dot y \mapsto [ 1 , \dot z ] \in \bbC \rP (1) - [0 , 1] ,
\quad \dot x , \dot y \in \bbR .
\end{align*}
Then the source submersion $\bs$ can be trivialized as
\begin{align*}
\bx (z , w) 
:=& \: \left[ \frac{\bar w - z}{(1 + |\bar w - z|^2)^\frac{1}{2}} , 
\frac{1}{(1 + |\bar w - z|^2)^\frac{1}{2}} \right] _\rT ^{ w },
\quad z, w \in \bbC. \\
\dot \bx ( \dot z , \dot w) 
:=& \: \left[ \frac{\dot z \bar {\dot w} - 1}{(|\dot z|^2 + |\dot z \bar {\dot w} - 1 |^2)^\frac{1}{2}} , 
\frac{\dot z}{(|\dot z | ^2 + |\dot z \bar {\dot w} - 1|^2)^\frac{1}{2}} \right] _\rT ^{ \dot w },
\quad \dot z, \dot w \in \bbC.
\end{align*}
\end{nota}

For any
$ k = \left( 
\begin{smallmatrix}
\alpha & \beta \\
- \bar \beta & \bar \alpha 
\end{smallmatrix}
\right) \in \rK, n = 
\left( 
\begin{smallmatrix}
1 & w \\
0 & 1 
\end{smallmatrix}
\right) \in \rN,$
one has the Iwasawa decomposition $n k = k' a' n' $,
where 
$$ k ' = \left( 
\begin{array}{cc}
\alpha ' & \beta ' \\
- \bar \beta ' & \bar \alpha '
\end{array}
\right) \in \rK,
\alpha ' = \frac{\alpha - \bar w \beta }{(|\beta |^2 + |\alpha - \bar w \beta |^2)^{\frac{1}{2}}},
\beta '  = \frac{ \beta }{(|\beta |^2 + |\alpha - \bar w \beta |^2)^{\frac{1}{2}}}.$$
Hence, one can easily write down the source, target and inverse maps
\begin{align}
\label{CP2Formula}
\bs ( [\alpha , \beta ] _ \rT ^{ w }) =& \: [\alpha , \beta ] \\ \nonumber
\bt ( [\alpha , \beta ] _ \rT ^{ w }) =& \: [\alpha  - \bar w \beta , \beta ] \\ \nonumber
([\alpha , \beta ] _\rT ^{ w })^{-1} =& \: 
\left[ \frac{\alpha - \bar w \beta }{(|\beta |^2 + |\alpha - \bar w \beta |^2)^{\frac{1}{2}}},
\frac{ \beta }{(|\beta |^2 + |\alpha - \bar w \beta |^2)^{\frac{1}{2}}} \right] ^{- w} _\rT. 
\end{align}
It follows that in the $\bx $ and $\dot \bx$ coordinates
$\bs (\bx (z , w) ) = [z , 1] , \bs (\dot \bx (\dot z , \dot w)) = [ 1 , \dot z ]$.
The inverse can also be written down:
$$ ([1 , 0] ^ w _ \rT) ^{-1} = ([1 , 0] ^{- w} _\rT) , (\bx (z , w) )^{-1} = \bx (z + \bar w , - w) ,
\quad \forall z , w \in \bbC. $$
 

\begin{rem}
It is clear that the symplectic groupoid defining the Bruhat Poisson sphere is a Lauter-Nistor groupoid.
Indeed, many Poisson homogeneous spaces constructed by Lu (see \cite{Lu;PoissonCohNotes}), with open symplectic leaves, 
have symplectic groupoids satisfying the Lauter-Nistor conditions.
Finally, note that we shall not use the symplectic structure in this thesis.
\end{rem}

The Poisson bi-vector field can be also be explicitly written down \cite{Ryot;NecklaceCP1}.
On the stereographic coordinate patch excluding ${}_\rT e$, one has
$$ \varPi = (1 + x^2 + y^2 ) \partial _x \wedge \partial _ y; $$
On the opposite coordinate patch one has
$$ \varPi = (\dot x^2 + \dot y^2)(1 + \dot x^2 + \dot y^2) \partial _{\dot x } \wedge \partial _{\dot y}.$$

As an illustration, we describe the metric on the open leaf induced by the Poisson bi-vector field.
For simplicity, take the round metric on the sphere
$$ (1 + x^2 + y^2)^{-2} (d x^2 + d y^2), $$
and the dual metric on $\cA = T^* \bbC \rP (1)$:
$$ g_\cA := (1 + x^2 + y^2)^2 ((\partial _x )^2 + (\partial _y)^2). $$ 
Then the metric on the open leaf $\bbC \rP (1) - \{ {} _\rT e  \} $ is defined by
\begin{align*}
g_\cA (\nu ^{-1} \partial _x, \nu ^{-1} \partial _x)
= \: & g_\cA ((1+ x^2 +y^2)^{-1} d y, (1+ x^2 + y^2)^{-1} d y) = 1 \\
= \:& g_\cA (\nu ^{-1} \partial _y, \nu ^{-1} \partial _y), \\
g_\cA (\nu ^{-1} \partial _x, \nu ^{-1} \partial _y) = \: & 0,
\end{align*}
where $\nu (\omega ) := \iota _{\omega } \varPi, \: \forall \omega \in T^* \bbC \rP (1) $ is the anchor map

\begin{rem}
Here, we observe that the metric we obtained is just the Euclidean metric on $\bbR ^2$.
On the polar coordinates $ (\dot r , \dot \vartheta ) \mapsto \dot \bx ( \dot r e ^{i \dot \vartheta })$,
the metric $g _{\rM _0 } $ is just $ \dot r ^{-1} d \dot r ^2 + d \vartheta ^ 2 $.
A metric of this form is known as `scattering metric' in the edge calculus literature (see \cite{Albin;EdgeInd}).
We shall use this fact later in Section 5. 
However, it is important to note that the compactification to the Bruhat sphere is {\it not} the same as the standard
compactification to the disk with boundary. 
\end{rem} 

\subsubsection{\bf Inverse and the Laplace-Fourier transform}
Observe that, over ${}_\rT e$, one has 
$$ \bs ^{-1} ({}_\rT e) = \bt ^{-1} ({}_\rT e) = \rN \cong \bbR^2 $$
as a Lie group.
Therefore, given any pseudo-differential operator $\varPsi = \{ \varPsi _x \} _{x \in \bbC \rP (2)}$,
it follows that $\varPsi _{{}_\rT e}$ is an operator on $\bbR^2$ that is invariant under translation.
As we shall see in this section, 
the simple structure on $\bbR^n$ enables one to study inverses through the Laplace-Fourier transform,
which in turn gives a simple Fredholmness criterion. 

Set $\nabla$ be the usual flat, translation invariant connection on $\bbR^n$, $\chi = 1$ on $\bbR^n \times \bbR^n$.
One can regard $\bbR^n$ as a groupoid over a one point space.
Recall, from Definition \ref{GpoidTotalDfn}, 
the total symbol of any properly supported $\varPsi _{ {} _\rT e} \in \Psi ^\infty _\varrho (\bbR^n)$ is defined by
\begin{equation}
\label{FourierLaplace}
\sigma (\zeta ) := (\varPsi _{ {} _\rT e} )_p (e^{-i \langle p , \zeta \rangle}).
\end{equation}
By virtue of Lemma \ref{Kennedy}, one has
$$ \varPsi _{ {} _\rT e} (f) (p) 
= \int _{\zeta \in \bbR^n} \sigma (\zeta ) e^{i \langle p, \zeta \rangle} \hat f (\zeta ) \: d \zeta.$$
It would be useful to consider $\varPsi $ as convolution with a distribution.
Define 
$$ \psi (f) := \varPsi _{ {} _\rT e} (f (- p)) (0) 
= \int _{\zeta \in \bbR^n} \sigma (\zeta ) \int _{q \in \bbR^n} e^{i \langle q, \zeta \rangle} f (q) \: d q d \zeta, $$
so that one has
$$ \varPsi _{ {} _\rT e} (f) (p) = \psi _q (f (p - q)).$$
Note that $\psi $ is just the reduced kernel in Definition \ref{RedKer}, regarding $\bbR^n$ as a groupoid over a point.

Assume that one has the estimate 
$$ C (1 + |\zeta |)^m \geq |\sigma (\zeta  )| \geq C' (1 + |\zeta |)^m > 0 $$
for some constants $C, C' > 0 $ (which implies that $\varPsi $ is elliptic of order $m$).
It is straightforward to check that $(\sigma (\zeta ))^{-1} $ is also a symbol. 
Since the symbol map is a homomorphism, it follows that the inverse of $\varPsi $ is given by
\begin{equation}
\varPsi _{ {} _\rT e} ^{-1} (f) (p) 
= \int _{\zeta \in \bbR^n} (\sigma (\zeta ))^{-1} e^{i \langle q, \zeta \rangle} \hat f (\zeta ) \: d \zeta .
\end{equation}

Next, we describe the kernel of $\varPsi ^{-1}$ in greater detail. 
Note that Equation (\ref{FourierLaplace}) is still valid for $\zeta \in \bbC^n$.
Such extension is known as the Laplace-Fourier transform and shall be denoted by $\tilde \sigma (\zeta )$
or $\fF (f) $ if $f \in C ^\infty _c (\bbR ^n)$.
Indeed, one has
\begin{lem}
For any properly supported, invariant pseudo-differential operator $\varPsi $ on $\bbR^n$,
the Laplace-Fourier transform $\tilde \sigma (\zeta )$ is a holomorphic function on $\bbC^n$.
\end{lem}
 
In the case when $\varPsi $ is a differential operator, 
it was shown in \cite[Chapter 4.2]{Shimakura;Book} that the reduced kernel of $\varPsi ^{-1}$ decays exponentially,
depending on the zeros of $\tilde \sigma (\zeta )$, i.e., the poles of $\tilde \sigma (\zeta )^{-1}$.
Here, we prove a similar result for general pseudo-differential operators.
\begin{prop}
\label{DecayProp}
Let $\fH $ be a holomorphic function on the strip 
$$ S_\theta := \{ (\zeta _1 , \cdots \zeta _n ) \in \bbC^n : | \im (\zeta _i) | < \theta , \quad \forall i \},$$
and satisfies the estimate
\begin{equation}
\label{HoloEst}
\left| \partial _I \fH (\zeta ) \right| \leq C _I (1 + |\zeta |)^{m - |I|} , \quad \zeta \in S_\theta ,
\end{equation}
for each multi-index $I$ and some $C_I > 0$, $m \in \bbR$.
Let $\kappa$ be the distribution
$$ \kappa (f) := \int _{\zeta \in \bbR^n} \fH (\zeta ) \hat f (\zeta ) \: d \zeta, \quad f \in C^\infty _c (\bbR^n).$$
Then $\kappa $ is $C^\infty$ on $\bbR^n \backslash \{ 0 \}$.
Furthermore, for any $0 < \varepsilon < \theta $, 
one has
$$ \kappa |_{\bbR \backslash \{ 0 \}} = e^{- \varepsilon |p|} F , \quad \forall | p | > 1 $$
for some smooth function function $F$ with bounded derivatives.
\end{prop} 
\begin{proof}
First of all, since $\varsigma (\zeta ), \zeta \in \bbR^n$ is a symbol, 
it is well known that $\kappa $ is $C^\infty$ on $\bbR^n \backslash \{ 0 \}$,
and for any natural number $N$ and multi-index $I$, there exists $C_{I, N} > 0$ such that
\begin{equation}
\label{PdOKerEst}
|\partial _I \kappa ( p )| \leq C_{I, N} (1 + | p |)^N , \quad \forall \zeta \in \bbR^n , | p | \geq 1.
\end{equation}

By the well known Paley-Weiner theorem, $\fF (f)$ is holomorphic on $\bbC$ for any $f \in C^\infty _c (\bbR^n)$,
and for any natural number $N$,
there exists constants $C_N$ such that 
$$ \left| \fF (u) (\zeta ) \right| \leq C_N (1 + |\zeta |)^{-N} $$
for any $\zeta \in S _\theta $.
Using Equation (\ref{HoloEst}) in the hypothesis, the integrand 
$$ \fH (i (\varepsilon , \varepsilon , \cdots , \varepsilon ) + \zeta )
\times \fF (f) (i (\varepsilon , \varepsilon , \cdots , \varepsilon ) + \zeta ), \quad \zeta \in \bbR^n $$
lies in $L^1 (\bbR^n)$ for any $0 < \varepsilon < \theta $.
Therefore we can use Fubini's theorem to compute the integral
$$ \int _{\zeta \in \bbR^n} \fH (\zeta ) \fF (f) \: d \zeta 
= \int \cdots \int \left( \int \fH (\zeta ) \fF (f) (\zeta ) d \zeta _1 \right) d \zeta _2 \cdots d \zeta _n. $$
We then use the Cauchy integral formula to shift the contour of $\zeta _1$-integration to 
$$ \xi _1 + i \varepsilon , \xi _1 \in (-\infty , \infty). $$
The integral becomes
\begin{align*}
& \int \cdots \int \left( \int \fH (\xi _1 + i r, \zeta _2 , \cdots , \zeta _n ) 
\int e^{- i \langle (i \varepsilon + \zeta _1 , \zeta _2 , \cdots \zeta _n ) , q \rangle } f (q ) 
\: d q d \xi _1 \right) d \zeta _2 \cdots d \zeta _n \\
=& \int \cdots \int \left( \int \fH (\xi _1 + i \varepsilon , \zeta _2 , \cdots , \zeta _n ) 
\int e^{-i \langle (\xi _1 , \zeta _2 \cdots , \zeta _n , q \rangle}
(e^{\varepsilon q_1} f (q )) d q d \zeta _2 \right) d \zeta _3 \cdots d \zeta _n d \xi _1 \\ 
=& \cdots = \int \fH (\xi _1 + i \varepsilon , \xi _2 + i \varepsilon , \cdots , \xi _n + i \varepsilon ) 
\int e^{- i \langle (\xi _1 , \xi _2 , \cdots \xi _n ) , q \rangle } e^{\varepsilon (q_1 + \cdots + q_n)} f (q ) 
\: d q d \xi 
\end{align*}
by using Fubini's theorem and Cauchy integral formula repeatedly.

Define the distribution
$$ \tilde \kappa _\varepsilon ( g ):= 
\int \fH (\xi _1 + i \varepsilon , \xi _2 + i \varepsilon , \cdots , \xi _n + i \varepsilon ) 
\int e^{- i \langle (\xi _1 , \xi _2 , \cdots \xi _n ) , q \rangle } g (q ) \: d q d \xi .$$
Since $\fH (\xi _1 + i \varepsilon , \xi _2 + i \varepsilon , \cdots , \xi _n + i \varepsilon )$ is a symbol for 
$\xi \in \bbR^n$ by assumption,
using Equation (\ref{PdOKerEst}) again, 
one conclude that $\tilde \kappa _\varepsilon $ is $C ^\infty $ on $\bbR^n \backslash \{ 0 \}$,
and for any natural number $N$ and multi-index $I$, there exists $C_{I, N} > 0$ such that
$$ |\partial _I \tilde \kappa _\varepsilon ( p )| 
\leq C_{I, N} (1 + | p |)^N , \quad \forall \zeta \in \bbR^n , | p | \geq 1. $$
Furthermore, by uniqueness of kernel, it follows that
$$\kappa (p) = e^{\varepsilon  (p _1 + \cdots + p _n) } \tilde \kappa _\varepsilon (p)$$
on $\bbR^n \backslash \{ 0 \}$.
Since $ p _1  + \cdots +  p _n - (- | p |) $ is bounded above on the subset $\{ p _1 , p _2 , \cdots p _n < 1 \}$,
one can write
$$ \kappa = e^{- \varepsilon | p |} \tilde F $$
for some smooth function $\tilde F$ satisfying Equation (\ref{PdOKerEst}) on the subset 
$$\{ | p | > 1 \} \bigcap \{ p _1 , p _2 , \cdots , p _n < 1 \}.$$
Repeating the arguments by considering the contours 
$$ (\xi _1 \pm i \varepsilon , \xi \pm _2 i \varepsilon , \cdots , \xi _n \pm i \varepsilon ), $$
one gets a similar estimate on each quadrant.
The assertion follows by combining these estimates. 
\end{proof}

Remark that the assumption of Proposition \ref{DecayProp} is very mild.
For example, one has
\begin{lem}
Let $P$ be a polynomial of order $n$, $P_{\mathrm {top}}$ be its highest order part.
Let $f$ be a compactly supported function on $\bbR^n$.
Suppose that $P_{\mathrm {top}} |_{\bbR^n} $ is elliptic, and $P + \fF (f) \neq 0 $ on $\bbR ^n$.
Then $P + \fF (f) \neq 0$ on some strip $S_\theta , \theta > 0$,
and $(P + \fF (f))^{-1}$ satisfies the assumption of Proposition \ref{DecayProp}
\end{lem}  

Also, we recall the following well know fact about the obstruction to existence of invertible perturbations
(see, for example, \cite{Connes;Book} for an overview of the subject):
\begin{lem} 
For any  properly supported, invariant, elliptic pseudo-differential operator 
$\varPsi _{ {}_\rT e } \in \Psi ^{[\infty]} _\varrho (\bbR ^n)$,
there exists $K \in \Psi ^{- \infty } _\varrho (\bbR ^n )$ such that 
$\varPsi _{ {} _\rT e} + K $ is invertible if and only if the $\mathbb K$-theoretic analytic index 
$$ \ind _{\mathrm {Ana}} (\varPsi _x) \in \mathbb K ^0 (C ^\infty _c (\mathbb R ^n ) , \circ) $$
vanishes. Here, $\circ $ denotes the convolution product on $C ^\infty _c (\bbR ^n )$.
\end{lem}

Finally, we end up with:
\begin{thm}
\label{FredThm}
Let $\varPsi = \{ \varPsi _x \} _{x \in \bbC \rP (2) } \in \Psi ^{[m]} _\mu (\rSU (2) \times \rN / \rT)$ be elliptic. 
Let $\tilde \sigma (\zeta )$ be the Laplace-Fourier transform of $\varPsi _{{}_\rT e}$.
Suppose $\sigma (\zeta )$ satisfies the estimate
$$ C (1 + |\zeta |)^m \geq |\tilde \sigma (\zeta )| \geq C' (1 + |\zeta |)^m $$
for some $C, C' > 0 $, on some strip $ \zeta \in S_\theta $ for some $\theta > 0 $.
Then $\varPsi $ is Fredholm.
\end{thm}
\begin{proof}
Given any $\varPsi $ as in the hypothesis.
Let $\tilde \sigma (\zeta )$ be the Laplace-Fourier transform of $\varPsi _{{}_\rT e}$.
Put $\varsigma := \tilde \sigma ^{-1}$. 
Then $\varsigma $ satisfies the hypothesis of Proposition \ref{DecayProp}.
Hence $\varPsi _{{}_\rT e}^{-1}$ has a reduced kernel of the form
$$ \psi = e^{- \varepsilon |p|} F (p) $$
on $\bbR^n \backslash \{ 0 \}$, 
where $F (p) \in C^\infty (\bbR^n \backslash \{ 0 \})$ satisfies Equation (\ref{PdOKerEst}).

We need to prove that $\varPsi _{{}_\rT e} ^{-1} \in \fU (\cG _{{}_\rT e})$.
To do so, write $\psi := \psi _\mu  + \psi _e $,
where $\psi _\mu $ is compactly supported and 
$\psi _e \in C^\infty (\bbR^n)$, $\psi _e = 0$ on a neighborhood of $0$. 
Let $\varPsi _\mu $ and $\varPsi _e$ be the corresponding pseudo-differential operators.
Then $\varPsi _\mu  \in \Psi ^{[-m]} _\mu (\cG _{{}_\rT e})$.
It remains to consider $\psi _e$.
Since $\psi _e$ decays exponentially, 
it is clear that one can find a sequence $\{ \kappa_j \}, j = 1,2, \cdots $ in $C^\infty _c (\bbR^n)$ such that 
$$ \| \psi _e - \kappa_j \|_{\bL^1 (\bbR^n)} \rightarrow 0. $$
Let $K _j \in \Psi ^{- \infty} _\mu (\cG _{\rT})$ be the corresponding invariant pseudo-differential operators.
Then by Definition \ref{OpNorm},
$$\| \varPsi _e - K _j \| _1 \rightarrow 0. $$
It follows from Definition \ref{OpNorm} of the full norm that $K _j \rightarrow \varPsi _e $.
Hence $\varPsi _e \in \fU (\cG_{{}_\rT e})$ as well.
The result follows from Lemma \ref{NisLem}. 
\end{proof}

Proposition \ref{DecayProp} and Theorem \ref{FredThm} not only give a criterion for an operator $\nu (\varPsi )$, 
where $\varPsi \in \Psi ^{[\infty]} _\mu (\cG) $ to be Fredholm, 
they also give a more precise description for the parametrix of $\nu (\varPsi )$ modulo compact operators.
\begin{thm}
\label{ParaThm}
Let $\varPsi \in \Psi ^{[m]} _\mu \cG $ be an elliptic operator satisfying the hypothesis of Theorem \ref{FredThm}.
There exists operators $Q \in \Psi ^{- [m]} _\mu (\cG)$, 
and $S \in \Gamma ^\infty (\bt ^{-1} \rE \otimes \bs ^{-1} \rE ')$ 
(regarded as a reduced kernel in $\Psi ^{- \infty } (\cG)$)
of the form 
$$S (a) = e ^{- \varepsilon \tilde d (a , \bs (a))} \tilde \kappa , \quad a \in \cG, $$
for some $\varepsilon > 0 $,
where $\tilde d $ is a smooth function on $\cG ^2$ satisfying $\tilde d - d \leq 1$,
and $\tilde \kappa \in \Gamma ^\infty _b (\bt ^{-1} \rE \otimes \bs ^{-1} \rE) $ such that 
$$\nu (\varPsi ) \nu (Q + S) - \id $$ 
is a compact operator. 
\end{thm}
\begin{proof}
By standard arguments one can find $Q \in \Psi ^{- [m]} _\mu (\cG)$ such that
$$ \varPsi Q - \id = R _1 , \quad R _1 \in \Psi ^{- \infty} _\mu (\cG). $$
On the other hand, Proposition \ref{DecayProp} implies that one has
$$(\varPsi _{{}_\rT e} )^{-1} = \tilde Q + e ^{- \varepsilon \tilde d (a , \bs (a))} F $$ 
for some $F \in \Gamma ^\infty (\bt ^{-1} \rE \otimes \bs ^{-1} \rE |_{\cG _{ {} _\rT e}} ) $ with bounded derivatives.
Since both $Q _{{}_\rT e}, \tilde Q $ are properly supported parametrices of $\varPsi _{{}_\rT e}$,
it follows that
$$ (\varPsi _{{}_\rT e} )^{-1} - Q _{{}_\rT e} = S _{{}_\rT e} $$ 
for some $S _{{}_\rT e} \in \Gamma ^\infty (\bt ^{-1} \rE \otimes \bs ^{-1} \rE |_{\bs ^{-1} ({}_\rT e)} ) $ of the form
$$ S _{{}_\rT e } (a) = e ^{- \varepsilon \tilde d (a , \bs (a))} \kappa .$$

Let $\rU \subset \rM $ be a local trivialization of $\bs $ around ${} _\rT e $,
i.e., $\rU \times \cG _{ {}_\rT e } \cong \bs ^{-1} \rU $.
Fix any function $\chi \in C^\infty _c (\rU ) $ such that $\chi = 1 $ on a smaller neighborhood of ${} _\rT e $.
Define a section $ \tilde S (a) \in \Gamma ^\infty (\bt ^{-1} \rE \otimes \bs ^{-1} \rE ) $ as follows:
If $a \in \bs ^{-1} (\rU)$, $a$ is identified with a point $(x, p) \in \rU \times \cG _{ {} _\rT e } $,
and we define
$$\tilde S (a) := S (p) \chi (x) .$$ 
Otherwise, we define $\tilde S (a) := 0 $. 
By the computations in the proof of Lemma \ref{CP1Est},
$\tilde S $ satisfies the estimate
$$ \tilde S (a) = e ^{- \varepsilon \tilde d (a , \bs (a))} \tilde \kappa ,$$
where $\kappa (a) $ and $\kappa (a ^{-1})$ are both sections of bounded derivatives. 
Moreover, it is obvious that
$\tilde S |_{\cG _{{}_\rT e)} } = S _{{}_\rT e} $.
It follows that 
$$ \tilde R := \varPsi (Q + \tilde S) - \id \in \Phi ^{- \infty } (\cG)$$
satisfies $\tilde R _{{}_\rT e} = 0$.
By Corollary \ref{CptCor} below, it follows that that $\nu (\tilde R) = \nu (\varPsi ) \nu (Q + \tilde S) - \id $
is a compact operator.
\end{proof}

\subsection{ Exponentially decaying kernels}
Inspired by the results of Proposition \ref{DecayProp} and Theorem \ref{FredThm}, 
we construct the pseudo-differential calculus with bounds,
in parallel with the theory of poly-homogeneous distributions for manifolds with corners 
(see \cite[Chapter 5]{Melrose;Book}).

First, let $\cG \rightrightarrows \rM$ be a Lauter-Nistor groupoid.
We say that
\begin{dfn}
{
The groupoid $\cG$ is of {\it sub-exponential growth } if for any $\varepsilon > 0$,
$$ \int _{a \in \bs ^{-1} (x)} e ^{- \varepsilon d (x , a)} \mu _x (a) \leq C $$
for some constant $C$ independent of $x \in \rM$;
it is of {\it polynomial growth} if for some integer $N $ and constant $C$,
$$ \int _{a \in B (x, r) } \mu _x (a) \leq C r ^N .$$
}
\end{dfn}
Clearly, polynomial growth implies sub-exponential growth.

\begin{exam}
{
Since each $\bs$-fiber of the symplectic groupoid over the Bruhat sphere is quasi-isometric to the 
Euclidean space $\bbR^2$,
the groupoid $\rT \backslash (\rSU (2) \times \rN )$ is of polynomial growth.
}
\end{exam}

Recall that in Section 2, given a Hausdorff groupoid $\cG$, 
we defined the groupoid $ \tilde \cG := \{ (a , b ) \in \cG \times \cG : \bs (a ) = \bs (b) \} $.
Also recall that for any $X \in \Gamma ^\infty (\cA )$,
$X$ determines a right invariant vector field $X ^\cR \in \Gamma ^\infty (\Ker (d \bs ))$.
Here, we furthermore define vector fields on $\tilde \cG $ by
\begin{align*}
X ^{\tilde \cR } (a , b) :=& ( X ^\cR (a ) , 0 ) \in T _a \cG \times T _b (\cG ) \subseteq \tilde T \cG \\
X ^{\tilde \cL } (a , b) :=& ( 0 , X ^\cR (b ) ),
\end{align*}
for any $( a , b ) \in \tilde \cG $. 
Similarly, given any vector bundle $\rE \to \rM$,
and $\cA $-connection ${}^\cA \nabla ^\rE $ on $\rE$,
right translation defines a connection $\hat \nabla ^{\bt ^{-1} \rE }$ on $\bt ^{-1} \rE \to \cG _x$,
for each $x \in \rM$.
We shall consider the family of pullback connections
$\hat \nabla ^{\tilde \bs ^{-1} (\bt ^{-1} \rE ) \otimes \tilde \bt ^{-1} (\bt ^{-1} \rE ' ) }$
on $ \tilde \bs ^{-1} (\bt ^{-1} \rE ) \otimes \tilde \bt ^{-1} (\bt ^{-1} \rE ' ) 
\to \cG _x \times \cG _x \subseteq \tilde \cG $.

Fix a Riemannian metric $g _\cA $ on $\cA $, 
which in turn determines a metric on each of $\cG _x$.
For each $( a , b ) \in \tilde \cG $, 
define $d (a , b) $ to be the Riemannian distance on $\cG _{\bs (a)} = \cG _{\bs (b)}$ between $a$ and $b$.

\begin{dfn}
For each $\varepsilon > 0$, the $\varepsilon $-calculus of order $- \infty , \Psi ^{- \infty} _\varepsilon (\cG , \rE)$, 
is defined to be the space of sections
$\psi \in \Gamma ^\infty (\bt ^{-1} \rE \otimes \bs ^{-1} \rE ' )$, regarded as reduced kernels, 
with the property that there exists some $\varepsilon ' > \varepsilon $
such that for all $(a , b ) \in \tilde \cG , m = 0, 1, 2 \cdots , $ 
$$ e ^{\varepsilon ' d (a , b) }
(\hat \nabla ^{ \tilde \bs ^{-1} (\bt ^{-1} \rE ) \otimes \tilde \bt ^{-1} (\bt ^{-1} \rE ' )} )^k
(\widetilde \bm ^{-1} \psi )(a , b) \leq C _k $$
for some constants $C _k > 0 $.

For each $m \in \bbZ , \varepsilon  > 0$, 
the {\it (classical) $\varepsilon $-calculus of order $m$} is defined to be the space
$$ \Psi ^{[m]} _\varepsilon (\cG , \rE)
:= \Psi ^{[m]} _\mu (\cG , \rE) + \Psi ^{- \infty } _\varepsilon (\cG , \rE). $$
\end{dfn}

As in the case of manifolds with boundary \cite{Mazzeo;EdgeRev, Melrose;Book}, 
we need to compute the composition rule of the calculus.
\begin{lem}
For any $\varepsilon _1 , \varepsilon _2 \geq 0 $
$$ \Psi ^{- \infty } _{\varepsilon _1} \circ \Psi ^{- \infty } _{\varepsilon _2} 
\subseteq \Psi ^{- \infty } _{\min \{ \varepsilon _1 , \varepsilon _2 \}}.$$ 
\end{lem}
\begin{proof}
For simplicity we only consider the scalar case.
It suffices to consider the convolution product
$ u _1 \circ u _2 $
for any $u _1 \in \Psi ^{- \infty } _{\varepsilon _1} (\cG ), 
u_2 \in \Psi ^{- \infty } _{\varepsilon _2} (\cG )$.
In view of the formula
$$ u _1 \circ u _2 (a) = \int _{b \in \cG _{\bs (a)}} u _1 (a b ^{-1}) u _2 (b) \mu _{\bs (a)} (b)
= \int _{c \in \bs ^{-1} (\bt (a))} u _1 (c ^{-1}) u _2 (c a) \mu _{\bt (a)} (c),$$
one can without loss of generality assume $\varepsilon _1 \leq \varepsilon _2$.
Then by definition one has the estimates 
$ u _1 (a) \leq e ^{- \varepsilon ' _1 d (a, \bs (a))} C ,
u _2 (a) \leq e ^{- \varepsilon ' _2 d (a, \bs (a))} C' $
for some $\varepsilon '_1 > \varepsilon _1 , \varepsilon ' _2 > \varepsilon _2 $.
One may further assume that $\varepsilon ' _1 < \varepsilon ' _2 $.

The hypothesis implies for any $a \in \cG$
\begin{align*}
| u _1 \circ u _2 (a) |
\leq & C _1 \int _{b \in \cG _{\bs (a)}} 
e ^{- \varepsilon ' _1 d (a , b) } e ^{ - \varepsilon ' _2 d (b , \bs (b))} \mu _{\bs (a)} (b) \\
\leq & C _1 \int _{b \in \cG _{\bs (a)}} 
e ^{- \varepsilon ' _1  |d (a , \bs (a)) - d (b , \bs (b))| - \varepsilon ' _2 d (b , \bs (b))} \mu _{\bs (a)} (b) \\
= & C _1 \int _{b \in B _a } 
e ^{- \varepsilon ' _1  d (a , \bs (a)) } e ^{- (\varepsilon ' _2 - \varepsilon ' _1) d (b , \bs (b))} 
\mu _{\bs (a)} (b) \\ 
&+ C _1 \int _{b \not \in B _a} 
e ^{ \varepsilon ' _1  d (a , \bs (a)) } e ^{- (\varepsilon ' _2 + \varepsilon ' _1) d (b , \bs (b))} \mu _{\bs (a)} (b),
\end{align*}
where $B _a$ denotes the set $\{ b \in \cG _{\bs (a)} : d (b , \bs (b)) < d (a , \bs (a)) \} $
for each $a$.
Hence for the first integral, one has
\begin{align*}
\int _{b \in B _a } 
e ^{- \varepsilon ' _1  d (a , \bs (a)) } & e ^{- (\varepsilon ' _2 - \varepsilon ' _1) d (b , \bs (b))} 
\mu _{\bs (a)} (b)e ^{- \varepsilon ' _1  d (a , \bs (a)) } \\
= & \: e ^{- \varepsilon ' _1  d (a , \bs (a)) } 
\int _{b \in B _a } e ^{- (\varepsilon ' _2 - \varepsilon ' _1) d (b , \bs (b))} \mu _{\bs (a)} (b) \\
\leq & \: e ^{- \varepsilon ' _1  d (a , \bs (a)) } 
\int _{b \in \cG _{\bs (a)}} e ^{- (\varepsilon ' _2 - \varepsilon ' _1) d (b , \bs (b))} \mu _{\bs (a)} (b),
\end{align*}
and the last integral is finite and only depends on $\bs (a)$.
As for the second integral, write
\begin{align*}
\varepsilon ' _1  d (a , \bs (a)) -  &(\varepsilon ' _2 + \varepsilon ' _1) d (b , \bs (b)) \\
= &- \varepsilon ' _1  d (a , \bs (a)) 
+ 2 \varepsilon ' _1 ( d (a , \bs (a)) - d (b, \bs (b)))
- (\varepsilon ' _2 - \varepsilon ' _1) d (b , \bs (b)).
\end{align*}
Since $d (b, \bs (b)) \geq d (a , \bs (a)) $ for any $b \not \in B _a$.
It follows that the second integral is again bounded by
$$ e ^{- \varepsilon ' _1 d (a , \bs (a))}
\int _{b \in \cG _{\bs (a)}} e ^{- (\varepsilon ' _2 - \varepsilon ' _1) d (b , \bs (b))} \mu _{\bs (a)} (b).$$
Adding the two together and rearranging, one gets
$ e ^{\varepsilon ' _1 d _\bs (a)} (u _1 \circ u _2 ) (a) $ is a bounded function, as asserted.

To prove the assertion for derivatives,
observe that by right invariance of $\mu$,
$$ \widetilde \bm ^* (u _1 \circ u _2 )(a, b) = \int u _1 (a c ^{-1} ) u _2 (c b ^{-1}) \mu _{\bs (a)} (c) ,$$
for any $(a, b) \in \tilde \cG $.
It follows that for any $X , Y \in \Gamma ^\infty (\cA)$,
\begin{align*}
\fL _{X ^{\tilde \cR}} \widetilde \bm ^* (u _1 \circ u _2 )(a, b) 
=& \int \fL _{X ^{\tilde \cR}} (\widetilde \bm ^* u _1 )(a, c ) (\widetilde \bm ^* u _2 )(c, b) \mu _{\bs (a)} (c) \\
\fL _{Y ^{\tilde \cL}} \widetilde \bm ^* (u _1 \circ u _2 )(a, b) 
=& \int (\widetilde \bm ^* u _1 )(a, c ) \fL _{X ^{\tilde \cR}} (\widetilde \bm ^* u _2 )(c, b) \mu _{\bs (a)} (c)
\end{align*}
(here, note that $\fL _{X ^{\tilde \cR}} \bm ^* u _1 (a, c)$ only differentiates in the $a $-direction),
and so on for higher derivatives.
\end{proof}

Note that the vector representation $\nu $ is well defined on $\Psi ^{ -\infty } _\varepsilon (\cG , \rE)$
because of the sub-exponential growth assumption.

To apply the results of Section \ref{LauNis},
we first verify that
\begin{prop}
\label{C*SubsetProp}
For any $\varepsilon > 0$,
$$ \Psi ^{- \infty } _\varepsilon \subseteq \fC ^* (\cG , \rE),$$
where $\fC ^* (\cG , \rE )$ is defined in Definition \ref{OpNorm}.
\end{prop}
\begin{proof}
Let $\{\phi _n \} \in C ^\infty (\bbR)$ be a series such that $0 \leq \phi _n \leq 1$,
$\phi _n = 1 $ on $[0 , n) $, and $\phi _n = 0 $ on $[n +1 , \infty)$,
and define $\chi _n (a) := \phi _n (d _\bs (a)) \in C ^\infty _c (\cG)$.
Given any $\kappa \in \Psi ^{- \infty } _\varepsilon (\cG , \rE)$,
Write $\kappa (a) = e ^{- \varepsilon d _\bs (a)} u (a)$, 
where $u (a)$ is bounded.

Consider $\kappa _n := \chi _n \kappa \in \Gamma ^\infty _c (\cG , \rE ) \cong \Psi ^\infty _\mu (\cG , \rE)$.
For any $x \in \rM , n \in \bbN$,
one has
\begin{align*}
\int _{a \in \bs ^{-1} (x)} | \kappa - \kappa _n | (a) \mu _x (a)
=& \: \int _{a \in \cG _x \backslash B (\rM , n) } e ^{- \varepsilon d _\bs (a)} 
(1 - \chi _n ) | u | (a) \mu _x (a) \\
\leq & \: e ^{- \frac{\varepsilon n }{2}} 
\int _{a \in \cG _x \backslash B (\rM , n) } e ^{- \frac{\varepsilon}{2} d _\bs (a)} 
(1 - \chi _n ) | u | (a) \mu _x (a), 
\end{align*}
where $B (\rM , n) := \{ a \in \cG : d (a , \bs (a)) < n \}$.
By the sub-exponential growth assumption, the integral is bounded by some constant $C$,
independent of $x$.
It follows that $\sup _{x \in \rM} \| \kappa - \kappa _n \| _{\bL ^1 (\bs ^{-1} (x))} \to 0 $
as $n \to \infty $.
Also, observe that $\kappa (a ^{-1}) = e ^{- \varepsilon d _\bs (a)} u (a ^{-1})$ 
(since $d _{\bs } (a ^{-1}) = d _\bs (a)$).
Applying exactly the same arguments to $\kappa (a ^{-1})$
one arrives at
$$ \| \kappa - \kappa _n \|_1 \to 0. $$
Hence $\kappa \in \fC ^* (\cG , \rE)$.
\end{proof}

Combining the above Proposition \ref{C*SubsetProp} with Lemma \ref{ExactC^*},
one has:
\begin{cor}
\label{CptCor}
For any $\varepsilon > 0$, $\varPsi \in \Psi ^{- \infty} _\varepsilon (\cG , \rE)$
such that $\varPsi |_{\rZ_j} = 0$, 
for any invariant sub-manifolds $\rZ _j$,
then $\nu  _0 ( \varPsi ) : \bL ^2 (\rM _0) \to \bL ^2 (\rM _0)$ is a compact operator.
\end{cor}

As a simple application of the calculus with bound, we can rewrite Theorem \ref{ParaThm} as
\begin{cor}
For any $\varPsi \in \Psi ^{[m]} _\mu (\cG , \rE)$ satisfying the hypothesis of Theorem \ref{ParaThm},
there exist $\tilde Q \in \Psi ^{[- m]} _\varepsilon (\cG , \rE)$ such that
$$ \nu (\varPsi) \tilde Q - \id \in \Psi ^{[- \infty]} _\varepsilon (\cG , \rE) $$
is compact.
\end{cor}
\vspace{18cm}
 $ \; $

\pagebreak \thispagestyle{firstpage}
\section{The heat calculus}
\subsection{ The heat kernel of perturbed Laplacian operators}
In this section, we construct the heat kernel of some second order pseudo-differential operators on a groupoid 
$\cG \rightrightarrows \rM$ with $\rM$ compact.

Given a vector bundle $\rE$ over $\rM$, fix an $\cA$-connection  $\nabla ^\rE$ on $\rE$.
Then the pull-back defines a (family of $\bs$-fiberwise) connection on 
the bundle $\bs ^{-1} \rE \rightarrow \bs ^{-1} (x), x \in \rM$,
which we shall still denote by $\nabla ^\rE$.
Also, recall that we fixed a metric on $\cA$, hence a Riemannian metric on the fibers $\bs ^{-1} (x), x \in \rM$,
which we shall still denote by $g _\cA$.
We define the Laplacian by taking the trace of the square of $\nabla ^\rE$. More precisely:
\begin{dfn}
\label{LapDfn}
The {\it Laplacian} $\Delta ^\rE \in \Psi ^2 _\mu (\cG)$ is the family of operators $\{ \Delta ^\rE _x \}_{x \in \rM}, $
where
$$ \Delta ^\rE _x := \sum _{i=1}^n (\nabla ^\rE _{X_i } \nabla ^\rE _{X_i } - \nabla ^\rE _{\nabla ^\rE _{X_i} X_i}), $$
and $X_i$ is any local orthonormal basis of $T \cG _x $.
\end{dfn}
Note that $\Delta ^\rE $ is elliptic, and its principal symbol does not depend on the chosen connection $\nabla ^\rE$.

We consider an operator of the form 
\begin{equation}
\label{GenLap}
\Delta ^\rE + F + K,
\end{equation}
where $F \in \Gamma ^\infty (\bt ^{-1} \rE \otimes \bt ^{-1} \rE)$, considered as a differential operator of order 0;
and $K \in \Psi ^{- \infty }_\mu (\cG , \rE)$. 
We shall denote the reduced kernel of $K$ by $\kappa $.

Since the restriction of $\bt ^{-1} \rE$ to each $\bs$-fiber $ \cG _x $ is a vector bundle with bounded geometry,
we have the Sobolev norms
$ \| \cdot \| _{\infty , l } $ defined by Equation (\ref{L-Infty}). 
For $u \in \Gamma ^\infty (\bt ^{-1} \rE)$, we define 
$$ \| u \| _l := \sup_{x \in \rM} \{ \| u |_{\bs ^{-1} (x)} \|_{\infty , l} \}. $$

Denote by $\bt ^{-1} \rE \otimes \bs ^{-1} \rE ' \ltimes (0, \infty ) $ 
the pullback of $\bt ^{-1} \rE \otimes \bs ^{-1} \rE ' \to \cG$
by the projection $\cG \times (0, \infty) \to \cG$. 
\begin{dfn}
{
A {\it (groupoid) Heat kernel} of $\Delta ^\rE + F + K $ is a continuous section 
$$ Q \in \Gamma ^0 (\bt ^{-1} \rE \otimes \bs ^{-1} \rE' \ltimes (0 , \infty ) ) ,$$
such that $ Q (a , t) , Q (a ^{-1} , t) $ are smooth when restricted to all $\cG _x \times (0 , \infty ) $,
and satisfies: 
\begin{enumerate}
\item
The heat equation
$$ (\partial _t + \Delta ^\rE + F + K) Q (a, t) = 0. $$
Here, we use the fact that $\bs ^{-1} \rE ' |_{\cG _x} \cong \cG _x \times \rE '_x $,
and let $\Delta ^\rE + F + K$ to act on the $\bt ^{-1} \rE $ factor of
$ Q (a, t) |_{\cG _x} \in \Gamma ^\infty (\bt ^{-1} \rE \otimes \bs ^{-1} \rE ' )
\cong \rE ' _x \times \Gamma ^\infty (\bt ^{-1} \rE  )$ for each $t$ fixed;
\item
The initial condition
$$ \lim _{t \rightarrow 0^+} Q \circ u = u, 
\quad \forall u \in \Gamma ^\infty _c (\bt ^{-1} \rE), $$
where $\circ $ denotes the convolution product.
\end{enumerate}
}
\end{dfn}

Let $Q $ be a groupoid heat kernel. 
Then it is clear that for any $x \in \rM $, 
$( a , b ) \in \cG _x \times \cG _x \mapsto Q (a b ^{-1} , t) $ is a heat kernel of 
$( \Delta ^\rE + F + K )_x $ on the manifold with bounded geometry $\cG _x$.
Using the uniqueness of the heat kernel on manifolds with bounded geometry, 
it is clear that:
\begin{lem}
A groupoid heat kernel $Q $ of $\Delta ^\rE + F + K $, if it exists, is unique.
\end{lem}

\subsubsection{\bf The formal solution}
Before we start, we need to define some notation.

Recall that there exists $r_0 > 0$ such that $\exp ^\nabla $ is a diffeomorphism from
the set 
$$ \cA_{r_0} := \{ X \in \cA :g _\cA (X, X) < r_0 ^2 \} $$
onto its image.
For each $x \in \rM$, we denote the polar coordinates on $\cA _x$, the fiber of $\cA$ over $x$, by $(r , \vartheta )$.
The image of $\cA _{r_0} $ under $\exp ^\nabla$ is denoted by $B (\rM , r_0 )$.
Note that since
$$ d (\exp ^\nabla (r, \vartheta ) , x ) = r, $$
therefore $B (\rM , r_0 ) = \{ a \in \cG : d (a , \bs (a)) < r_0 \}$, as expected.
The exponential map also defines a local trivialization of $\bt ^{-1} \rE$:
For each $a = \exp ^\nabla X \in B (\rM , r_0), E \in \rE _{\bs (a)}$ where $X \in \cA _{\bs (a)}$, define
$ T (a) (E) \in \bt ^{-1} \rE _a $ to be the parallel transport of $ E $ to $ a $ along the curve 
$ \exp ^\nabla \tau X , \tau \in [0, 1] $.
Hence $T$ is a map from the set $\{ (a , E) : a \in B (\rM , r_0), E \in \rE_{\bs (a)} \} $ to 
$\bt ^{-1} \rE |_{B (\rM , r_0)}$,
and we denote its inverse map by $T^{-1} $.
When restricted to $\bt ^{-1} \rE |_{\cG _x} $ for some $x \in \rM$,
the image of $T ^{-1}$, lies in $E _x$. 
In that case we shall still denote the restricted map by
$T ^{-1} : \bt ^{-1} \rE |_{\cG _x \bigcap B (\rM , r_0)} \rightarrow \rE _x $.

Lastly, we let $J := \det (d \exp ^\nabla) \circ (\exp ^\nabla)^{-1} $ to be the Jacobian, 
and $V := d (a , \bs (a)) \times d \exp^\nabla (\partial _r )$ be the radial vector field.

Consider a kernel of the form
$$ q (a , t) \Phi (a, t) \in \Gamma ^\infty (\bt ^{-1} \rE \otimes \bs ^{-1} \rE' \ltimes (0, \infty)), $$
where $q : B (\rM , r_0 ) \times (0 , \infty ) \rightarrow \bbR $ is the Gaussian function
$$ q (a , t) := (4 \pi t)^{- \frac{n}{2}} \: e^{- \frac{d (a , \bs (a))^2 }{4 t}}. $$
A straightforward calculation shows that:
\begin{lem}
One has
$$ (\partial _t + \Delta ^\rE + F) q (a , t) \Phi (a, t)
= q (a , t) (\partial _t + \Delta ^\rE + F + t ^{-1} \nabla ^\rE _V + \frac{\fL _V J}{2 t J} ) \Phi (a, t). $$
\end{lem}
 
\begin{lem}
\label{HeatPowerSeries}
There exists a formal power series 
$$ \Phi (a, t) = \sum _{i=1} ^\infty t^i \Phi _i (a), \quad \Phi _i \in \Gamma ^\infty $$
satisfying the equation
\begin{equation}
(\partial _t + \Delta ^\rE + F + t ^{-1} \nabla ^\rE _V + \frac{\fL _V J}{2 t J} ) \Phi (a, t) = 0.
\end{equation}
\end{lem}
\begin{proof}
Equating coefficients one gets
\begin{align*}
\nabla ^\rE _V ( J ^\frac{1}{2} \Phi _0 ) = & \: 0 \\
\nabla ^\rE _V ( J ^\frac{1}{2} \Phi _i ) + i \Phi _i =& - (\partial _t + \Delta ^\rE + F ) \Phi _{i-1} ,
\quad i = 1, 2 \cdots 
\end{align*}
These are simple ordinary differential equations, with explicit solutions
\begin{align*}
\Phi _0 (\exp X) =& J ^{- \frac{1}{2}} T (\exp X) \\
\Phi _i (\exp X) =& 
- J ^{ - \frac{1}{2}} T \int _0 ^1 J ^{\frac{1}{2}} T^{-1} ((\partial _t + \Delta ^\rE + F ) 
\Phi _{i-1} (\exp \tau X )) \tau ^{i-1} \: d \tau .
\end{align*}
\end{proof} 

Fix a cutoff function $\chi $ supported on $B (\rM , r_0)$ such that $\chi = 1 $ on the smaller set
$B (\rM , \frac{r_0}{2}) := \{ a \in \cG : d (a , \bs (a)) \leq \frac{r_0}{2} \}$.
Write 
$$ G _N (a, t) := \chi (a) q (a, t) \sum _{i=1}^N t^i \Phi _i (a), 
\quad t \in (0, \infty ).$$
Then one has
\begin{lem}
\label{ParaProof}
For any $N > \frac{n}{2}$, 
\begin{enumerate}
\item
For any $k, l \in \bbN$, there exists a constant $C_{k, l}$ such that 
$$\| \partial _t ^k ((\partial _t + \Delta ^\rE + F) G _N) \|_l \leq 
C_{k, l} t ^{N - \frac{n}{2} - k - \frac{l}{2}} ;$$ 
\item
For any $t_0 > 0$, the map
$$u \mapsto G _N (\cdot , t) \circ u, 0 \leq t \leq t_0 $$ 
is a uniformly bounded family of operators on 
$\Gamma ^l (\bt ^{-1} \rE)$, and for any $u \in \Gamma ^l (\bt ^{-1} \rE)$,
$$ \lim _{t \rightarrow 0^+} \| G _N \circ u - u \| _l = 0. $$
\end{enumerate}
\end{lem}
\begin{proof}
On $B (\rM , \frac{r_0}{2})$,
from the proof of Lemma \ref{HeatPowerSeries}, one has 
\begin{align*}
(\partial _t + \Delta ^\rE + F) G _N (a) =& \: t ^N q (a, t) \Phi _N (a) \\
=& \: (4 \pi )^{- \frac{n}{2}} t ^{N - \frac{n}{2}} e ^{- \frac{d (a , \bs (a))^2}{4 t} } \Phi _N (a).
\end{align*}
It is elementary that $ e ^{- \frac{d (a , \bs (a))^2}{4 t} }$ is bounded for any $a, t$,
and $\Phi _N $ is smooth and hence has bounded derivatives.
On $\cG \backslash B (\rM , \frac{r_0}{2}) $ observe that 
$e ^{- \frac {d (a , \bs (a))}{4 t}}$ and all its derivatives decay faster than any powers as $t \rightarrow 0$.
That proves (1) in the case $l = k = 0$. 
Other cases follow from a similar argument, with the additional observation that 
\begin{align*}
\partial _t e^{- \frac {y^2 }{t}} =& - t ^{-1} (\frac{y ^2}{t} ) e ^{ - \frac{y ^2}{t}} = O (t^{-1}) \\
\partial _y e^{- \frac {y^2 }{t}} =& - t ^{-\frac{1}{2}} (\frac{y ^2}{t} )^{\frac{1}{2}} e ^{ - \frac{y ^2}{t}} 
= O (t^{- \frac{1}{2}}) .
\end{align*}

To prove (2), write for any $a \in \cG$,
\begin{align*}
G _N \circ u (a) := & \: \int _{\cG _{\bs (a)}} G _N (a b^{-1} ) u (b ) \mu _{\bs (a)} (b) \\ 
= & \: \int _{\bs ^{-1} (\bt (a))} G _N (c ^{-1}) u ( c a ) \mu _{\bt (a)} (c) 
\quad \text {(using the right invariance of $\mu$)} \\
= & \: \int _{\bt (a)} 
(4 \pi t)^{- \frac{n}{2}} e^{\frac{d (c ^{-1} , \bt (c))^2 }{4 t}} 
\chi (c ^{-1}) \Big( \sum _{i=0}^N t ^i \Phi _i (c^{-1}) \Big) u (c a) \mu _{\bt (a)} (c).
\end{align*}
By right invariance and symmetry of the distance function $d (\cdot , \cdot)$,
one has $d (c ^{-1} , \bt (c)) \\ = d (c , \bs (c)) $.
Hence $\chi ( c^ {-1}) = \chi (c) $, 
and $ e^{\frac{d (c ^{-1} , \bt (c))^2 }{4 t}} = e^{\frac{d (c , \bs (c))^2 }{4 t}}$.
Therefore the integrand is supported on $B (\rM , r _0)$ and the integral can be computed by a change of variable
$c = \exp ^\nabla X , X \in \cA _{\bt (a)},  g_\cA (X, X) \leq r_0 ^2 $:
\begin{align*}
\int _{\bs ^{-1} (\bt (a))} 
(4 \pi t)^{- \frac{n}{2}} & e^{- \frac{d (c ^{-1} , \bt (c))^2 }{4 t}} 
\chi (c ^{-1}) \Big( \sum _{i=0}^N t ^i \Phi _i (c^{-1}) \Big) u (c a) \mu _{\bt (a)} (c) \\
=& \int _{\bs ^{-1} (\bt (a))} 
(4 \pi t)^{- \frac{n}{2}} e^{- \frac{d (c , \bs (c))^2 }{4 t}} 
\chi (c) \Big( \sum _{i=0}^N t ^i \Phi _i (c^{-1}) \Big) u (c a) \mu _{\bt (a)} (c) \\
=& \int _{X \in \cA _{\bt (a)}} (4 \pi t)^{- \frac{n}{2}} e^{- \frac{g _\cA (X, X) }{4 t}} \chi (\exp ^\nabla X) \\
& \times \Big( \sum _{i=0}^N t ^i \Phi _i ((\exp ^\nabla X)^{-1}) \Big)
u ((\exp ^\nabla X) a) \det (d \exp ^\nabla )(X) \: d X.
\end{align*}
It is clear that the last expression converges to 
$$ \chi (\exp ^\nabla 0) (\sum _{i=0}^N t ^i \Phi _i ((\exp ^\nabla 0)^{-1})) 
u ((\exp ^\nabla 0) a) (\det (d \exp ^\nabla )(0)) = u (a), $$
since $(4 \pi t)^{- \frac{n}{2}} e^{- \frac {g _\cA (X, X) }{4 t}}$ is just the Gaussian heat kernel on the
usual Euclidean space.  
\end{proof}

\subsubsection{\bf From parametrix to heat kernel}
\label{LeviSeries}
In the last section we constructed an approximate solution to the heat kernel.
In this section we use the method of Levi parametrix to construct a heat kernel.
We turn to operators of the form
$$  \Delta ^\rE + F + K. $$

For each $N > \frac{n}{2}$, define the sections 
$ R ^{(k)} _n \in \Gamma ^\infty (\bt ^{-1} \rE \otimes \bs ^{-1} \rE ' _ {[0, \infty )})$:
\begin{align*}
R ^{(1)} _N :=& \: (\partial _t + \Delta ^\rE + F + K ) G _N \\
R ^{(k)} _N :=& \: \int _0 ^t R _N (\cdot , t - \tau ) \circ R ^{(k-1)} _N (\cdot , \tau ) d \tau \\
=& \: \int _0 ^t \int _{\bs ^{-1} (a)} R_n (a b^{-1} , t - \tau ) R^{(k-1)} _N (b , \tau ) \mu _{\bs (a)} (b) d \tau  \\
Q ^{(0)} _N :=& \: G _N \\ 
Q ^{(k)} _N :=& \: \int _0 ^t G _N (\cdot , t - \tau ) \circ R ^{(k)} _N (\cdot , \tau ) d \tau, \quad k \geq 1 .
\end{align*}
Then one has the estimates
\begin{lem}
\label{RGrowth}
Let $N > \frac{n + l}{2} $. There exists constants $\tilde C _l, l \in \bbN$ such that 
$$ \| R ^{(k)} (\cdot , t) \| _l 
\leq \tilde C _l \tilde C_0 ^k M ^k (1 + t ^{N - \frac{n+l}{2}} )^k t ^{k-1} ((k-1)!)^{-1}.$$
\end{lem}
\begin{proof}
Using the same arguments as in the proof of (2) Lemma \ref{ParaProof}, 
one has $K G _N = \kappa (\cdot ) \circ G _N (\cdot , t) \rightarrow \kappa $ in the $\| \cdot \| _l $-norm
as $t \rightarrow 0$.
Therefore $ K G _N $ is a continuous section over $ \cG \times [0, \infty)$,
and its $l$-partial derivatives extends continuously to $t \in [0, \infty)$.  
Combining with (1) of Lemma \ref{ParaProof}, 
it follows that the integrand is a continuous section on $\cG \times [0, t] $, 
so the integral exists (and is finite).

Combining (1) of Lemma \ref{ParaProof} and the boundedness of $K$ to obtain for each $l$,
$$\| R ^{(1)} _N (\cdot , t ) \| _l 
= \| (\partial _t + \Delta ^\rE + F + K ) G _N (\cdot , t) \|_l
\leq \tilde C _l (1 + t ^{N - \frac{n+l}{2}} ) $$ 
for some $\tilde C_l > 0$.
Expand $R ^{(k)} $ as a multiple integral:
\begin{align*}
R ^{(k)} _N (a , t) =
\int _{0 \leq t_{k-1} \leq \cdots \leq t _1 \leq t }
& \int _{ b_1 , b_2, \cdots b_{k -1} \in \bs ^{-1} (a) } \\
R _N ^{(1)} (a b _1 ^{-1} & , t - t_1 ) R _N ^{(1)} (b_1 b_2 ^{-1} , t_1 - t_2) \cdots \\
\times R _N ^{(1)} & (b _{k-1} b _k ^{-1}, t _{k-2} - t _{k-1} ) 
R _N ^{(1)} (b _{k-1} , t_{k-1} ) \mu (b _1) \cdots \mu (b _{k-1}).
\end{align*}

Next, consider the domain of integration.
Since both $G _N $ and $\kappa $ have compact supports, 
$R _N ^{(1)}$ is compactly supported for each $t \geq 0$.
In particular, there exists $\rho > 0$ such that 
$R _N ^{(1)} (c _1 c _2 ^{-1} , t ) = 0  $ for any $c _1 , c _2  \in \cG $ such that 
$\bs (c _1) = \bs (c _2) $ and $d (c _1 , c _2) \geq \rho $.
Using the bounded geometry property of the $\bs$-fibers, we take 
$$M := \sup _{c \in \cG} \int _{B (a, \rho )} \mu _{\bs (c)} < \infty.$$
Then it follows that the volume of the domain of integration is bounded by 
$$ \int _{b _1 \in B (a , \rho )} \int _{b _2 \in B (b_1 , \rho )} \cdots \int _{b _{k-1} \in B (b _{k-2} , \rho )}
\mu _{\bs (a)} (b _1) \cdots \mu _{\bs (a)} (b _{k-1} ) \leq M ^{k-1}. $$
By elementary calculation, one also gets
$$ \int _{0 \leq t_k \leq \cdots \leq t _1 \leq t } d t _1 d t _2 \cdots d t _{k - 1} = t ^{k-1} ((k - 1)! )^{-1}. $$
Finally, one has for any $a \in \cG$,
\begin{align*}
| R ^{(k)} _N (a , t) |_l \leq
\int _{0 \leq t_{k-1} \leq \cdots \leq t _1 \leq t }
& \int _{ b_1 , b_2, \cdots b_{k -1} \in \bs ^{-1} (a) }
| R _N ^{(1)} (a b _1 ^{-1} , t - t_1 ) |_l \\
\times & | R _N ^{(1)} (b_1 b_2 ^{-1} , t_1 - t_2)| _0 \cdots 
|R _N ^{(1)} (b _{k-1} b _k ^{-1} , t _{k-2} - t _{k-1} ) |_0 \\
& \times | R _N ^{(1)} (b _{k-1} , t_{k-1} ) |_0 \mu _{\bs (a)} (b _1) \cdots \mu _{\bs (a)} (b _{k-1}) \\
\leq \int _{0 \leq t_{k-1} \leq \cdots \leq t _1 \leq t }
& \int _{b _1 \in B (a , \rho )} \int _{b _2 \in B (b_1 , \rho )} \cdots \int _{b _{k-1} \in B (b _{k-2} , \rho )} \\
& \tilde C _l \tilde C _0 ^{k-1} (1 + t ^{N - \frac{n+l}{2}} )^k 
\mu _{\bs (a)} (b _1) \cdots \mu _{\bs (a)} (b _{k-1} ) \\
\leq \tilde C _l \tilde C_0 ^k M ^k t ^{k-1} (1 & + t ^{N - \frac{n+l}{2}} )^k ((k-1)!)^{-1}.
\end{align*}
The assertion follows by taking supremum over $a \in \cG$.
\end{proof}

\begin{lem}
\label{QEst}
Assume that $l > 1, 2N > n + l$.
\begin{enumerate}
\item
There exists constants $C' _l$ such that
$$ \| Q _N ^{(k)} (\cdot , t) \|_l \leq C' _l \tilde C_0 ^k M ^k (1 + t ^{N - \frac{n+l}{2}} )^k t ^k (k!)^{-1}; $$   
\item
The kernel $Q _N ^{(k)} (a, t)$ is continuously differentiable with respect to $t$ and 
$$ (\partial _t + \Delta ^\rE + F + K) Q ^{(k)} _N = R ^{(k+1)} _N + R ^{(k)} _N.$$
\end{enumerate}
\end{lem}
\begin{proof}
Define the section
$$ B (a, t, s) := (G _N (\cdot , t - s ) \circ R ^{(k)} _N (\cdot , s)) (a),
\quad \forall a \in \cG, t \in [0, \infty ), s \in [0, t].$$
Since $G _N (\cdot , t - s) $ is $C^l$ by our construction,
by (2) of Lemma \ref{ParaProof}, one has for any $0 \leq s \leq t$,
\begin{align*}
\| b (\cdot , t , s) \| _l \leq & \: C' \int _0 ^t \| R ^{(k)} _N (\cdot , s) \| _l d s \\
\leq & \: C' \tilde C_0 ^k M ^k (1 + t ^{N - \frac{n+l}{2}} )^k \int _0 ^t s ^{k-1} ((k-1)!)^{-1} d s \\
\leq & \: C' \tilde C_0 ^k M ^k (1 + t ^{N - \frac{n+l}{2}} )^k t ^{k} (k!)^{-1},
\end{align*}
from which (1) follows. 
To prove (2), one has
\begin{align*}
(\partial _t + \Delta ^\rE + F + K) & (\int _0 ^t B (a, t, s ) d s ) (a, t) \\
=& \: B (a, t, t) 
+ \int _0 ^t (\partial _t + \Delta ^\rE + F + K) G _N (\cdot , t - s) \circ R ^{(k)} _N (\cdot , s) d s \\
=& \: R ^{(k)} _N (a, t)
+ \int _0 ^t R ^{(0)} (\cdot , t - s) \circ R ^{(k)} _N (\cdot , s) d s \\
=& \: R ^{(k)} _N (a, t) + R ^{(k+1)} _N (a, t).
\end{align*}
\end{proof} 
Finally, we can construct the heat kernel
\begin{lem}
\label{HeatKer}
For any $l, N$ with $2 N > n + l + 1$,
the series   
$$ \sum _{k=0} ^\infty (-1)^k Q ^{(k)} _N (\cdot , t) $$
converges to a limit $Q (\cdot , t) \in \Gamma ^0 (\bt ^{-1} \rE \otimes \bs ^{-1} \rE ' \times (0, \infty) ) $, 
independent of $N$, in the $\| \cdot \|_l$ norm.
Furthermore,
\begin{enumerate}
\item
The section $Q$ is the heat kernel of $\partial _t + \Delta ^\rE + F + K $;
\item
$ G _N $ approximates $Q$ in the sense that 
$$ \| Q - G _N \| _l = O (t) $$
as $t \rightarrow 0$.
\end{enumerate}
\end{lem}
\begin{proof}
From (1) of Lemma \ref{QEst}, one has $Q ^{(k)} _N < \frac{1}{2 ^k }$ for sufficient large $k$.
Convergence of the series $\sum _{k=0} ^\infty (-1)^k Q _N ^{(k)}$ follows from the comparison test.
Assertion (2) follows from $ Q - G _N = \sum _{k=1} ^\infty Q _N ^{(k)}$,
and implies the initial condition of (1), i.e.,
$$ \lim _{t \rightarrow 0^+} \| Q \circ u - u \| _l = 0, $$
since 
$$ \| Q \circ u - u \| \leq \| (Q - G _N ) \circ u \| _l + \| G _N \circ u - u \| _l \rightarrow 0. $$
To show that $(\partial _t + \Delta ^\rE + F + K) Q = 0$, 
observe that $\| (\partial _t + \Delta ^\rE + F + K) Q _N ^{(k)} \| _l \leq 2 ^{-k} $ for sufficient large $k$.
Therefore one has 
\begin{align*}
(\partial _t + \Delta ^\rE + F + K) \sum _{k=1}^\infty (-1) ^k Q _N ^{(k)}
=& \: \sum _{k=1} ^\infty (-1)^k (\partial _t + \Delta ^\rE + F + K) Q _N ^{(k)} \\
=& \: R ^{(1)} _N + \sum _{k=2} ^\infty (-1)^k (R ^{(k)} _N + R ^{(k-1)} _N) \\
=& \: 0.
\end{align*}
\end{proof}

\begin{nota}
We shall denote the heat kernel of the Laplacian $\Delta ^\rE + F + K$, as constructed above, by
$$ e ^{- t (\Delta ^\rE + F + K) } := Q (\cdot , t). $$
\end{nota}

\begin{rem}
Alternatively, let $e ^{-t (\Delta ^\rE + F)} $ be the heat kernel of 
$\Delta ^\rE + F$ constructed using the same method as above.
Then a heat kernel of $\Delta ^\rE + F + K$ is given by
\begin{equation}
\label{HeatPert}
e ^{-t (\Delta ^\rE + F + K)}  = e ^{-t (\Delta ^\rE + F)} + \sum_{i = 1} ^\infty t^i \tilde Q ^{(i)} ,
\end{equation}
where 
$$ \tilde Q ^{(i)} 
:= \int _{0 < \tau _0 < \cdots < \tau _i < 1} 
e ^{-t (\Delta ^\rE + F)} (\cdot , \tau _0 t) \circ \kappa 
\circ e ^{-t (\Delta ^\rE + F)} (\cdot , \tau _1 t) \circ \kappa \circ \cdots
\circ \kappa \circ e ^{-t (\Delta ^\rE + F)} (\cdot , \tau _i t),$$
and the integration is over the Lebesgue measure.
\end{rem}
As in the case of manifolds with bounded geometry, 
the heat kernel of Laplacian on groupoids satisfies the following `off diagonal' estimate:
\begin{prop}
\label{HeatDecay}
Fix $\varepsilon > 0$ such that for any $a \in \cG$,
$\kappa (a b ^{-1}) = 0$ and $ G _N (a b ^{-1}, t) \\ = 0 $ for any $t$,
whenever $b \in \cG _{\bs (a)} \backslash B (a, \varepsilon )$.
Let $t > 0$ be fixed.
For any $\lambda > 0$, there exists $C > 0 $ such that 
\begin{equation}
\label{QDecay}
| e ^{- t (\Delta ^\rE + F + K) }(a, t) | 
\leq C e ^{- \lambda d (a, \bs (a))}, \quad \forall a \in \cG, d (a , \bs (a) ) > 2 \varepsilon ,
\end{equation}
and $Q (a ^{-1} , t) \in \bL ^1 (\cG _{\bs (a)})$ 
\end{prop}
\begin{proof}
Let $I \in \bbN $ be such that $I \varepsilon \leq d (a , \bs (a)) \leq (I + 1) \varepsilon $.
Then $Q _N ^{(k)} (a, t) = 0 $ for any $k < I$.
Therefore one has
\begin{align*}
| Q (a, t) | e ^{\lambda d (a, \bs (a))}
\leq & \: \sum _{k = I} ^\infty e ^{\lambda (I + 1) \varepsilon } 
C' _0 \tilde C_0 ^k M ^k (1 + t ^{N - \frac{n}{2}} )^k t ^k (k!)^{-1} \\
=& \: e ^{\lambda (I + 1) \varepsilon } 
\frac{ C' _0 \tilde C_0 ^I M ^I (1 + t ^{N - \frac{n}{2}} )^I t ^I }{I!}
\times \sum _{k = 0} ^\infty \frac{ \tilde C_0 ^k M ^k (1 + t ^{N - \frac{n}{2}} )^k t ^k I! }{(k + I)!}.
\end{align*}
It is clear that the last expression goes to $0$ as $I \rightarrow \infty$, 
so Equation (\ref{QDecay}) is proved.
From Equation (\ref{QDecay}), one has
$$Q (a ^{-1} , t) \leq C e ^{- \lambda d (a ^{-1}, \bt (a))} = C e ^{- \lambda d (a, \bs (a)) }. $$
It follows that $Q (a^{-1} , t) \in \bL ^1 (\cG _{\bs (a)})$
because $\cG _{\bs (a)}$ has at most exponential volume growth.
\end{proof}

\subsubsection{\bf The heat kernel of the vector representation}
We turn to study the heat kernel of $\nu (\partial _t + \nabla ^\rE + F + K)$,
where $\nu$ is the vector representation.
The construction becomes very simple, 
once we know the heat kernel of $(\partial _t + \nabla ^\rE + F + K)$.
\begin{thm}
\label{VectorHeat}
If $Q $ is a heat kernel of $\partial _t + \nabla ^\rE + F + K$, 
then $\nu (Q)$ is a heat kernel of $\nu (\partial _t + \nabla ^\rE + F + K)$ in the sense that
\begin{align}
\nu (\partial _t + \nabla ^\rE + F + K) \nu (Q) f =& \: 0, \quad \forall t > 0 \\ \nonumber
\lim _{t \rightarrow 0^+} \| \nu (Q) f - f \| =& \: 0
\end{align}
for any $f \in \Gamma ^\infty (\rE)$.
\end{thm}
\begin{proof}
By Proposition \ref{HeatDecay}, $\nu Q$ is well defined for each $t \leq 0$.
By definition one has
$$ \bt ^{-1} ( \nu (\partial _t + \nabla ^\rE + F + K) \nu (Q) f)
= (\partial _t + \nabla ^\rE + F + K) Q (\bt ^{-1} f) = 0.$$
The second equality follows by a similar argument.   
\end{proof} 
One important observation from Theorem \ref{VectorHeat} is that the heat kernel of the vector representation
is not a smoothing operator.
However, if $\cG \rightrightarrows \rM$ is a Lauter-Nistor groupoid in the sense of Definition \ref{BdGpoid},
then for any $f \in \Gamma ^\infty _c (\rE | _{\rM_0})$, 
one has
\begin{equation}
\nu (K) (f) (x) = \int _{a \in \cG _x } \kappa (a ^{-1} ) f (\bt (a)) \mu _{x} (a) 
= \int _{y \in \rM _\alpha } \kappa |_{\cG _{\rM _0}} (x , y ) f (y) \mu _{\rM _0}, 
\end{equation}
where $\rM _\alpha $ is the connected component of $\rM _0$ containing $x$ 
and we have used the identification $\cG _{\rM _0} \cong \coprod _\alpha \rM _\alpha \times \rM _\alpha $. 

\subsubsection{\bf Application: Heat kernel in edge calculus}
As an application of our construction, 
we give a simple proof to Albin's conjecture on generalization of \cite[Theorem 4.3]{Albin;EdgeInd}.
We refer to the same paper for details.
\begin{thm}
\label{AlbinConj}
A Laplacian operator on any manifolds $\rM $ with iterated complete edge has a heat kernel.
\end{thm}
\begin{proof}
By \cite{Nistor;LieMfld}, the pseudo-differential calculus is defined by a groupoid $\cG$ 
over the compactification $\rM$ of $\rM _0 $.
In particular, any Laplacian on $\rM _0$ is the vector representation of a Laplacian operator on $\cG$. 
The lemma follows from our constructions above.
\end{proof}

\subsection{ Transverse regularity of the heat kernel}
In the last Section, we proved that the series   
$ \sum _{k=0} ^\infty (-1)^k Q ^{(k)} _N (\cdot , t) $
converges to the heat kernel $Q (\cdot , t) $ in the $\| \cdot \| _l $ norms.
It follows that $Q $ is smooth on each $\bs $-fiber.
In this section, we consider the problem of regularity of the heat kernel $Q$.

\subsubsection{\bf Riemannian metrics and connections on the groupoid $\cG$}
Let $\cG$ be a groupoid with compact units $\rM$, let $\bs$ be the source map.
As in the beginning of this section,
we have already fixed an invariant metric $g _\cA$ on the foliation $\ker (d \bs ) \subset T \cG$.
We shall extend $g _\cA$ to $T \cG$.
Fix a distribution $\cH \subset T \cG$ complementary to $\ker (d \bs )$.
Then the differential $d \bs $ identifies $\cH \cong \bs ^{-1} T \rM$.
It follows that any metric on $\rM$ defines a metric $g _\cH$ n $\cH$.
We define the metric $g _\cG$ on $\cG$ by taking the orthogonal sum of $\cH $ and $\ker (d \bs )$. 

The distribution $\cH $ canonically induces a splitting 
$$ T \tilde \cG = \ker (d \tilde \bs ) \oplus \ker (d \tilde \bt ) \oplus \cH ^{(2)},$$
such that $\cH = \{ d \tilde \bs (X ) : X \in \cH ^{(2)} \} = \{ d \tilde \bt (X ) : X \in \cH ^{(2)} \}$
(see \cite{Heitsch;FoliHeat}).
Indeed, one can write down $\cH ^{(2)}$ explicitly: 
$$ \cH ^{(2)} := (\cH \times \cH ) \bigcap T \tilde \cG .$$
Also, note that the relation $\bs ^{(2)} = \bs \circ \tilde \bt = \bs \circ \tilde \bs $ implies
$$\ker (d \bs ^{(2)} ) = \ker (d \tilde \bs ) \oplus \ker (d \tilde \bt ).$$

Given any metric $g _\cG $ as above,
the splitting $T \tilde \cG = \ker (d \tilde \bs ) \oplus \ker (d \tilde \bt ) \oplus \cH ^{(2)} $
defines a metric on $\tilde \cG$, which shall be denoted by $\tilde g _\cG $.

Next, we equip $T \cG $ with a special connection, following \cite{Heitsch;FoliHeat}.
Recall that one has identification $T \cG = \Ker ( d \bs ) \oplus \bs ^{-1} T ^* \rM $.
Denote the orthogonal projection onto $\Ker ( d \bs )$ by $P ^\cV $.
Take the Levi-Civita connection $\nabla ^{T \cG }$ on $(\cG , g _\cG ) $.
Then $\nabla ^{T \cG} $ induces a connection on $\cV := \Ker ( d \bs )$ by 
$$ \nabla ^\cV _X Y := P ^\cV \nabla ^{T \cG} _X Y 
\quad \forall X \in T \cG , Y \in \Gamma ^\infty (\Ker (d \bs) ) \subset \Gamma ^\infty (T \cG).$$
We define the connection $\nabla ^{\cV \oplus \cH} $ on $ T \cG$ by taking the direct sum of 
$\nabla ^\cV $ and $\bs ^{-1 } \nabla ^{T \rM}$.

\subsubsection{\bf The regularity theorem} 
In this section, we state and prove our transverse regularity theorem.
Let $\cG \rightrightarrows \rM $ be a groupoid with $\rM $ compact. 
We shall assume that the Lie algebroid $\cA $ is orientable. 
Let $\mu $ be the $\bs$-fiber-wise invariant Riemannian volume form.

Recall that $\tilde \cG := \{ (a , b ) \in \cG \times \cG : \bs (a ) = \bs (b) \} $ 
and $\widetilde \bm (a , b) = a b ^{-1} , \quad \forall (a , b ) \in \tilde \cG $.
Also, we write $\widetilde \bm _* $ to denote the differential of $\widetilde \bm $,
regarded as a bundle map, 
i.e. $\widetilde \bm \in \Gamma ^\infty (\Hom T (\tilde \cG , \widetilde \bm ^{-1} T \cG ))$;
and $\fL ^{(m)} \mu $ to denote the $m$-th Lie derivative of $\mu$.
(see Appendix \ref{DGNonsense}).

\begin{thm}
\label{EstReg}
Assume that
\begin{enumerate}
\item
The source map $\bs : \cG \to \rM $ is a fiber bundle;
\item
For each $m \in \bbN$, there exist constants $C _m, \varepsilon _m > 0$ such that 
$$ | (\nabla ^{\Hom T (\tilde \cG , \widetilde \bm ^{-1} T \cG )}) ^m \widetilde \bm _* | (b' , b) 
\leq C _m e ^{\varepsilon _m ( d _\bs (b' , \bs (b')) + d _\bs (b , \bs (b)))}; $$
\item
The Lie derivatives of $\mu $ satisfy the estimate 
$$ | \fL ^{(m)} \mu (X _1 ^{\tilde \cH} , \cdots , X _m ^{\tilde \cH} ) (b' , b)| 
\leq C _m e ^{\varepsilon _m ( d _\bs (b' , \bs (b')) + d _\bs (b , \bs (b)))} |X _1 | \cdots |X _m| .$$
\end{enumerate}
Then for any $F \in \Gamma ^\infty (\bt ^{-1} \rE \otimes \bs ^{-1} \rE '), K \in \Psi ^{- \infty } _\mu (\cG , \rE )$, 
the heat kernel $e ^{-t (\Delta ^\rE + F + K ) } \in \Gamma ^\infty _b .$
\end{thm}
\begin{proof}
Recall from Lemma \ref{HeatKer} that the heat kernel is defined to be the sum 
$$ e ^{- t (\Delta ^\rE + F + K)} = \sum _{k = 0 } ^\infty ( -1 ) ^k Q ^{(k)}, $$
where, using Equation (\ref{ConvDef3}), the $Q ^{(k)} $ have the form:
\begin{align*}
Q ^{(0)} (a , t) =& \: G _N (a , t)\\
Q ^{(k)} (a , t) =& \: \int _0 ^t \int _{(b' , b) \in \tilde \bt ^{-1} (\bs (a))}
\widetilde \bm ^{-1} G _N (b' , b, t - \tau ) \tilde \bs ^{-1} R ^{(k-1) } (b', b, \tau) \tilde \mu (\tilde b) d \tau,
\end{align*}
where $R _N ^{(k)} $ is defined by taking convolution product of 
$R _N ^{(1)} := (\partial _t + \Delta ^\rE + F + K ) G _N $ with itself $k$-times.

Fix a connection $\nabla ^\rE $ on $\rE \to \rM$.
We denote by $\nabla ^{\bt ^{-1} \rE \otimes \bs ^{-1} \rE'}$ to be the tensor of the pullbacks of $\nabla ^\rE$
by $\bs $ and $\bt$. 
Hence $\nabla ^{\bt ^{-1} \rE \otimes \bs ^{-1} \rE' }$ is a connection on $\cG$.
Pulling-back again by $\tilde \bt$, 
one has the bundle $\tilde \bt ^{-1} (\bt ^{-1} \rE \otimes \bs ^{-1} \rE ' ) $ over $ \tilde \cG $, 
and the corresponding connection $\nabla ^{\tilde \bt ^{-1} (\bt ^{-1} \rE \otimes \bs ^{-1} \rE ' )}$.

We begin with estimating the covariant derivatives of $R _N ^{(1)}$.
Taking covariant derivative throughout the proof of (2) of Lemma \ref{ParaProof},
one gets 
$$\nabla ^{\bt ^{-1} \rE \otimes \bs ^{-1} \rE' } (\kappa \circ G _N ) 
\to \nabla ^{\bt ^{-1} \rE \otimes \bs ^{-1} \rE' } G _N ,$$
as $t $ goes to $0$.
Modifying the arguments of the proof of (1) of Lemma \ref{ParaProof} in the same manner, 
one gets the estimate
$$ \| (\nabla ^{\bt ^{-1} \rE \otimes \bs ^{-1} \rE' } )^m ((\partial _t + \Delta ^\rE + F) G _N ) \|_0
\leq C _m ^{(1)} t ^{N - \frac{n}{2} - \frac{l}{2} - m}.$$ 
Combining the two, it follows that
$$ \| (\nabla ^{\bt ^{-1} \rE \otimes \bs ^{-1} \rE' } ) ^m R ^{(1)} _N (\cdot , t) \| _0 
\leq C _m ^{(1)} (t ^{N - \frac{n + l}{2} - m} + 1) $$
for some constants $C _m ^{(1)}$ independent of $t$.

Next, we estimate the derivatives of $R _N ^{(k)}$.
Write 
$$ R _N ^{(k)} (a, t) = \int _0 ^t \int _{(b', b) \in \tilde \bs ^{-1} (a)}
\widetilde \bm ^{-1} R _N ^{(1)} (b' , b, t - \tau) \tilde \bs ^{-1} R _N ^{(k-1)} (b' , b, \tau) 
\tilde \mu (b', b) d \tau .$$
Then the corollaries of Lemma \ref{DFiberInt} imply for any (local) vector field $X$ on $\cG$,
\begin{align*} 
(\nabla &^{\bt ^{-1} \rE \otimes \bs ^{-1} \rE' } _X R _N ^{(k)}) (a, t) \\
=& \int _0 ^t \int _{(b', b) \in \tilde \bs ^{-1} (a)}
\nabla ^{\tilde \bt (\bt ^{-1} \rE \otimes \bs ^{-1} \rE')}
\big( \widetilde \bm ^{-1} R _N ^{(1)} (b' , b, t - \tau) \tilde \bs ^{-1} R _N ^{(k-1)} (b' , b, \tau) \big)
(X ^{\tilde \cH}) \tilde \mu \, d \tau \\
&+ \int _0 ^t \int _{(b', b) \in \tilde \bs ^{-1} (a)}
\widetilde \bm ^{-1} R _N ^{(1)} (b' , b, t - \tau) \tilde \bs ^{-1} R _N ^{(k-1)} (b' , b, \tau) 
\big( \fL ^{(1)} \tilde \mu (X ^{\tilde \cH}) \big) \, d \tau ,
\end{align*}
where $X ^{\tilde \cH} \in \Gamma (\tilde \cH ) \subset \Gamma (T \tilde \cG)$ is the horizontal lift of $X$.
Observe that for all $t > 0$, $ R _N ^{(1)} (\cdot , t )$ is supported on a set of the form 
$\{ a \in \cG : d _\bs (a , \bs (a) ) \leq \rho \} $ 
for some $\rho > 0 $.
It follows that $ R _N ^{(k-1)} $ is supported on the set 
$\{ a \in \cG : d _\bs (a , \bs (a) ) \leq (k - 1 ) \rho \} $;
and $ \widetilde \bm ^{-1} R _ N^{(1)} (b', b, t - \tau)$  is supported on the compact set 
$ \{ d _\bs ( b' , b ) \leq \rho \}. $ 
Hence, for each $a \in \cG$ fixed, the domain of integration can be re-written as
$$ B (a , \rho ) := \{ b \in \cG : \bs (b ) = \bs (a) , d _\bs (a , b ) \leq \rho \}, $$
and whose volume is bounded by some constant $M$ independent of $a$.

Expanding the first integrand using Leibniz rule, one gets:
\begin{align*}
\nabla ^{\tilde \bt (\bt ^{-1} \rE \otimes \bs ^{-1} \rE')} &
\big( \widetilde \bm ^{-1} R _N ^{(1)} (b' , b, t - \tau) \tilde \bs ^{-1} R _N ^{(k-1)} (b' , b, \tau) \big)
(X ^{\tilde \cH}) \\
=& \big( \nabla ^{\widetilde \bm ^{-1} (\bt ^{-1} \rE \otimes \bs ^{-1} \rE' )}
\widetilde \bm ^{-1} R _N ^{(1)} (b' , b, t - \tau) (X ^{\tilde \cH}) \big)
\tilde \bs ^{-1} R _N ^{(k-1)} (b' , b, \tau) \\
&+ \widetilde \bm ^{-1} R _N ^{(1)} (b' , b, t - \tau) 
\big( \nabla ^{\tilde \bs ^{-1} (\bt ^{-1} \rE \otimes \bs ^{-1} \rE' ) }
\tilde \bs ^{-1} R _N ^{(k-1)} (b' , b, \tau) (X ^{\tilde \cH }) \big) \\
=& \big( (\widetilde \bm ^{-1} \nabla ^{\bt ^{-1} \rE \otimes \bs ^{-1} \rE' }
R _N ^{(1)} (b' , b, t - \tau ) ) (\widetilde \bm _* X ^{\tilde \cH} ) \big)
\tilde \bs ^{-1} R _N ^{(k-1)} (b' , b, \tau) \\
&+ \widetilde \bm ^{-1} R _N ^{(1)} (b' , b, t - \tau ) 
\big( \tilde \bs ^{-1} ( \nabla ^{ \bt ^{-1} \rE \otimes \bs ^{-1} \rE' }
R _N ^{(k-1)} (b' , b, \tau ) (X ) ) \big),
\end{align*}
where the last line follows from of Equation (\ref{PullbackDer}) and the observation that 
$d \tilde s (X ^{\tilde \cH } ) = X $.
Using hypothesis (2), one has for any $(b' , b) \in \tilde \cG $,
\begin{align*}
\big| ( \widetilde \bm ^{-1} \nabla ^{\bt ^{-1} \rE \otimes \bs ^{-1} \rE' } 
R _N ^{(1)} ) & (\widetilde \bm _* X ^{\tilde \cH} ) \big| (b', b, \tau) \\
\leq & \| \nabla ^{\bt ^{-1} \rE \otimes \bs ^{-1} \rE' } R _N ^{(1)} (\cdot , \tau ) \| _0
| X (b ' b ^{-1} ) | C _0 e ^{\varepsilon _0 ( d _\bs (b' , \bs (b')) + d _\bs (b , \bs (b)))} \\
\leq & C _m ^{(1)} (t ^{N - \frac{n + l}{2} - 1} + 1) | X (b ' b ^{-1} )| 
C _0 e ^{\varepsilon _0 ( d _\bs (b' , \bs (b')) + d _\bs (b , \bs (b)))} .
\end{align*}

Now, one can estimate $ \| \nabla ^{\bt ^{-1} \rE \otimes \bs ^{-1} \rE' } R _N ^{(k)} \| _0$.
For any $a \in \cG$,
\begin{align*}
\big| \nabla ^{\bt ^{-1} \rE \otimes \bs ^{-1} \rE' } & R _N ^{(k)} (a , t ) \big| \\
\leq \int _0 ^t & \int _{ B (a , \rho ) } 
 C _0 e ^{\varepsilon _0 2 k \rho }
\| \nabla ^{\bt ^{-1} \rE \otimes \bs ^{-1} \rE' } R _N ^{(1)} (\cdot , t - \tau ) \| _0 
\| R _N ^{(k - 1)} (\cdot , \tau ) \| _0 \mu (b ) \: d \tau \\
+& \int _0 ^t \int _{b \in B (a , \rho )} \| R _N ^{(1)} (\cdot , t - \tau ) \| _0
\| \nabla ^{\bt ^{-1} \rE \otimes \bs ^{-1} \rE ' } R _N ^{(k - 1)} (\cdot , \tau ) \| _0 \mu (b ) \: d \tau\\
+& \int _0 ^t \int _{b \in B (a , \rho )} \| R _N ^{(1)} (\cdot , t - \tau ) \| _0
\| R _N ^{(k - 1)} (\cdot , \tau ) \| _0  C _0 e ^{\varepsilon _0 2 k \rho } \mu (b ) \: d \tau,
\intertext{where we used hypothesis (3) to estimate the last term, }
\leq \int _0 ^t & \int _{b \in B (a , \rho )} C _1 ^{(1)} (1 + (t - \tau )^{N - \frac{n + l}{2} - 1}) 
C _0 e ^{\varepsilon _0 2 k \rho } \\
& \times \tilde C_0 ^{k-2} M ^{k-2} (1 + \tau ^{N - \frac{n+l}{2}} )^{k-1} \tau ^{k-2} ((k-2)!)^{-1} 
\mu (b) \: d \tau \\
+& \int _0 ^t \int _{b \in B (a , \rho )} \tilde C _0 (1 + (t - \tau )^{N - \frac{n+l}{2}} )
\| \nabla ^{\bt ^{-1} \rE \otimes \bs ^{-1} \rE ' } R _N ^{(k - 1)} (\cdot , \tau ) \| _0 \mu (b ) \: d \tau \\
+& \int _0 ^t \int _{b \in B (a , \rho )} \tilde C _0 (1 + (t - \tau )^{N - \frac{n+l}{2}} ) \\
& \times \tilde C_0 ^{k-2} M ^{k-2} (1 + \tau ^{N - \frac{n+l}{2}} )^{k-1} \tau ^{k-2} ((k-2)!)^{-1} 
C _0 e ^{\varepsilon _0 2 k \rho } \mu (b ) \: d \tau \\
\intertext{}
\leq \int _0 ^t & \tilde C _0 (2 + t ^{N - \frac{n+l}{2}} ) M
\| \nabla ^{\bt ^{-1} \rE \otimes \bs ^{-1} \rE ' } R _N ^{(k - 1)} (\cdot , \tau ) \| _0 \: d \tau \\
&+ (C _1 ^{(1)} + \tilde C _0 ) \tilde C _0 ^{k - 2} M ^{k - 1} C _0 e ^{\varepsilon _0 2 k \rho } 
(2 + t ^{N - \frac{n+l}{2}} ) ^k t ^{k-1} ((k-1)!)^{-1}
\end{align*}
Using an induction argument, it is straightforward (but tedious) to obtain the following estimate
for any $N > \frac{n}{2} + m + 1$:
\begin{equation}
\| \nabla ^{\bt ^{-1} \rE \otimes \bs ^{-1} \rE' } R _N ^{(k)} (\cdot , t) \| _0
\leq k \tilde C _1 ^k M ^k e ^{\varepsilon _0 2 k \rho } (t ^{N - \frac{n}{2} } + 2 ) ^k t ^{k-1} ((k - 1) !)^{-1},
\end{equation}
for some constant $ \tilde C _1 $.

It is straightforward (if not tedious) to repeat the same arguments above to get estimates for higher derivatives:
\begin{equation}
\| (\nabla ^{\bt ^{-1} \rE \otimes \bs ^{-1} \rE' })^m R _N ^{(k)} (\cdot , t) \| _0
\leq k^m \tilde C _m ^k M ^k e ^{\varepsilon' _m 2 k \rho } (t ^{N - \frac{n}{2} } + m ) ^k t ^{k-1} ((k - 1) !)^{-1},
\end{equation}
for some constants $ \tilde C _m , \varepsilon ' _m$.
Finally, arguments similar to the proof of (1) Lemma \ref{QEst} gives the estimate
$$ \| \nabla ^{\bt ^{-1} \rE \otimes \bs ^{-1} \rE ' } _X Q _N ^{(k)} (a , t) \| _0 
\leq k C ' _1 e ^{2 \varepsilon _0 k R} t ^k ((k - 1 )! ) ^{-1} .$$
It follows that $ \sum _{k = 0 } ^\infty (-1 )^k Q ^{(k)} _N $ converges uniformly in all derivatives up to order $m$,
provided $N > \frac{n}{2} + m $.
Since $N$ is arbitrary, it follows that
$ e ^{- t (\Delta ^\rE + F + K)} \in \Gamma ^\infty _b (\bt ^{-1} \rE \otimes \bs ^{-1} \rE ' )$.
\end{proof}

\subsubsection{\bf Example: the Bruhat sphere}
We again look at the example of the Bruhat sphere. 
We shall explicitly define a metric on the groupoid $\cG = \rT \backslash (\rK \times \rN)$.
Observe that $\cG $ is an associated bundle over $\rK = \bbC \rP (1)$.
It is well known that one has identifications as vector bundles
\begin{align*}
T \cG \cong & \: (T \rT ) \backslash (T \rK \times T \rN ) \\
\Ker (d \bs ) \cong & \: \rT \backslash (\rK \times T \rN ).
\end{align*}

Observe that 
$\cG = \rT \backslash (\rK \times \rN)$ is an associated bundle of the principal bundle $\rK \to \rT \backslash \rK$,
hence the arguments in \cite[Section 11]{KMS;Book} can be used to fix a complementary distribution to $\Ker (d \bs)$.
Fix an $\ad \rK$-invariant metric $g _\fk$ on $\rK$, the Lie algebra of $\rK$.
Let $\ft^\bot $ be the orthogonal complement of $\ft \subset \fk$.
Define $\rT ^\bot$ to be the distribution on $\rK$
$$ \rT ^\bot := \{ d R ^\rK _k (X) : k \in \rK, X \in \ft ^\bot \} \subseteq T \rK ,$$
and the distribution on $\cG$ 
$$ \cH := \{ d \wp _\rT (X , 0) \in T (\rT \backslash (\rK \times \rN)) : (X , 0) \in \rT ^\bot \times T \rN \},$$
where $\wp$ denotes the projection onto he coset space. 
It is easy to see that $\rH$ is a distribution complementary to $\ker (d \bs ) = \rT \backslash (\rK \times T \rN)$.
To define a metric on $\rH$, one simply takes the pullback of the round metric $g _\fk$ on $\rT \backslash \rK$,
more explicitly,
$$ g _\cH ({\wp}_\rT (d R_k ^\rK  (X _1 ) , 0) , {\wp} _\rT (d R ^\rK _k ( X _2 ) , 0)) := g _{\fk} (X _1 , X _2). $$
Finally, we define a metric $g _\cG$ on $\cG = \rT \backslash (\rK \times \rN)$ by taking the orthogonal sum of 
$g_\bs $ and $g _\cH$.

In our special case $\cG = \rT \backslash (\rK \times \rN)$, $\tilde \cG ^{(2)} $ is diffeomorphic to
$\rT \backslash (\rK \times \rN \times \rN)$,
where $\rT$ acts on $\rK$ by right multiplication and on $\rN \times \rN$ by conjugation,
and the diffeomorphism is given explicitly by
$$ {}_\rT (k , n_1 , n_2) 
\mapsto ({}_\rT (k , n_1) , {} _\rT (k , n_2)) \in \cG \times \cG \supseteq \tilde \cG ^{(2)}. $$ 

Consider the map $\widetilde \bm : \tilde \cG ^{(2)} \to \cG , \widetilde \bm (a, b) := a b ^{-1}$.
One has the commutative digram
\begin{equation}
\begin{CD}
\rK \times \rN \times \rN @> \widehat \bm >> \rK \times \rN \\
@VVV @VVV \\
\tilde \cG ^{(2)} \cong \rT \backslash (\rK \times \rN \times \rN ) 
@>\widetilde \bm >> \cG = \rT \backslash (\rK \times \rN),
\end{CD}
\end{equation}
where $\widehat \bm : \rK \times \rN \times \rN \to \rK \times \rN$ 
is defined to be the function 
$$ \widehat \bm (k, n_1 , n_2) := (k', n_1 n_2 ^{-1} ). $$

We verify that the metric we constructed satisfies the assumptions of Theorem \ref{EstReg}. 
Hence the heat kernel of a Laplacian operator on the Bruhat sphere is smooth. 
\begin{lem}
\label{CP1Est}
For each $m \in \bbN$,
there exists a polynomial $p _m $ on $\rN \times \rN = \bbR ^2 \times \bbR ^2$ such that,
for any ${} _\rT (k, n_1 , n_2 ) \in \tilde \cG ^{(2)}, X \in T _{ {}_\rT (k , n_1 , n_2) } \tilde \cG ^{(2)} $,
$$ \big| (\nabla ^{\Hom T (\tilde \cG , \widetilde \bm ^{-1} T \cG )})^m  \widetilde \bm ( X ) \big|_{g _\cG } 
\leq | p _m (n_1 , n_2) | | X | _{g ^{(2)} _\cG}.$$ 
\end{lem}
\begin{proof}
We prove this lemma by direct computation.
First, one obtains a formula for $\widetilde \bm $ using Equation (\ref{CP2Formula}).
Namely, for any 
$k = 
\left( 
\begin{smallmatrix}
\alpha & \beta \\
- \bar \beta & \bar \alpha 
\end{smallmatrix}
\right) ,
n_1 =
\left( 
\begin{smallmatrix}
1 & w' \\
0 & 1
\end{smallmatrix}
\right) ,
n_2 =
\left( 
\begin{smallmatrix}
1 & w \\
0 & 1
\end{smallmatrix}
\right) ,$
one has
$$ \widetilde \bm ({}_\rT (k , n_2 , n_2)) =  
\left[ \frac{\alpha - \bar w \beta }{(|\beta |^2 + |\alpha - \bar w \beta |^2)^{\frac{1}{2}}},
\frac{ \beta }{(|\beta |^2 + |\alpha - \bar w \beta |^2)^{\frac{1}{2}}} \right] ^{w' - w} _\rT .$$
First consider the $ \cH $-component. 
Any vector $X \in \rT ^\bot $ can be written in the form
$$ X = \partial _t \Big| _{t = 0}
\left(
k e^{ t \left(
\begin{smallmatrix}
0 & v \\
- \bar v & 0
\end{smallmatrix} 
\right) } , n_1 , n_2 \right) ,
\quad v \in \bbC .$$
Then, $d \widetilde \bm (X) $ is by definition:
\begin{align*}
& d \widehat \bm (X) :=  
\partial _t \Big| _{t = 0}
\widetilde \bm \left(
e^{ t \left(
\begin{smallmatrix}
0 & v \\
- \bar v & 0
\end{smallmatrix} 
\right) } k , n_1 , n_2 \right) = \\
& 
\left(
\left( 
\begin{array}{cc}
\frac{(|\beta |^2 - |\alpha |^2 )(\bar w _2 v  - w _2 \bar v) 
+ |w _2|^2 (\bar \alpha \bar \beta v - \alpha \beta \bar v )}{2 Q }
 & \frac{v}{Q } \\
- \frac{\bar v}{Q } & \frac{- (|\beta |^2 - |\alpha |^2 )(\bar w _2 v  - w _2 \bar v) 
- |w _2|^2 (\bar \alpha \bar \beta v - \alpha \beta \bar v )}{2 Q }
\end{array} 
\right) k',
0 \right) ,
\end{align*} 
where we denoted $Q := |\beta  |^2 + |\alpha - \bar w _2 \beta | ^2 $.
It follows that
$$ d \widetilde \bm (d \wp _\rT (X)) = \frac{1}{Q} d \wp _\rT (X) .$$
Similarly, one has
\begin{align*}
\partial _t \big| _{t=0} \widehat \bm & (k , w _1 + z _1 t , w _2 ) 
= (0 , z _1 , 0 ) \\
\partial _t \big| _{t=0} \widehat \bm & (k , w _1 , w _2 + z _2 t ) \\
&= 
\left(
\left(
\begin{array}{cc}
\frac{(\alpha - \bar w _2 \beta )\bar \beta z _2 - (\bar \alpha - w _2 \bar \beta ) \beta \bar z _2 }{2 Q } & 
\frac{- |\beta | ^2 \bar z _2 }{Q} \\
\frac{ | \beta | ^2 z _2 }{Q} &
\frac{- (\alpha - \bar w _2 \beta )\bar \beta z _2 + (\bar \alpha - w _2 \bar \beta ) \beta \bar z _2 }{2 Q }
\end{array}
\right) k',
- z _2 \right) .
\end{align*}
We estimate a lower bound for $Q $ in terms of $w$.
Since it is clear that $Q \neq 0$, we may without loss of generality assume that $|w _2 | > 1$.
Suppose $| \beta | \leq \frac{1}{2 | w |} \leq \frac{1}{2}$.
Then the relation $| \alpha | ^2 + |\beta |^2 = 1 $ implies $| \alpha | \geq \frac{3}{4}$.
It follows that 
$| \alpha - \bar w \beta | ^2 \geq (\frac{3}{4} - \frac{1}{2} ) ^2 = \frac{1}{4} \geq \frac{1}{4 | w_2 |^2} $.
Hence $|\beta |^2 + | \alpha - \bar w \beta | ^2 \geq \frac{1}{4 |w |^2 } $ in both cases.

Finally, from the above computations, we observe that the coefficients of the 
$m$-th covariant derivatives of $\widetilde \bm$
are of the form 
$$ Q ^{- m} p _I (\alpha , \bar \alpha , \beta , \bar \beta , \omega , \bar \omega ),$$
where $p _I$ are polynomials. The assertion follows.
\end{proof} 

It is easy to see that the $\bs$-fiber-wise Riemannian volume form $\mu$ also satisfies similar estimates.
Therefore we conclude that
\begin{cor}
For any vector bundle $\rE \to \bbC \rP (1) $, any Riemannian metric $g _\cA $ on $\cA$,
$F \in \Gamma (\rE \otimes \rE ') $ and $K \in \Psi ^{- \infty } _\mu (\cG , \rE)$,
the heat kernel 
$$e ^{- t (\Delta ^\rE + F + K ) } 
\in \Gamma ^\infty _ b (\bt ^{-1} \rE \otimes \bs ^{-1} \rE ' \ltimes (0, \infty) )  .$$
\end{cor}

\subsection{ Short time asymptotic expansion of the heat kernel}
Let $\cG \rightrightarrows \rM$ be a groupoid with $\rM$ compact,
and the Lie algebroid $\cA \to \rM $ of even rank $\varkappa$.
Let $\nabla ^\rE$ be a Clifford $\cA$-connection and $\eth$ be the corresponding Dirac operator. 
Then a straightforward calculation shows that 
$$\eth ^2 = \Delta ^\rE + (\frac {1}{4} \tilde R + F ^{\rE / \rS}), $$ 
where $\tilde R$ is the scalar curvature and $F ^{\rE / \rS} $ is the twisting curvature.
Therefore the construction of the heat kernel above applies. 

Before stating our main result Lemma \ref{AsymLem}, we first need to define some notation.
Let $\bbC ^{k \times k} $ be the set of all matrices with coefficients in $\bbC$.
Given any power series $h : \bbC ^{k \times k} \to \bbC $ 
$$ h ( Z _{i j} ) = h (0) + \sum _{I} h _I Z _{I} ,$$
where the sum is over all multi-indexes $I = \{i_1 j_1 , i_2 j_2 , \cdots , i_p j_p \} $,
and $ Z _I := Z _{i_1 j_1} Z _{i_2 j_2} \\ \cdots Z _{i_p j_p}$.
Let $( \omega _{i j } ) \in \wedge \rV '$ be a matrix of 2-forms on some vector space $\rV$.
We {\it define } $h (\omega _{i j} )$ to be the polynomial 
$$ ( \omega _{i j} ) \mapsto h (0 ) + \sum _I h _I \omega _I \in \bigoplus _{l \text { even} } \wedge ^l \rV ',$$
where $\omega _I := \omega _{i_1 j_1} \wedge \omega _{i_2 j_2} \wedge \cdots \wedge \omega _{i_p j_p} $.

In particular, take $h$ to be the Taylor series expansion of 
$$ Z \mapsto 
\sqrt {\det \Big( \frac{Z}{2} ( \sinh \frac{Z}{2} )^{-1} \Big)} : \bbC ^{\varkappa \times \varkappa } \to \bbC ,$$
where $ Z \mapsto \frac{Z}{2} ( \sinh \frac{Z}{2} )^{-1} 
: \bbC ^{\varkappa \times \varkappa } \to \bbC ^{\varkappa \times \varkappa } $ is defined by the power series of 
$ z \mapsto \frac{z}{2} (\sinh \frac{z}{2}) ^{-1} : \bbC \to \bbC $.
Define the $\widehat \bbA $-genus by
\begin{equation}
\widehat \bbA := h (R).
\end{equation}
It is straightforward to check that $\widehat \bbA \in \Gamma ^\infty (\wedge \cA ')$ is a well defined section.

\begin{lem}
\label{AsymLem}
The heat kernel 
$e ^{- t \eth ^2} \in \Gamma ^\infty (\bs ^{-1} \rE \otimes \bt ^{-1} \rE ' \ltimes (0 , \infty) )$ 
has an asymptotic expansion
$$ e ^{- t \eth ^2 } (x, t) \cong (4 \pi )^{- \frac{n}{2}} \sum _{i = 0} ^\infty t ^{i - \frac{n}{2}} Q _i (x) ,
\quad \forall x \in \rM \subset \cG, $$
for some $Q _i \in \Gamma ^\infty (\Cl (\cA ') \otimes \End _{\Cl (\cA ')} (\rE ) ) $. 
Furthermore
\begin{enumerate}
\item
The coefficient $Q _i \in \Cl _{2 i} (\cA ') \otimes \End _{\Cl (\cA ')} (\rE ) $;
\item
One has
$$ \big( \str Q _{\frac{\varkappa }{2}} \big) \mu _\cA 
= \text{order $ \varkappa $ component of } \widehat \bbA \wedge \exp F ^{\rE / \rS }. $$
\end{enumerate}
\end{lem}
\begin{proof}
Regarding all operators involved as families of operators along the $\bs$-fibers,
the heat kernel of $\eth |_{\bs ^{-1} (x)}$ is just
$$ Q _x (a, b) := Q (a b ^{-1}), \bs (a ) = \bs (b) = x. $$ 
The computations of the asymptotic expansion of $Q _x (a, a) = Q (x)$ is very standard. 
See, for example, \cite[Chapter 4]{BGV;Book}.
\end{proof}

For convenience, we denote the order $ \varkappa $ component of $\widehat \bbA \wedge \exp F ^{\rE / \rS } $ by
$\Omega _\varkappa (\widehat \bbA \wedge \exp F ^{\rE / \rS } ) $.

It is easy to compute the asymptotic expansion of the heat kernel of the operator $\eth ^2 + K$.
From Equation (\ref{HeatPert}), the heat kernel of $\eth ^2 + K$ can be written as
$$ e ^{- t (\eth ^2 + K)} (a , t) = e ^{- t \eth ^2} (a , t) 
+ \sum _{i= 1} ^\infty (-1) ^i t ^i \tilde Q ^{(i)} (a , t) ,$$
where $\tilde Q $ is the heat kernel of $\Delta ^\rE $,
and $\tilde Q ^{(i)} := \int _{0 \leq \tau _0 \leq \cdots \leq \tau _i \leq 1} 
\tilde Q (\cdot , \tau _0 t) \circ \kappa \circ \tilde Q (\cdot , \tau _1 t) \circ \kappa \circ \cdots
\circ \kappa \circ \tilde Q (\cdot , \tau _i t).$
Since $\tilde Q ^{(i)} (\cdot , 0)$ are smooth, it follows immediately that
\begin{cor}
The heat kernel 
$e ^{- t (\eth ^2 + K) } \in \Gamma ^\infty (\bs ^{-1} \rE \otimes \bt ^{-1} \rE ' \ltimes (0, \infty ) )$ 
of the Laplacian $ \eth ^2 + K $ has an asymptotic expansion
$$ e ^{- t (\eth ^2 + K) } \cong (4 \pi )^{- \frac{n}{2}} \sum _{i = 0} ^\infty t ^{i - \frac{n}{2}} Q _i (x) ,
\quad \forall x \in \rM \subset \cG, $$
for some $Q _i \in \Gamma ^\infty (\Cl (\cA ') \otimes \End _{\Cl (\cA ')} (\rE ) ) $. 
Furthermore
\begin{enumerate}
\item
The coefficient $Q _i \in \Cl _{2 i} (\cA ') \otimes \End _{\Cl (\cA ')} (\rE ) $;
\item
One has
$$ \big( \str Q _{\frac{\varkappa }{2}} \big) \mu _\cA 
= \Omega _\varkappa (\widehat \bbA \wedge \exp F ^{\rE / \rS } ). $$
\end{enumerate}
\end{cor}

\vspace{9cm} 

\pagebreak \thispagestyle{firstpage}
\section{The renormalized trace and index theorem}
Consider a Fredholm operator on $\Gamma ^\infty (\rE) $ of the form $\nu (\eth + \varPsi )$, 
where $\eth$ is a Dirac operator and $R \in \Psi ^{- \infty} _\mu (\cG , \rE)$.

We saw that the heat kernel is not a smoothing operator in general,
and the usual trace formula
$$ \int _{\rM } \kappa (x , x) \mu (x) $$ 
cannot be applied.
Instead, one need to consider an extension of the trace functional, known as the renormalized trace.

\subsection{ The renormalized integral}
We shall only consider the case of the Bruhat sphere.
In this case, one has the two stereographic projection coordinates
$$ r e ^{i \vartheta } \mapsto [r e ^{i \vartheta }, 1] \text{ and } 
\dot r e ^{- i \vartheta } \mapsto [1 , \dot r e ^{- i \vartheta }] ,$$
and one can consider the cutoff integrals
$$ \int _{ r \leq r_0} \int _{0 \leq \vartheta \leq 2 \pi } f ([r e ^{i \vartheta }, 1]) r d \vartheta d r 
= \int _{\dot r \geq \frac{1}{r_0}} \int _{0 \leq \vartheta \leq 2 \pi } 
f ([1 , \dot r e ^{- i \vartheta }]) \frac{1}{\dot r ^3}  d \vartheta d \dot r $$
for any $f \in C^\infty (\bbC \rP (1))$ as $r _0 \to \infty$.
Using standard arguments reviewed in \cite{Paycha;Renorm}, one has:
\begin{lem}
\label{0Renorm}
For any $k = 1, 2, \cdots,$ and $F \in C ^\infty _c (\bbR)$, one has the expansion as $r _0 \to \infty$:
$$ \int _{\frac{1}{r _0}} ^\infty F \lambda ^{- k} d \lambda 
= \sum _{j = 1}^{k-1} C_j r _0 ^j + R \log r_0 + C _0 + O (r _0 ^{-1}) $$
for some constants $ C_0 , \cdots , C_j , R $.
In particular, the constant term $C _0 $ is given by the formula
$$ C_0 = \partial _\lambda ^{k - 1} F (0) \sum _{j = 1} ^{k - 1}\frac{1}{j} + 
\frac{1}{(k - 1)!} \int _0 ^\infty \partial _\lambda ^k F (\lambda ) \log \lambda d \lambda 
:= \sideset{_{\mathrm R}}{_0 ^\infty} \int F (\lambda ) \lambda ^{- k} d \lambda . $$ 
\end{lem} 

We return to the case of the Bruhat sphere.
Given any section $\omega \in \Gamma (\wedge ^2 \cA ')$,
such that $\omega $ is two times differentiable on $\bbC \rP (1) $ and three times differentiable on 
$\bbC \rP (1) \setminus \{ {}_\rT e \}$, 
we define:
\begin{dfn}
\label{CP1Renorm}
The {\it renormalized integral} of $\omega $ is the number
$$ \Rint \omega :=
\sideset{_{\mathrm R}}{_0 ^\infty} \int 
\Big( \int _0 ^{2 \pi} ( f \circ \dot \bx ) (\dot r e ^{- i \vartheta }) d \vartheta \Big) \dot r ^{- 3} d \dot r ,$$
where $ \omega = f \mu _0 $, and $\mu _0 $ is the volume form defined by the round metric. 
Here, note that $(\dot r , \vartheta ) \mapsto f \circ \dot \bx (\dot r e ^{- i \vartheta })$ 
is two times differentiable on the whole $\bbR ^2$-space. 
\end{dfn}

\begin{rem}
We may choose other volume forms instead of the round one,
and the result depends on our trivialization.
This discrepancy is well known.
See \cite{Paycha;Renorm} for a review.
\end{rem}

\subsection{ The renormalized trace and trace defect formula}
With the renormalized integral defined, it is natural to define the renormalized trace.
\begin{dfn}
Let $\rE $ be a $\bbZ _2$-graded vector bundle over $\rM$.
Let $\mu $ be a fixed $\bs$-fiberwise volume on $\cG$ identifying 
$\Psi ^{- \infty } (\cG , \rE) \cong \Gamma ^\infty (\bt ^{-1} \rE \otimes \bs ^{-1} \rE ' )$.
For any $K \in \Psi ^{- \infty } (\cG , \rE) $ with reduced kernel 
$\kappa \in \Gamma ^\infty (\bt ^{-1} \rE \otimes \bs ^{-1} \rE ' )$,
the {\it renormalized (super)-trace} of $K$ is defined to be 
$$ \sideset{_{\mathrm R}}{} \Str ( \varPsi ) := \Rint \str (\kappa |_{\rM} ) \mu .$$
\end{dfn}

In this section, we compute explicitly 
$$ \Rint \big( f \circ g (x) - g \circ f (x) \big) \mu _{\rM _0} (x),$$
where for simplicity we assume $f, g \in C^\infty (\cG)$ and $g$ is compactly supported 
(hence the convolution products are well defined).
In general, the expression is non zero. Hence proving that the `renormalized trace' is not a trace.

\begin{thm}
One has the trace defect formula
\begin{align*}
\Rint \big( & f \circ g (x) - g \circ f (x) \big) \mu _{\rM _0} (x) \\ \nonumber
=& - \pi \int _\bbC \big( \re ( w' ) \partial _{\dot x} + \im ( w') \partial _{\dot y} \big) 
\big( f ([1 , 0] ^{ \bar w' } _\rT) g([1 , 0] ^{ - \bar w' } _\rT) \big) | d w' |^2 . 
\end{align*}
\end{thm}
\begin{proof}
By Definition \ref{ConvDfn}, the convolution product $f \circ g$, 
written in Notation \ref{BruhatPair}, is given by the formula 
$$ f \circ g (z) 
= \int _{w \in \bbR^2 } 
f (\bx (z + \bar w , - w) ) g (\bx (z, w)) | d w |^2 . $$
As in Lemma \ref{0Renorm}, we need to consider 
$$ \int _{z \in B (0, r_0) } \int 
f (\bx (z + \bar w , - w) ) g (\bx (z, w)) | d w |^2 | d z |^2 $$
as $r _0 \to \infty $.
Performing the $z$-integral first and changing variable $z' = z - \bar w , w' = - w$,
the integral becomes
$$ \int \int _{z' \in B (- \bar w', r_0) }  
f (\bx (z ', w') ) g (\bx (z' + \bar w', - w')) | d z' |^2 | d w' |^2.$$
On the other hand, one has
$$ \int _{z \in B (0, r_0)} g \circ f |d z |^2 
= \int \int _{z \in B (0, r_0)} f (\bx (z , w)) g (z + \bar w , - w) |d z|^2 |d w|^2 .$$
Combining the two integrals, one gets
\begin{align*}
\int _{z \in B (0, r_0)} f \circ g - g \circ f & \: \mu _{\rM _0} (z) \\
= \: & \int \int _{z \in B (- \bar w , r_0) \backslash B (0, r_0)} 
f (\bx (z , w) ) g (\bx (z + \bar w , - w) ) |d z |^2 |d w |^2 \\
&- \int \int _{z \in B (0 , r_0) \backslash B (- \bar w, r_0)} 
f (\bx (z , w) ) g (\bx (z + \bar w , - w) ) |d z |^2 |d w |^2 
\end{align*}
after canceling the common domain.
In order to compute the integral, one needs to parametrize the domains 
$B (- \bar w , r_0) \backslash B (0, r_0)$ and $ B (0 , r_0) \backslash B (- \bar w, r_0)$.
For each $w \in \bbC$, consider the sets
\begin{align*}
S ^+ _w :=& \: \Big\{ - \frac{ r_0 e ^{i \varphi } \bar w }{| w |} - \lambda \bar w 
: - \frac{\pi}{2} \leq \varphi \leq \frac{\pi}{2} , 0 \leq \lambda \leq 1 \Big\} \\
S ^- _w :=& \: \Big\{ - \frac{ r_0 e ^{i \varphi } \bar w }{| w |} - \lambda \bar w 
: \frac{\pi}{2} \leq \varphi \leq \frac{3 \pi}{2} , 0 \leq \lambda \leq 1 \Big\}.
\end{align*}
It is elementary to see that
$$ S ^+ _w \backslash S ^- _w = B (- \bar w , r_0) \backslash B (0, r_0) 
\text { and } S ^- _w \backslash S ^+ _w = B (0 , r_0) \backslash B (- \bar w , r_0) $$
modulo sets of measure 0.
With these natural parameterizations, 
one has
\begin{align*}
\int _{z \in B (0 , r_0 )} (f \circ g - g \circ f) & \: | d z |^2  \\
= \int \int _0 ^{2 \pi} \int _0 ^1 & f ( \bx (- \frac{ r_0 e ^{i \varphi } \bar w }{| w |} - \lambda \bar w  , w ))
\\ \times &
g (\bx (- \frac{ r_0 e ^{i \varphi } \bar w }{| w |} - (1 - \lambda ) \bar w , - w )) 
r _0 |\bar w |\cos \varphi d \lambda d \varphi | d w |^2.
\end{align*}
Next, we approximate $f $ by its Taylor series at $\bs ^{-1} ([1 , 0]) $ as $r _0 \to \infty$.
More precisely, define the trivialization
$$\dot \bx (\dot z , \dot w) 
:= \left[ \frac{1}{(1 + |\dot z|^2 )^\frac{1}{2}} , \frac{\dot z}{(1 + |\dot z|^2 )^\frac{1}{2}} \right] ^{\dot w}_\rT.$$
Write $\dot z = \dot x + i \dot y , \dot w = \dot u + i \dot v, \; \; \dot x , \dot y , \dot u , \dot v , \in \bbR$.
Using the change in coordinate formula $\bx (z , w) = \dot \bx (\frac{1}{z} ,\frac{z ^2}{|z |^2} w)$
and the expansions
\begin{align*}
\frac{1}{- \frac{ r_0 e ^{i \varphi } \bar w }{| w |}- \lambda \bar w }
=& - \frac{|w|}{r_0 e ^{i \varphi } \bar w }
\left(1 - \frac{\lambda | \bar w |}{r_0 e ^{i \varphi }}
+ O (r _0 ^{-2}) \right) \\
\frac{ r_0 e ^{i \varphi } \bar w + \lambda \bar w |w|}{ r_0 e ^{- i \varphi } + \lambda |w| } 
=& e ^{2 i \varphi } \bar w + 2 i e ^{2 i \varphi } \sin \varphi \frac{\lambda \bar w |w|}{r _0 } + O (r ^{-2}),
\end{align*}
one gets
\begin{align*} 
f \Big( \bx \Big( - & \frac{ r_0 e ^{i \varphi } \bar w }{| w |} - \lambda \bar w  , w \Big) \Big) \\
=& f ([1 , 0] ^{ e ^{2 i \varphi } \bar w } _\rT) 
- \Big( \frac{ \re (e ^{- i \varphi } w )}{r_0 |w|} \partial _{\dot x}
+ \frac{ \im ( e ^{- i \varphi } w )}{r_0 |w|} \partial _{\dot y} \Big) f ([1 , 0] ^{ e ^{2 i \varphi } \bar w } _\rT)\\
&+ \frac{2 \lambda |\bar w| \sin \varphi }{r_0} 
\big( \im (e ^{2 i \varphi } \bar w) \partial _{\dot u} 
- \re (w ^{2 i \varphi } \bar w) \partial _{\dot v} \big) f([1 , 0] ^{ e ^{2 i \varphi } \bar w } _\rT)
+ O (r _0 ^{-2}).
\end{align*}
Combining with a similar expression for $g$, 
the integrand has an expansion:
\begin{align}
\label{ExpandInt}
\nonumber
f \Big( \bx \Big( - & \frac{ r_0 e ^{i \varphi } \bar w }{| w |} - \lambda \bar w  , w \Big) \Big) 
g \Big( \bx \Big( - \frac{ r_0 e ^{i \varphi } \bar w }{| w |} - (1 - \lambda ) \bar w , - w \Big) \Big) \\ \nonumber
= \: & f ([1 , 0] ^{ e ^{2 i \varphi } w } _\rT) g ([1 , 0] ^{ - e ^{2 i \varphi } w } _\rT) \\ 
&- \Big(\frac{ \re (e ^{i \varphi } w )}{r_0 |w|} \partial _{\dot x} 
+ \frac{ \im (e ^{i \varphi } w )}{r_0 |w|} \partial _{\dot y} \Big) 
\Big( f ([1 , 0] ^{ e ^{2 i \varphi } \bar w } _\rT) g([1 , 0] ^{ - e ^{2 i \varphi } \bar w } _\rT) \Big) \\ \nonumber
&+  f ([1 , 0] ^{ e ^{2 i \varphi } \bar  w } _\rT)\Big( \frac{2 (1 - \lambda ) |\bar w| \sin \varphi }{r_0} \Big)
\big( \im (e ^{2 i \varphi } \bar w) \partial _{\dot u} 
- \re (w ^{2 i \varphi } \bar w) \partial _{\dot v} \big) g([1 , 0] ^{ e ^{2 i \varphi } \bar w } _\rT) \\ \nonumber
&+  g ([1 , 0] ^{ - e ^{2 i \varphi } \bar w } _\rT) \Big( \frac{2 \lambda |\bar w| \sin \varphi }{r_0} \Big)
\big( \im (e ^{2 i \varphi } \bar w) \partial _{\dot u} 
- \re (w ^{2 i \varphi } \bar w) \partial _{\dot v} \big) f([1 , 0] ^{ e ^{2 i \varphi } \bar w } _\rT) \\ \nonumber
&+ O (r _0 ^{-2}).
\end{align}
We compute the (renormalized) integral of each terms in Equation (\ref{ExpandInt}).
First consider the last term:
\begin{align*}
\int _0 ^1 \int _0 ^{2 \pi} \int _\bbC &
g ([1 , 0] ^{ - e ^{2 i \varphi } \bar w } _\rT) \Big( \frac{2 \lambda |\bar w| \sin \varphi }{r_0} \Big)
\big( \im (e ^{2 i \varphi } w) \partial _{\dot u} 
- \re (e ^{2 i \varphi } w) \partial _{\dot v} \big) f([1 , 0] ^{ e ^{2 i \varphi } \bar w } _\rT) \\
\times & r_0 | w | \cos \varphi | d w |^2 d \varphi d \lambda \\
=& \int _0 ^1 \int _0 ^{2 \pi} \int _\bbC 
g ([1 , 0] ^{ - \bar w' } _\rT) \big( \im (e ^{2 i \varphi } \bar w) \partial _{\dot u} 
- \re (e ^{2 i \varphi } \bar w) \partial _{\dot v} \big) f([1 , 0] ^{ e ^{2 i \varphi } \bar w } _\rT) \\
& \times 2 | w' |^2 \lambda \sin \varphi \cos \varphi | d w' |^2 d \varphi d \lambda ,
\end{align*}
by changing variable $w' := e ^{ 2 i \varphi } w $.
The integral vanishes since $\int _0 ^{2 \pi} \cos \varphi \sin \varphi d \varphi = 0$.
Using the same arguments the integral of the first and the third term are both $0$.
It remains to consider the second term.
Again we change variable $w' := e ^{ 2 i \varphi } w $ to get:
\begin{align*}
- \int _0 ^1 \int _0 ^{2 \pi} \int _\bbC &
\Big( \frac{ \re ( e ^{- i \varphi } w ')}{|w'|} \partial _{\dot x} 
+ \frac{ \im (e ^{- i \varphi } w' )}{|w'|} \partial _{\dot y} \Big) 
\Big( f ([1 , 0] ^{ \bar w' } _\rT) g([1 , 0] ^{ - \bar w' } _\rT) \Big) \\
& \times | w' | \cos \varphi | d w' |^2 d \varphi d \lambda.
\end{align*}
Applying the identities 
$\int _0 ^{2 \pi} \cos \varphi \sin \varphi d \varphi = 0, \int _0 ^{2 \pi} \cos ^2 \varphi d \varphi = \pi$,
one finally obtains:
\begin{align*}
\Rint \big( & f \circ g (x) - g \circ f (x) \big) \mu _{\rM _0} (x) \\ 
=& - \pi \int _\bbC \big( \re ( w' ) \partial _{\dot x} + \im ( w') \partial _{\dot y} \big) 
\big( f ([1 , 0] ^{ \bar w' } _\rT) g([1 , 0] ^{ - \bar w' } _\rT) \big) | d w' |^2 . 
\end{align*}
\end{proof}

\subsection{ The McKean-Singer formula and index formula}
We recall the derivation of index formulas using the McKean-Singer formula.

Fix a Riemannian metric $g _\cA $ on $\cA $. 
Denote the invariant $\bs$-fiberwise Riemannian volume form by $\mu $.  
Let $\rE $ be a $\Cl (\cA ')$ module, and $(\eth + \varPsi )$ be a perturbed Dirac operator.
Consider
\begin{equation} 
\lim _{t \to \infty } \sideset{_\mathrm R}{} \Str (e ^{- t(\eth + \varPsi )^2})
 - \lim _{t \to 0^+} \sideset{_\mathrm R}{} \Str (e ^{- t (\eth + \varPsi)^2}) 
= \int _0 ^\infty \partial _t \sideset{_\mathrm R}{} \Str (e ^{- t (\eth + \varPsi)^2} ) d t.
\end{equation}
For the right hand side,
one has
$$ \partial _t \sideset{_\mathrm R}{} \Str (e ^{- t(\eth + \varPsi)^2} ) 
= \sideset{_\mathrm R}{} \Str (\partial _t e ^{- t(\eth + \varPsi)^2} )
= \sideset{_\mathrm R}{} \Str ([\eth + \varPsi , (\eth + \varPsi ) e ^{- t(\eth + \varPsi)^2}] ).$$
One can then use the trace defect formula in the last section to compute
$\sideset{_\mathrm R}{} \Str ([\eth + \varPsi , (\eth + \varPsi ) e ^{- t(\eth + \varPsi)^2}] )$.
The actual calculation is very complicated. 
Nevertheless we denote the result by
$$ \boldsymbol \eta (\eth + \varPsi)
:= \int _0 ^\infty \sideset{_\mathrm R}{} \Str ([\eth + \varPsi , (\eth + \varPsi ) e ^{- t(\eth + \varPsi)^2}] ) \: 
d t. $$

It remains to study the limits
$\lim _{t \to 0 ^+ } \sideset{_\mathrm R}{} \Str (e ^{- t(\eth + \varPsi )^2}) $
$ \lim _{t \to \infty } \sideset{_\mathrm R}{} \Str (e ^{- t (\eth + \varPsi)^2}) $.
Much work have already been done. We first consider the $t \to 0^+$-limit.
\begin{prop}
\label{tto0Prop}
For the $t \to 0$ limit, one has
\begin{equation}
\lim _{t \to 0 ^+ } \sideset{_\mathrm R}{} \Str (e ^{- t(\eth + \varPsi )^2}) 
= \sideset{_\mathrm R}{} \int \Omega _\varkappa ( \widehat \bbA \wedge \exp ( - F ^{\rE / \rS}) ) .
\end{equation}
\end{prop}
\begin{proof}
Recall that, by Lemma \ref{AsymLem}, one has the asymptotic expansion 
$$ e ^{- t (\eth + \varPsi )^2} (x, t) 
\cong (4 \pi )^{- \frac{n}{2}} \sum _{i = 0} ^\infty t ^{i - \frac{n}{2}} Q _i (x). $$
Since $\str Q _i = 0 $ for any $i < \frac{n}{2} = 1$, 
it follows that 
\begin{align}
\label{tto0}
\lim _{t \to 0^+} (4 \pi )^{- \frac{n}{2}} \sum _{i = 0} ^\infty t ^{i - \frac{n}{2}} \str (Q _i (x))
=& (4 \pi )^{-1} \str Q _{\frac{n}{2}} (x) \\ \nonumber
=& (4 \pi )^{-1} \frac{ \Omega _\varkappa ( \widehat \bbA \wedge \exp ( - F ^{\rE / \rS}) ) }{ \mu },
\end{align}
by (2) of Lemma \ref{AsymLem}.
Since $\rM$ is compact, 
the convergence in Equation (\ref{tto0}) is uniform in all derivatives.
Since Definition \ref{CP1Renorm} of the renormalized integral only involves integration and evaluation 
of the derivatives of the integrands,
it follows that
$$ \lim _{t \to 0 ^+ } \sideset{_\mathrm R}{} \int \str e ^{- t (\eth + \varPsi )^2} (x, t) \mu =
(4 \pi ) ^{-1} \sideset{_\mathrm R}{} \int \Omega _\varkappa ( \widehat \bbA \wedge \exp ( - F ^{\rE / \rS}) ) $$
as well.
\end{proof}

As a direct consequence of Proposition \ref{tto0Prop}, one has
\begin{thm}
For any perturbed Dirac operators $\eth + \varPsi $,
not necessary Fredholm,
one has
\begin{equation}
\label{CheatEq}
\lim _{t \to \infty } \sideset{_\mathrm R}{} \Str (e ^{- t(\eth + \varPsi )^2})
= (4 \pi )^{-1} \sideset{_\mathrm R}{} \int \Omega _\varkappa ( \widehat \bbA \wedge \exp ( - F ^{\rE / \rS}) )  
+ \boldsymbol \eta (\eth + \varPsi ),
\end{equation} 
provided the limits on both sides exist.
\end{thm}

We turn to study the behavior as $t \to \infty $.
Let $\eth + \varPsi $, be a perturbed Dirac operator.
Note that $\eth + \varPsi$ is essentially self-adjoint.
In addition, we assume that $\eth _{ {}_\rT e} + R _{{}_\rT e} $ is invertible.
It follows from Corollary \ref{NisLem} that $\nu (\eth + \varPsi ) $ is Fredholm.
Since one has $\cG _{\rM _0} \cong \rM _0 \times \rM _0$,
it follows that $0$ is (at most) an isolated point of $\boldsymbol \sigma (\eth _x + R _x ) $ for $x \neq {} _\rT e $.
Our last objective is to study the behavior of the renormalized integral
$$ \sideset{_\mathrm R}{} \int \str e ^{ - t ( \eth + \varPsi )^2 } \mu , $$ 
as $t \to \infty$.

From our assumptions, it is clear that the null space $\Ker (\nu ((\eth + \varPsi )^2))$, is finite dimensional.
Denote by $P ^0 $ the projections onto $\Ker (\nu ((\eth + \varPsi )^2))$.
Let $u _1 , \cdots , u _N \in \bL ^2 (\rM _0 , \rE) $ 
be any orthonormal basis of $\Ker (\nu ((\eth + \varPsi )^2))$.
Then $P ^0 _{x} $ has a kernel
$$ \sum _{i =1 } ^N u _i (y) u _i (y') , 
\quad (y , y') \in \rM _0 \times \rM _0 \cong \cG _{\rM _0}.$$
Consider the regularity of $ u _i $.
Applying the parametrix formula
$$ \nu _0 (Q _1) \nu _0 (\varPsi) - \id = \nu _0 (R _1) $$
to $u _i$, where $ Q _1 \in \Psi ^{[- m ]} _\mu (\cG , \rE) , R _1 \in \Psi ^ {- \infty} _\mu (\cG , \rE)$,
one has 
$$ u _i = \nu _0 (R _1) u _i ,$$
for each $i $.
Using Lemma \ref{SoboBdLem}, it follows that $u _i \in \bW ^\infty (\rM _0 , \rE)$.

By the identification $\nu ((\eth + \varPsi )^2) \cong (\eth + \varPsi )^2 _x , x \neq {}_ \rT e$,
$\Ker ((\eth + \varPsi )^2 _x )$ is finite dimensional and consists of elements in 
$\bW ^\infty (\cG _x , \bt ^{-1} \rE )$.
Denote the projection onto the kernel of $(\eth + \varPsi) ^2 _x $ by $P ^0 _x$
(note that $P ^0 _{ {} _\rT e } = 0 $ since $(\eth + \varPsi )^2 _x $ is invertible).
Then, using again the fact that $0$ is at most an isolated point of $\bsigma _{\bL ^2 }( (\eth + \varPsi )^2 _x ) $,
one has the following well known variation of \cite{Simon;SpecHeatKer}: 
\begin{lem}
\label{SimonLongTime}
There exists some $\lambda > 0$ such that for each $x \in \rM $,
\begin{equation}
\lim _{t \to \infty } e ^{t \lambda } (e ^{- t (\eth - \varPsi )^2 _x } - P ^0 _x ) = 0
\end{equation}
in all Sobolev norms. 
\end{lem}

Unfortunately, we do not know any direct way to prove that
$$ \sideset{_{\mathrm R}}{} \Str ( e ^ {- t (\eth + \varPsi )^2 }) \to \sideset{_\mathrm R}{} \Str ( P^0 ) $$
as $t \to \infty $.
Instead, we observe that $\nu ((\eth + \varPsi )^2 ) $ can be identified with an edge operator on $\rM _0 = \bbR ^2 $,
studied in \cite{Albin;EdgeInd}.
From Lemma \ref{AlbinConj}, 
the heat kernel $e ^ {- t (\eth + \varPsi )^2 }$ coincides with the heat calculus constructed in 
\cite[Section 4]{Albin;EdgeInd}.
Furthermore, it is easy to see that Definition \ref{CP1Renorm} coincides with \cite[Equation (6.1)]{Albin;EdgeInd},
for the heat kernel.
Therefore, by \cite[Lemma 6.1]{Albin;EdgeInd}, 
one has
\begin{equation}
\label{ttoinfty}
\lim _{t \to \infty } \sideset{_\mathrm R}{} \Str (e ^{- t(\eth + \varPsi )^2})
= \sideset{_\mathrm R}{} \Str (P ^0) = \ind (\nu (\eth + \varPsi ) ).
\end{equation}
Note that the last equality follows from the fact that $\str (P ^0 )$ is an integrable function on $\rM _0$,
hence the renormalized integral coincides with the usual integral, 
which turns out to be $\ind (\nu (\eth + \varPsi ) ) $ because $P ^0$ is just the projection to the null space of 
$\nu ((\eth + \varPsi )^2) $.

Finally, combining Equations (\ref{CheatEq}) and (\ref{ttoinfty}), and results in Section 3,
one gets
\begin{thm}
For any self adjoint perturbed Dirac operator 
$\eth + \varPsi \in \Psi ^1 _\mu (\cG , \rE )$ on the symplectic groupoid $\cG = \rT \backslash (\rSU (2) \times \rN)$
of the Bruhat sphere,
such that the Fourier-Laplace transform  
$ \fF ((\eth + \varPsi )_{{}_\rT e} ) $ is invertible on a tubular neighborhood of the real axis,
$\nu _0 (\eth + \varPsi ) : \bW ^1 (\rE ) \to \bW ^0 (\rE)$ is Fredholm;
and its Fredholm index is given by the Atiyah-Singer index formula:
$$ \ind (\nu _0 (\eth + \varPsi ))
= (4 \pi ) ^{-1} \sideset{_\mathrm R}{} \int \Omega _\varkappa ( \widehat \bbA \wedge \exp ( - F ^{\rE / \rS}) ) 
+ \boldsymbol \eta (\eth + \varPsi ).$$
\end{thm}
\vspace{11cm}
$ \; $

\pagebreak \thispagestyle{firstpage}
\section{Concluding remarks}
In this last section, we make some remarks and highlight some open problems.

Our first objective in generalizing the calculus on manifolds with boundary 
was to extend the uniformly supported pseudo-differential calculus to include the parametrix of Fredholm operators.
We did so for the Bruhat sphere case in Section 3,
where we used the exponentially decaying calculus.
In the general case, one would derive an invertibility criterion on the $\bs$-fibers over the invariant sub-manifolds.
That involves understanding the representation theory of the isotropy subgroup 
$\cG ^x _x $ on sections over the $\cG ^x _x $-principle bundle $\bs ^{-1} (x)$.
It is known that the kernel of inverse of an uniformly supported pseudo-differential operator 
on a manifold with bounded geometry has exponential decay \cite{Shubin;BdGeom}.  
The only remaining problem 
is whether one can use a tubular neighborhood theorem to extend the fiber-wise inverse to the whole groupoid.
In the same vein, 
Medadze and Shubin \cite{Shubin;Lie2} 
proved that the space of pseudo-differential operators on an unimodular Lie group with exponentially decaying kernel
is closed under functional calculus. 
It would be interesting to prove an analogue for Lie groupoids.
More precisely:
\begin{conj}
Let $\cG \rightrightarrows \rM $ be a groupoid with compact units $\rM$ and polynomial growth.
Then the exponentially decaying calculus
$$ \bigcup _{\varepsilon > 0 } \Psi ^{[\infty ]} _\varepsilon (\cG) $$
is closed under holomorphic functional calculus.
\end{conj}
The main difficulty in proving the conjecture
lies in proving that the inverses of a smooth family of pseudo-differential operators is still a smooth family.
Such a result would enable one to construct, say, 
complex powers of elliptic operators in a framework more concrete than the axiomatic approach of \cite{Nistor;CplxPwr}. 

The discussion on extended calculus cannot be complete without mentioning what is missing in our construction,
compared with the case of edge manifolds.
In the latter case,
one can construct a `very residual' calculus,
consisting of functions (sections) on $\rM _0 \times \rM _0 $ with poly-homogeneous expansions near the singularities.
The full calculus is formed by adding the residual calculus to the decaying calculus.
Then it was shown that the full calculus contains the generalized inverses of (semi)-Fredholm operators.
The proof of these results uses order-by-order cancellations of the boundary defining function near the singular leaves.
It is not clear what analogue should be used for groupoids.
However, the techniques used in \cite{Nistor;Polyhedral,Nistor;Polyhedral2},
and the occurrence of stereographic coordinates 
(which just measures the distance from the opposite of the singularity) in Section 5
might offer a strong hint. 

Our next task was to construct the heat kernel of perturbed Laplacian operators on a groupoid in Section 4.
The proof of existence is fairly classical. 
The mystery lies in the proof of transverse smoothness of the heat kernel, 
which requires considering a (rather arbitrary) 
transverse metric and bounding the derivatives of the multiplication operator.
At this point,
we conjecture that a transverse metric satisfying the hypothesis of 
Theorem \ref{EstReg} exists for all Hausdorff groupoids,
and can be constructed by gluing exponential coordinates (as in Nistor \cite{Nistor;IntAlg'oid}).

We went on to derive an Atiyah-Singer type index formula on the Bruhat sphere in Section 5.
We cheated by using the stereographic coordinates on the Bruhat sphere to define the renormalized integral.
Therefore the arguments cannot be easily generalized beyond the flag manifolds.
We further cheated by using known results from edge calculus to show that the renormalized trace converges to the
Fredholm index.
We expect a direct proof of Equation (\ref{ttoinfty}) would be possible by better understanding the resolvent and/or 
null space projection of the Laplacian operator,
that would involve results in functional calculus or residue calculus,
as described earlier. 
One immediate observation form the renormalized index theory is that the renormalized index,
as well as the $\mathbb K$-theoretic index, 
of an elliptic (pseudo)-differential operator are well defined even for non-Fredholm operators.
We have not studied the connection between the two,
but the arguments involved should be straightforward (see, for example, \cite[Proposition 3]{Nistor;Family}).

On the side of generalizing the renormalized trace,
we think one possible way to proceed is to use the $Q$-weighted trace machinery developed by Paycha et. al. 
(see \cite{Paycha;Renorm} for an introduction),
but that is more speculation than educated guess...

And the thesis ends here. 
However, 
the work in this thesis is just the beginning of a vast subject concerning singular pseudo-differential calculus
defined by groupoids.
In the limited space and time we had,
we were only able to achieve some success in the simplest case, namely the Bruhat sphere;
but the potential of the techniques illustrated here, is unlimited.

\pagebreak
\appendix

\thispagestyle{firstpage}
\section{Some preliminaries on differential geometry and pseudo- \\ differential calculus}

\subsection{ Notes on submersions and pullback vector bundles}
\label{DGNonsense}
In this section, we define some notations concerning pullback of vector bundles and recall some basic facts.
Let $\rB_1 , \rB_2 $ be manifolds, $ \pi : \rB _2 \to \rB _1 $ be a smooth map,
and $\rE $ be a vector bundle over $\rB _1$. 
Denote the bundle projection by $\wp : \rE \to \rB _1 $.

\begin{dfn}
The {\it pullback bundle} is the vector bundle over $\rB _2$:
$$ \pi ^{-1} \rE := \{ (x , e) \in \rB _2 \times \rE : \pi (x ) = \wp (e) \}, $$
with bundle projection $\pi ^{-1} \wp (x , e) := x $ and the fiber-wise linear operations. 
\end{dfn}
One has a natural map $\pi _\rE : \pi ^{-1} \rE \to \rE $ determined by the commutative diagram
$$
\begin{CD}
\pi ^{-1} \rE @> \pi _\rE >> \rE \\
@VVV @VVV \\
\rB _2 @> \pi >> \rB _1 
\end{CD}
\quad .
$$
Consider the particular case $\rE = T \rB _1$. 
One has $\pi _{T \rB _1 } : \pi ^{-1} T \rB _1 \to T \rB _1 $.
On other hand, one also has the differential $d \pi : T \rB _2 \to T \rB _1 $,
These two maps determine a bundle map $ \pi _* \in \Gamma ^\infty (\Hom ( T \rB _2 , \pi ^{-1} T \rB _1 ))$ by
$$ \pi _* ( X ) := (x , d \pi (X) ) , \quad \forall X \in T _x \rB _2 .$$
Also recall that one can ``pullback" a section to a section of the pullback bundle, i.e., 
one has the naturally defined map $\pi _\rE ^{-1} : \Gamma ^\infty (\rE ) \to \Gamma ^\infty (\pi ^{-1} \rE ) $, 
$$ (\pi _\rE ^{-1} f) (x) := f (\pi (x) ) , \quad \forall f \in \Gamma ^\infty (\rE ), x \in \rB _2 .$$ 

Given any connection $\nabla ^\rE $ on $\rE$,
recall that the pullback connection $\nabla ^{\pi ^{-1} \rE} $ is a connection on $\pi ^{-1} \rE$ characterized by
$$ (\nabla ^{\pi ^{-1} \rE} )_X (\pi ^{-1} f) (x) = ( x , \nabla ^\rE _{d \pi (X)} f (\pi (x)) ),$$ 
for any $x \in \rB _2 , X \in T _x \rB _2$.
It follows, by using the canonical identification 
$$ \Hom (\pi ^{-1} T \rB _1 , \pi ^{-1} \rE ) \cong \pi ^{-1} \Hom (T \rB _1 , \rE ), $$
that one can write
\begin{equation}
\label{PullbackDer}
\nabla ^{\pi ^{-1} \rE} (\pi _\rE ^{-1} f) = (\pi ^{-1} _{\Hom (T \rB _1 , \rE) } (\nabla ^\rE f )) \circ \pi _* ,
\end{equation}
for any section $f \in \Gamma ^\infty (\rE )$.
Moreover, applying covariant derivatives to Equation (\ref{PullbackDer}) and using the Leibniz rule, one gets
\begin{align*}
(\nabla ^{\pi ^{-1} \rE} )^2 (\pi _\rE ^{-1} f)
=& \nabla ^{\Hom (T \rB _1 \otimes \pi ^{-1} \rE ) } \nabla ^{\pi ^{-1} \rE} (\pi _\rE ^{-1} f) \\
=& \nabla ^{\Hom (T \rB _1 \otimes \pi ^{-1} \rE ) } 
((\pi ^{-1} _{\Hom (T \rB _1 , \rE) } (\nabla ^\rE f )) \circ \pi _* ) \\
=& ((\pi ^{-1} _{\Hom (T \rB _1 , \Hom T (T \rB _1 , \rE )) } (\nabla ^{\Hom (T B_1 , \rE )} \nabla ^\rE f )
\circ \pi _* ) \circ \pi _* \\
&+ (\pi ^{-1} _{\Hom (T \rB _1 , \rE) } (\nabla ^\rE f )) \circ (\nabla ^{\Hom (T B_2 , \pi ^{-1} T B_1 )} \pi _* ) \\
=& (\pi ^{-1} _{\Hom (T \rB _1 \otimes T \rB _1 , \rE ) } ( \nabla ^\rE )^2 f ) \circ (\pi _* \otimes \pi _* ) \\
&+ (\pi ^{-1} _{\Hom (T \rB _1 , \rE) } (\nabla ^\rE f )) \circ (\nabla ^{\Hom (T B_2 , \pi ^{-1} T B_1 )} \pi _* ),
\end{align*}
and so on for higher derivatives.

Suppose, furthermore, that one has a fiber bundle structure $\rZ \to \rB _2 \to \rB _1$.
Since $\pi $ is now a submersion, 
$ \cV : = \ker (d \pi ) \subseteq T \rB _2 $ defines a (regular) integrable foliation.
We shall assume that $\cV $ is orientable.
Hence all fiber $\pi ^{-1} (p) \cong \rZ $ are orientable.
Fix a complementary distribution $\cH$ to $\cV$.
For any (local) vector field $\tilde X \in \Gamma (T \rB _1 )$, 
denote the horizontal lift of $\tilde X $ by $\tilde X ^\cH $.

\begin{dfn}
Given any $\omega \in \Gamma ^\infty (\wedge ^k \cV ' )$,
the Lie differential (with respect to $\cH$ ) is the section 
$ \fL ^{\cH } \omega \in \Gamma ^\infty (\Hom (\cH , \wedge \cV ' ))$,
\begin{align}
\fL ^\cH \omega (X) (V_1 , V_2 , \cdots , V_k ) (p)
=& \: \fL _{\tilde X ^\cH } (\omega (V _1 , V _2 , \cdots , V _k )) (p) \\ \nonumber
&- \sum _{i = 1 } ^k \omega (V _1 , \cdots , [ \tilde X ^\cH , V _i ] , \cdots , V_k ) (p),
\end{align}
for any $X \in T _p \rB _2 $, where $\tilde X $ is any local extension of $d \pi (X) $.
\end{dfn} 

Let $\varkappa $ be the rank of $\cV$.
For any $\mu \in \Gamma ^\infty _c (\wedge ^\varkappa \cV ) $,
consider point-wise average $\langle \mu \rangle \in C ^\infty (B _1 ) $, defined by
$$ \langle \mu \rangle (p) := \int _{ x \in \pi ^{-1} (p) } \mu |_{\pi ^{-1} (p) } .$$

\begin{lem}
\label{DFiberInt}
For any vector $X \in T _p \rB _1 , p \in \rB _1 $, one has the formula
$$ \fL _X (\langle \mu \rangle )(p ) = \int _{x \in \pi ^{-1} (p) } \fL ^\cH \mu ( X ^\cH ) .$$
\end{lem}
\begin{proof}
First consider the trivial case $\rB _2 \cong \rU \times \rZ , \rU \subseteq \bbR ^n$ and 
$\cH$ be the distribution along $ \rU \times \{ z \}, z \in \rZ.$
By linearity, one may assume that $ X = \partial _j $.
Fix a volume form on $\rZ $ and denote by $\mu _0$ its pullback to $\rU \times \rZ $ by the projection map onto $\rZ$.
Then one can write $\mu = f (p , z ) \mu _0 $ for some $f \in C ^\infty _c (\rB _2)$.
Differentiating under the integral sign, one gets
$$ \fL _X \langle \mu \rangle = \int _{z \in \rZ } (\partial _j f (p, z)) \mu _0 (z).$$
It is clear that $\fL ^\cH \mu _0 = 0 $.
It follows that $\fL ^\cH \mu (\partial _j ) = (\partial _j f (p, z)) \mu _0 (z)$,
and the assertion follows.

Let $\cH '$ be any other complementary distribution. 
Then one has for any vector field $X$, $X ^{\cH '} = X ^{\cH } + X ^\cV $ for some vector field
$X ^\cV \in \Gamma ^\infty (\cV)$.
Using the definition, it is easy to check that
$$ \fL ^{\cH '} \mu (X ^{\cH '} ) - \fL ^\cH \mu (X ^\cH ) = \fL _{X ^\cV } \mu ,$$
where the right hand side is just the Lie derivative on the integrable foliation $\cV$.
Integrating fiber-wisely, 
one gets
$$ \int _{x \in \pi ^{-1} (p) } \fL ^{\cH '} \mu ( X ^{\cH '} ) 
= \int _{x \in \pi ^{-1} (p) } \fL ^\cH \mu ( X ^\cH ) + \int _{x \in \pi ^{-1} (p) } \fL _{X ^\cV} \mu .$$
The second term on the right hand side vanishes by Stoke's theorem. 
Therefore one still gets
$$ \int _{x \in \pi ^{-1} (p) } \fL ^{\cH '} \mu ( X ^{\cH '} ) = \fL _X \langle \mu \rangle .$$

Finally, the general case follows because the assertion is local and one can always restrict to local trivializations. 
\end{proof}

We shall briefly describe several obvious generalizations to Lemma \ref{DFiberInt}.
Fix a connection $\nabla ^{T \rB _1} $ on $\rB _1 $.
For any $\omega \in \Gamma ^\infty (\wedge ^k \cV )$, 
define $\fL ^{(n)} \omega \in \Gamma ^\infty (\Hom (\otimes ^n \cH , \wedge ^k \cV)) $ inductively by
\begin{align*}
\fL ^{(1)} \omega := & \: \fL ^\cH \omega \\
\fL ^{(m + 1)} \omega (X _0 , \cdots , X _m ) 
:= & \:
\fL ^\cH (\fL ^{(m)} \omega (\tilde X _1 ^\cH , \tilde X _2 ^\cH , \cdots , \tilde X _m ^\cH )) (X _0 ) (p) \\ \nonumber
&- \sum _{i = 1 } ^m (\fL ^{(m)} \omega )
(\tilde X _1 ^\cH  , \cdots , (\nabla ^{T \rB _1 } _{\tilde X _0 } \tilde X _i )^\cH , \cdots , \tilde X _m ^\cH ) (p),
\end{align*}
for any $X _0 , \cdots X _m \in \cH _p $, where $\tilde X _i $ is any local extension of $d \pi (X _i) $.
Then a straightforward computation using the Lemma \ref{DFiberInt} and the definitions gives
\begin{cor}
For any $\mu \in \Gamma ^\infty _c (\wedge ^\varkappa \cV )$,
$X _1 , \cdots , X _m \in T _p \rB _1 , p \in \rB _1$, one has
$$ \nabla ^m (\langle \mu \rangle )(X _1 , \cdots , X _m ) (p)
= \int _{x \in \pi ^{-1} (p)} \fL ^{(m)} \mu (\tilde X _1 ^\cH , \cdots , \tilde X _m ^\cH ) (x),$$
where $\tilde X _i $ is any local extension of $X _i $.
\end{cor} 

Lemma \ref{DFiberInt} can also be generalized in a different direction
Let $\rE $ be a vector bundle over $\rB _1 $. 
For any $f \in \Gamma ^\infty _c (\pi ^{-1} \rE) , \mu \in \Gamma ^\infty (\wedge ^\varkappa \cV)$, 
define
$$ \langle f \mu \rangle (p) 
:= \sum _{i = 1 } ^l \langle f _i \mu \rangle (p) e _i (p) \quad \in \Gamma ^\infty (\rE), \quad p \in \rB _1, $$
where $e _1 , \cdots , e _l $ is any local basis around $p$ and $f = \sum _{i = 1 }^l f _i \pi ^{-1} (e _i ) $ 
on $\pi ^{-1} (p)$.
The definition is independent of choice of a local basis.
Let $\nabla ^\rE $ be any fixed connection on $\rE $. 
Then a simple application of Lemma \ref{DFiberInt} leads to 
\begin{cor}
Given any $f \in \Gamma ^\infty _c (\pi ^{-1} \rE) , \mu \in \Gamma ^\infty (\wedge ^\varkappa \cV)$.
Then for any vector field $X \in \Gamma ^\infty (T \rB _1 ), p \in \rB _1$, 
$$ \nabla ^\rE \langle f \mu \rangle (X) (p)
= \int _{x \in \pi ^{-1} (p) } (\pi ^{-1} (\nabla ^ \rE ) f) (X ^\cH ) \mu (x) + f (\fL ^\cH \mu (X ^\cH)) (x) .$$
\end{cor}

\subsection{ Preliminaries on pseudo-differential calculus}
\label{PDONonSense}
In this section, we recall some basic definitions and results about pseudo-differential calculus.
All materials in this section are classical and can be found in, say, Hormander \cite{Hormander;1}.

\subsubsection{\bf Distributions and kernels}
Let $\Omega \subseteq \bbR$ be an open subset.
We denote by $C^\infty _c (\Omega )$ the space of smooth compactly supported functions on $\Omega $.
The space $C^\infty _c (\Omega ) $ is equipped with the $C^\infty $-topology:
$$ u _n \rightarrow u \; \text {if} \; 
\sup _{x \in \rK} | \partial _x ^I (u_n - u) | \rightarrow 0, $$
for any compact subset $\rK$ and any multi-index $I$.

A distribution (on $\Omega $) is a continuous linear map $\phi : C^\infty _c (\Omega ) \rightarrow \bbC$.
We shall denote the space of distributions by
$$ C^\infty _c (\Omega )'. $$
For any open subset $\rU \subset \Omega $, 
the restriction of $\phi $ to $\rU$ is defined to be the restriction of $\phi $ to $C^\infty _c (\rU)$
(extended to $C^\infty _c (\Omega )$ by 0).
The support of $\phi $, denoted $\Supp (\phi )$ ,
is the collection of points $x \in \Omega $ such that the restriction of 
$\phi $ to any open neighborhood of $x$ is non-zero. 
We say that $\phi \in C^\infty (\Omega )$ if there exist $\kappa \in C^\infty (\Omega )$ such that
$$ \phi (u) = \int _\Omega \kappa (x) u (x) \: d x, \quad \forall u \in C^\infty _c (\Omega ). $$
Note that such $\kappa $, if it exists, is unique.

The most important result about distributions is the Schwartz distribution theorem:
\begin{lem}
For any continuous map $A : C ^\infty _c (\rM ) \to C ^\infty _c (\rM )' $,
there exists a unique continuous linear functional $K : C ^\infty _c (\rM \times \rM) \to \bbC$
such that
$$ (A f ) (g) = K (f (x) g (y)), \quad \forall f , g \in C ^\infty _c (\rM). $$    
\end{lem}

\subsubsection{\bf Pseudo-differential operators on a manifold}
\begin{dfn}
Let $\Omega $ be an open subset on $\bbR^n$, and $m \in \bbR$. 
A symbol of order $\leq m$ is a smooth function 
$\sigma (x , \zeta )\in C^\infty (\Omega \times \bbR^n)$ 
such that for any compact $\rK \subset \Omega $ and  multi-index $I, J$,
there is a constant $C^\rK_{I,J}$ such that 
$$ 
\left|
\partial ^I _x \partial ^J _\zeta \sigma (x, \zeta )
\right|
\leq C^\rK_{I, J} (1 + |\zeta |^2)^\frac{m - |J|}{2} \quad \forall x \in \rK.$$
\end{dfn}

The set of symbols on $\Omega $ of order $\leq m$ shall be denoted by $\mathbf S^m (\Omega )$; 
and define 
$$\mathbf S ^{-\infty} (\Omega ) := \bigcap_{m \in \bbR} \mathbf S^m (\Omega ), 
\mathbf S^\infty (\Omega ) := \bigcup_{m \in \bbR} \mathbf S^m (\Omega ).$$

\begin{dfn}
A symbol $\sigma_l \in \mathbf S ^l (\Omega ) $ is called homogeneous of order $l$, if
$$ \sigma_l (x, \lambda \zeta ) = \lambda ^l \sigma (x, \zeta ), 
\quad \forall x \in \Omega , |\lambda | \geq 1, |\zeta | \geq 1 .$$
A symbol $\sigma \in \mathbf S ^m (\Omega )$ is said to be classical of order $m, m \in \bbZ$ 
if there are homogeneous symbols $\sigma _m , \sigma _{m-1} , \cdots$,
of orders $m, m-1, \cdots$ respectively, such that 
$$ \sigma - \sum_{l=0}^{N - 1} \sigma _{m - l} \in \mathbf S^{m - N} (\Omega )$$
for $N = 1, 2, \cdots$.
\end{dfn} 

The set of classical symbols of order $m \in \bbZ$ is denoted by $\mathbf S^{[m]} (\Omega )$. 

\begin{dfn}
Let $\rM$ be a manifold. 
A function $\sigma \in C^\infty (T^* \rM)$ is called a symbol of order $\leq m$ 
if for every coordinate patch $(\rU, \bx)$, 
$$ \sigma \circ (\bx^*) \in \mathbf S ^m (\bx (\rU)). $$  
Here, we have identified $T^* (\bx (\rU)) \cong \bx (\rU) \times \bbR^n$.
The symbol $\sigma $ is said to be homogeneous (resp. classical)
if $ \sigma \circ (\bx^*) $ is homogeneous (resp. classical).
\end{dfn}

The set of symbols of order $\leq m$ (resp. classical symbols of order $m$) is denoted by
$\mathbf S^m (\rM)$ (resp. $\mathbf S ^{[m]} (\rM)$). 

\begin{dfn}
A pseudo-differential operator on $\Omega \subseteq \bbR^n$ of order 
$\leq m$ is a linear operator $\varPsi : C^\infty _c (\rU) \rightarrow C^\infty (\rU)$ 
of the form
$$ (\varPsi u)(x) 
= (2 \pi)^{-n} \int_{\zeta \in \bbR^n} \int_{y \in \Omega} 
\sigma (x, \zeta ) e^{i \langle \zeta , x-y \rangle} u(y) \; d y \; d \zeta,
\quad u \in C^\infty_c (\Omega ),$$
for some symbol $\sigma \in \mathbf S ^m (\Omega )$.
If $\sigma $ is classical, i.e., $\sigma \in \mathbf S ^{[m]} (\Omega ), m \in \bbZ$, 
then we say that $\varPsi $ is a classical pseudo-differential operator of order $m$. 
\end{dfn}

\begin{dfn}
A pseudo-differential operator on a manifold $\rM$ of order $\leq m$ is a linear operator
$\varPsi : C^\infty _c (\rM) \rightarrow C^\infty (\rM)$ such that for any coordinate patch
$(\rU , \mathbf x)$, the induced map 
$$u \mapsto (\mathbf x^{-1} )^* (\varPsi (\mathbf x^* u)), \quad u \in C^\infty_c (\bx(\rU))$$
is a pseudo-differential operator on $\bx(\rU) \subseteq \bbR^n$ of order $\leq m$.
\end{dfn}

The set of pseudo-differential operators on $\rM$, of order $\leq m$ 
(resp. classical pseudo-differential operators of order $m$), 
is denoted by $\Psi^m (\rM)$ (resp. $\Psi ^{[m]} (\rM)$). 
We also define 
$$\Psi ^{-\infty} (\rM ) := \bigcap_{m \in \bbR} \Psi ^m (\rM ), 
\Psi ^\infty (\rM ) := \bigcup_{m \in \bbR} \Psi ^m (\rM ).$$
Note that $\Psi ^{- \infty} (\rM) = \bigcap _{m \in \bbZ} \Psi ^{[m]} (\rM )$.

\begin{dfn}
Let $\varPsi \in \Psi ^\infty (\rM)$ be a pseudo-differential operator with distributional kernel
$\kappa (x, y)$. 
The support of $\varPsi $, denoted $\Supp \varPsi $, is defined to be the support of $\kappa $. 
The operator $\varPsi $ is said to be properly supported if for any compact subset 
$\rK \subset \rM$, the set
$$ (\rK \times \rM) \bigcap \Supp (\varPsi) $$
is a compact subset of $\rM \times \rM$. 
\end{dfn}

We denote the space of properly supported pseudo-differential operators of order $\leq m $ by
$\Psi ^m_\varrho (\rM)$.
It is clear that a properly supported 
$\varPsi \in \Psi^\infty (\rM)$ extends uniquely to a linear operator from 
$C^\infty (\rM) $ to itself. 
It follows that the composition of two pseudo-differential operators $\varPsi \circ \varPhi $ 
is well defined whenever one of them is properly supported.

\subsubsection{\bf The symbol of a pseudo-differential operator}
Fix a connection $\nabla$ on $\rM$. 
Then there is a neighborhood of the zero section $\Omega \subset T \rM$
such that the exponential map 
$\exp _\nabla : \Omega \rightarrow \rM \times \rM$ is a diffeomorphism onto its image.
Fix a smooth function $\chi (x, y)$ supported on the image of $\exp _\nabla$ 
and equal to 1 on a smaller neighborhood of the zero section.
Define $\Theta (x, y) := \chi (x, y) \exp ^{-1}_\nabla (x, y)$.

\begin{dfn}
\label{TotalSym}
Given a $\varPsi \in \Psi ^m (\rM), m \in \bbR$. 
Define $\sigma (\varPsi ) \in S^m (\rM)$ by
$$ \sigma (\varPsi) (\zeta ) := \varPsi (e^{i \langle \zeta , \Theta (x, \cdot) \rangle } \chi (x, \cdot))(x),
\quad \zeta \in T^*_x \rM.$$
The function $\sigma (\varPsi) $ is called the total symbol of $\varPsi $ with respect to 
$(\nabla, \chi )$.

If the total symbol $\sigma (\varPsi )$ is classical, 
i.e., there exists homogeneous symbols $\sigma _m , \\ \sigma _{m-1} , \cdots$,
of orders $m, m-1, \cdots$ respectively, such that 
$$ \sigma - \sum_{l=0}^{N - 1} \sigma _{m - l} \in \mathbf S^{m - N} (\rM )$$
for $N = 1, 2, \cdots$,
then we say that $\varPsi $ is a classical pseudo-differential operator on $\rM$. 
In this case, we define the principal symbol of $\varPsi $ as
$$ \sigma _{\mathrm {top}} (\varPsi) := \sigma _{m}.$$
\end{dfn}

We denote the space of classical pseudo-differential operators on $\rM$ by 
$\Psi ^{[m]} (\rM)$.

\begin{rem}
\label{InvPrincipal}
It can be shown that if the total symbol with respect to some $(\nabla, \chi )$ is classical,
then the total symbol with respect to any set of $(\nabla' , \chi ')$ is classical.
Also, it is well known that the principal symbol is independent of $\nabla$ and $\chi $.
\end{rem}

The following lemma asserts that a pseudo-differential operator 
$\varPsi $ can be recovered from its total symbol, up to a smoothing operator.
\begin{lem}
\label{Kennedy}
\cite[Proposition 3.1]{Kennedy;Intrin}
Any pseudo-differential operator $\varPsi $ on $\rM$ can be written in the form
\begin{equation}
\varPsi u (x) =
\int _{\zeta \in T^*_x \rM} \int _{y \in \rM}
\sigma (\zeta ) e^{-i \langle \zeta , \Theta (x, y) \rangle} \chi (x, y) u(y) d y d \zeta 
+ \int_{y \in \rM} \kappa (x, y) u (y) d y,
\end{equation}
for some $\kappa (x, y) \in C^\infty (\rM \times \rM)$.  
\end{lem} 

\subsubsection{\bf Pseudo-differential operators between sections of vector bundles}
It is straightforward to generalize the notion of pseudo-differential operators to sections of a vector bundle:
Let $\rE \rightarrow \rM$ be a vector bundle of rank $k$.
Let $(\rU, \bx)$ be a trivial coordinate patch.
Then any smooth section $s \in \Gamma ^\infty (\rE |_\rU)$ can be regarded as a 
$\bbC^k $-valued smooth function on $\bx (\rU)$.
We say that a linear map $\varPsi : \Gamma ^\infty _c (\rE) \rightarrow \Gamma ^\infty (\rE)$, 
is a pseudo-differential operator if for any pair of standard basis vectors of $\bbC ^k$,
$\mathbf e_i $ and $ \mathbf e_j, i, j = 1, \cdots, k$, the induced map
$$ u \mapsto \langle \mathbf e_i , (\mathbf x^{-1} )^* (\varPsi (\mathbf x^* u \mathbf e_j) \rangle, 
\quad u \in C^\infty_c (\bx(\rU)),$$
is a pseudo-differential operator on $\bx(\rU) \subseteq \bbR^n$.

We denote the set of pseudo-differential operator, of order $\leq m$, on $\rE \rightarrow \rM $ by 
$ \Psi ^m (\rM , \rE), $
and so on.

It is clear that the notion of (total and principle) symbol of an element in $\Psi (\rM , \rE)$
can be generalized in a similar manner.
However, in this case, the symbol is an element in 
$$ \Gamma ^\infty (\wp ^{-1} (\rE \otimes \rE')),$$
where $\wp : T^* \rM \rightarrow \rM $ is the natural projection.
Likewise, an operator $\varPsi \in \Psi ^m (\rM , \rE)$ is said to be {\it elliptic } if 
its principal symbol $\sigma (\zeta )$ is invertible (as a matrix) whenever $\zeta \neq 0$.

Finally, note that a smoothing operator on $\Gamma ^\infty _c (\rE)$ is of the from
$$ u \mapsto \int _{y \in \rM} \kappa (x, y) u (y ) d y, $$
where $\kappa (x, y) \in \Gamma ^\infty (\rE _x \otimes \rE'_y )$,
and the integrand is considered as a map from $\rM$ to $\rE _x$, for each $x \in \rM$. 

\subsection{ Manifolds with bounded geometry}
\label{BdGeomNonSense}
In this section, 
we a study special class of manifolds, 
namely, manifolds of bounded geometry in the sense of Shubin \cite{Shubin;BdGeom}.
Our objective is to define various Sobolev spaces, 
which would serve as the natural domain for the pseudo-differential operators.
We shall refer the general theory to \cite{Nistor;GeomOp}.

\begin{dfn}
A Riemannian manifold $\rM$ is said to have bounded geometry if
\begin{enumerate}
\item
$ \rM $ has positive injectivity radius;
\item
The Riemannian curvature $R$ of $\rM $ has bounded covariant derivatives.
\end{enumerate}
\end{dfn}

\subsubsection{\bf Basic properties} 
Here, we recall some basic results concerning manifolds of bounded geometry.
\begin{lem}
\label{BallCover}
\cite[Lemma 1.2]{Shubin;BdGeom}
There exists $\epsilon _0 > 0$ such that for any $0 < \varepsilon < \varepsilon _0 $,
there is a countable set $\{ x_\alpha \} \subset \rM$ such that the balls 
$B (x_\alpha , \varepsilon )$ is a cover of $\rM$,
and any $x \in \rM$ belongs to at most $N$ balls $B (x_\alpha , 2 \varepsilon ) $, for some $N$ independent of $x$.
\end{lem}

Recall that for every point $x$ in a Riemannian manifold $\rM$, 
the exponential map is a homeomorphism from an open neighborhood of $0 \in T_x \rM$ to an open neighborhood of $x$.
Its inverse thus defines a local coordinate patch, 
known as the {\it (geodesic) normal coordinates (around $x$)}.

\begin{lem}
\label{BddPartition}
Let $\{ (B ( x _\alpha , \varepsilon ) , \bx _\alpha ) \} $ be a cover by normal coordinates patches,
such that the conclusion of Lemma \ref{BallCover} holds.
Then there exists a partition on unity $\theta _\alpha $ subordinated to $\{ B (x_\alpha , \varepsilon ) \}$,
such that for any $k \in \bbN$, 
all $k$-th order partial derivatives of $\theta _\alpha $ are bounded by some $C_k$, independent of $\alpha $.
\end{lem}

\begin{dfn}
Let $\rM$ be a manifold with bounded geometry.
A vector bundle $\rE \rightarrow \rM$ is said to have bounded geometry if for any $k \in \bbN$,
there exist $C_k > 0$ such that
for any trivial normal coordinate patches,
the all $k$-th order partial derivatives of the transition function is bounded by $C_k$.
\end{dfn}

\subsubsection{\bf Sobolev spaces}
\begin{dfn}
\label{SoboDfn}
Let $\rE$ be a vector bundle of bounded geometry. 
Fix a normal coordinates cover $\{ (\rU_\alpha , \bx _\alpha ) \}$ of $\rM$ such that $\rE | _{\rU _\alpha }$ is trivial,
and a locally finite partition of unity $\{ \theta _\alpha \} $ subordinated to $\{ \rU _\alpha \}$, 
as in Lemma \ref{BddPartition}.
Regard $\theta _\alpha s$ as a smooth vector valued function on $\bbR ^n$ through local coordinates.

On $\Gamma ^\infty (\rE)$, define the $\infty $-norms 
\begin{equation}
\| s \| _{\infty , l} := \sup _\alpha \{ |\partial ^I \theta _ \alpha s (x) | : x \in U _\alpha , |I| \leq l \}
\end{equation}
for each $l \in \bbN$. 
We say that a section $s \in \Gamma ^\infty (\rE) $ has bounded derivatives if 
$\| s \| _{\infty , l} < \infty $.

For each $m \in \bbR$, define the 2-norms
\begin{equation}
\label{L-Infty}
\| s \| _{2, m} := \Big( \sum _{\alpha } 
\| \theta _\alpha s \| _{\bW ^m (U _\alpha ) } ^2 \Big)^{\frac{1}{2}},
\end{equation}
where $ \bW ^m (U _\alpha ) $ is the $m$-th Sobolev norm on $U _\alpha \subset \bbR ^n$.
We denote the completion of $\Gamma ^\infty _c (\rE) $ with respect to 
$\| \cdot \| _{2 , m}$ by $\bW ^{m} (\rM , \rE)$.
\end{dfn}
Observe that, since all transition functions are uniformly bounded, 
the equivalence classes of these norms are independent of the choices made. 

\begin{rem}
For $m \in \bbZ$, $\bW ^m (\rM)$ can be equivalently defined by the collection of distribution 
$ u \in C _c ^\infty (\rM )' $
such that $ \fL _{X _1} \fL _{X _2} \cdots \fL _{X _m} u \in \bL ^2 (\rM) $
for any collection of vector fields $ X _1 , \cdots , X _m $ with unit length.
\end{rem} 

As in the case of $\bbR ^n$, one has the Sobolev embedding
\begin{lem}
\label{SoboEm}
For any integer $m, l$ such that $m > l + \frac{n}{2}$,
$$ \bW ^m (\rM ) \subseteq C ^l _b (\rM).$$
Furthermore, there exists a constant $C$, depending only on $m, l, n$,
such that 
$$\| u \| _{0, l} \leq C \| u \| _{2 , m} $$
for any $u \in \bW ^m (\rM )$.
\end{lem}

\begin{cor}
\label{SoboDecay}
Let $u \in \bW ^m (\rM)$, where $m > l + \frac{n}{2}$ for some integer $l$. 
Fix any point $x _0 \in \rM$.
For any $\varepsilon > 0$,
there exist integer $N _0 $ such that for any integer $N > N _0$,
$$\sup _{x \not \in B (x_0 , N)} | u (x) |_l \leq \varepsilon .$$
\end{cor}
\begin{proof}
Fix smooth functions $\chi _j , j \in \bbN$ such that $0 \leq \chi _j \leq 1 $,
$\chi _j = 0$ on $B (x _0 , j)$, and $\chi _j = 1 $ on $\rM \backslash B (x _0 , j+1)$.
Since $\chi _j \to 0$ as $j \to \infty$, 
it follows that $\| \chi _j u \| _{2, m} \to 0 $.
By the previous Lemma, one has 
$$\sup _{x \not \in B (x_0 , j)} |\chi _j (x) u (x) |_l = \| \chi _j u \| _{0, l} \leq C \| \chi _j u \| _{2 , m} .$$
The assertion follows because 
$$ \sup _{x \not \in B (x_0 , j + 1)} | u (x) |_l 
= \sup _{x \not \in B (x_0 , j + 1)} |\chi _j (x) u (x) |_l 
\leq \sup _{x \not \in B (x_0 , j)} |\chi _j (x) u (x) |_l ,$$
for all integer $j$.
\end{proof}

On a manifold with bounded geometry, a class of `uniformly bounded' pseudo-differential operators can also be defined.
Fix any covering $\{ U _\alpha , \bx _\alpha \}$ of $\rM$ by normal coordinates.
Let $\varPsi \in \psi ^m _\varrho (\rM)$.
Recall that $( \bx _\alpha ^{-1} )^* \psi \bx _\alpha ^* $ is a pseudo-differential operator on $U _\alpha $.
Let $\sigma _\alpha \in \mathbf S ^m (U _\alpha )$ be the total symbol of $( \bx _\alpha ^{-1} )^* \psi \bx _\alpha ^* $.
Then we say that 
\begin{dfn}
The pseudo-differential operator $\varPsi $ is {\it uniformly bounded} if 
\begin{enumerate}
\item
The support of $\varPsi $ is contained in the set
$$ \{ (x , y ) \in \rM \times \rM : d (x , y ) < r \} $$
for some $r > 0$;
\item
For any multi-indexes $I , J$,  
there exists a constant $C _{I J} $, independent of $\alpha $,
such that 
$$ | \partial _x ^I \partial _\zeta ^J \sigma _\alpha | \leq C _{I J} (1 + |\zeta |)^{m - |J|} .$$ 
\end{enumerate}
We denote the set of all, 
uniformly bounded pseudo-differential operators of order $\leq m$ by $\Psi ^m _b (\rM)$.
\end{dfn}

Finally, we can state the main result on boundedness of pseudo-differential operators on Sobolev spaces.
\begin{lem}
\label{SoboBdLem}
For any $\varPsi \in \Psi ^m _b (\rM , \rE), u \in \bW ^l (\rM , \rE)$,
$ \varPsi u \in \bW ^{l - m} (\rM , \rE) $.
Furthermore, the map $u \mapsto \varPsi (u) $ is a bounded map from $\bW ^l (\rM , \rE) $ to $\bW ^{l - m} (\rM , \rE)$.
\end{lem}


\end{document}